\documentclass[11pt, leqno]{amsart}
\usepackage{setspace}
\usepackage{amssymb}
\usepackage{euscript} 
\usepackage{cite}

\usepackage[all]{xy}

\usepackage{tabularx}

\theoremstyle{plain}
\usepackage{amsfonts}
\usepackage{operators}
\usepackage{graphicx}
\usepackage{amsmath}
\usepackage{amssymb}

\numberwithin{equation}{section}

\newtheorem{thm}{Theorem}[section]
\newtheorem{Nott}[thm]{Notation}
\newtheorem{Cor}[thm]{Corollary}
\newtheorem{Prop}[thm]{Proposition}
\newtheorem{Def}[thm]{Definition}
\newtheorem{Lm}[thm]{Lemma}
\newtheorem{Rem}[thm]{Remark}

\def\abs#1{\left|#1\right|}
\newcommand{\build}{\mathfrak{B}_\infty}
\newcommand{\ratbuild}{\mathfrak{B}_1}
\newcommand{\maap}{A_M}

\newcommand{\StGitF}{\mathfrak{o}_F \oplus \mathfrak{o}_F}
\newcommand{\tStGitF}{\mathfrak{o}_F \oplus t\mathfrak{o}_F}
\newcommand\inv{\mathop{ \rm inv}}
\newcommand\dett{\mathop{ \rm det}}
\newcommand\Specc{\mathop{ \rm Spec}}
\newcommand\Stabb{\mathop{ \rm Stab}}

\newcommand\dimm{\mathop{ \rm dim}}
\newcommand\Indd{\mathop{ \rm Ind}}
\newcommand\minn{\mathop{ \rm min}}
\newcommand\cIndd{\mathop{ \rm c-Ind}}
\newcommand\inff{\mathop{ \rm inf}}
\newcommand\Stt{\mathop{ \rm St}}
\newcommand\modd{\mathop{ \rm mod}}
\newcommand\redd{\mathop{ \rm red}}
\newcommand\cool{\overline{\mathbb{Q}_l}}
\newcommand\Gall{\mathop{ \rm Gal}(\bar{k}/k)}
\newcommand\Homm{\mathop{ \rm Hom}}

\newcommand\Nrd{\mathop{ \rm Nrd}}
\newcommand\Kerr{\mathop{ \rm Ker}}

\begin{document} 


\title[The cohomology of affine Deligne-Lusztig Varieties]{The cohomology of the affine Deligne-Lusztig varieties in the affine flag manifold of $GL_2$}
\author{Alexander Ivanov} 

\maketitle

\section{Introduction}

Let $k$ be a field with $q$ elements, and let $\bar{k}$ be an algebraic closure. Let further $F = k((t))$, $L = \bar{k}((t))$ and $\sigma$ the Frobenius morphism 
of $L/F$. Thus $\sigma(\sum a_n t^n) = \sum \sigma(a_n)t^n$.  Denote the valuation on $L$ by $v_L$ and the ring of integers 
in $F$ by $\mathfrak{o}_F = k[[t]]$. Let $G$ be a split connected reductive group over $k$, 
and let $T$ be a split maximal torus in $G$. Let $\tilde{W}$ be the corresponding affine Weyl group, and let $I \subset G(L)$ be 
an Iwahori subgroup corresponding to an alcove in the apartment of the Bruhat-Tits building of G, associated to $T$. Then a Bruhat decomposition of 
$G(L)$ is given by $G(L) = \bigcup_{w \in \tilde{W}} IwI$.

In \cite{Ra}, for every $b \in G(L)$ and $w \in \tilde{W}$, the affine Deligne-Lusztig variety inside the affine flag manifold $G(L)/I$ of $G$ 
is defined by $$X_w(b) = \{ xI \in G(L)/I \colon x^{-1}b\sigma(x) \in IwI\}.$$ 
It is a locally closed subset of $G(L)/I$, and one provides it with the reduced induced sub-Ind-scheme structure. In fact, it is 
a scheme locally of finite type, defined over $k$. Up to isomorphism, it depends only on the $\sigma$-conjugacy class 
$\{g^{-1}b \sigma(g) \colon g \in G(L)\}$ of $b$:
if $b^{\prime} = g^{-1}b\sigma(g)$, then $x \mapsto g^{-1}x$ is an isomorphism $X_w(b) \rightarrow X_w(b^{\prime})$.
The subgroup $J_b = \{g \in G(L) \colon g^{-1} b \sigma(b) = b\}$ of $G(L)$ acts by left multiplication on $X_w(b)$. If $b, c \in G(L)$ are 
$\sigma$-conjugate, then the groups $J_b$ and $J_c$ are conjugate. The action of $J_b$ on $X_w(b)$ induces an action of $J_b$ 
on the $l$-adic cohomology of $X_w(b)$ with compact supports. 

The aim of this work is to compute the induced representations of $J_b$ on these cohomology groups for $G = GL_2$. Consider 
the diagonal torus $T \subset GL_2$,  the standard Iwahori subgroup $I \subset GL_2(L)$ and the natural embedding $SL_2 \subseteq GL_2$. Let $W_a$ 
denote the affine Weyl group of $SL_2(L)$ corresponding to the torus $T^{SL_2} = T \cap SL_2$. The extended affine Weyl group $\tilde{W}$ of $GL_2(L)$ is 
given by a split extension $\tilde{W} = W_a \rtimes \mathbb{Z}$. We choose a splitting by sending $v \in \mathbb{Z}$ 
to $\begin{pmatrix} 0 & 1 \\ t & 0\end{pmatrix}^v$. 

Now fix a $b \in GL_2(L)$ and $w \in W_a$. Since the valuation of the determinant of all matrices in $I$ is zero, $X_{(w,v)}(b) \neq \emptyset$ implies 
$v = v_L (\dett (b))$. We write $X_w(b) := X_{(w,v_L \circ \dett (b))}(b)$. As already mentioned, $X_w(b)$ depends up to isomorphy only 
on the $\sigma$-conjugacy class of $b$. Further if $c$ is in the center of $GL_2(L)$,  then we have $X_w(b) = X_w(cb)$ for all $w \in W_a$ 
(Lemma \ref{lemma3}). All in all, we have only three essentially different cases for $b$  in which we determine $X_w(b)$ explicitly 
together with the action of $J_b$. The three cases are presented in Table \ref{table:adlv}.
In each of these cases, for all $w \in W_a$ such that $X_w(b) \neq \emptyset$ (except for some special values of $w$, see below), we have 
the $J_b$-equivariant isomorphisms:

\begin{equation}\label{imtro1} X_w(b) \cong \coprod_{J_b/K_b^{(m)}} S, \end{equation}
where $S$ is a smooth connected variety whose dimension depends on $\ell(w)$ and $b$, and $K^{(m)}_b$ is a subgroup of $J_b$, 
where $m \in \{0,1\}$ depends on $w$. The group $J_b$ acts by permuting the connected components and $K^{(m)}_b$ is the stabilizer of one of them.
All these objects for the three different cases are collected in Table \ref{table:adlv}.

We set $b_1 := \begin{pmatrix} 0 & 1 \\ t & 0 \end{pmatrix}$. For $m \in \{0,1\}$ we write $g_m := \begin{pmatrix} 1 & 0 \\ 0 & t^{m} \end{pmatrix}$. \
Further, let $D$ denote the quaternion division algebra over $F$, containing the unramified extension of $F$ of degree two, and $U_D$ the unit subgroup of its valuation ring. 

\begin{table}[ht]
\caption{Affine Deligne-Lusztig Varieties for $GL_2$}
\centering
\begin{tabular}{|c|c|c|c|c|}
\hline
$b$ & $X_w(b) \neq \emptyset$ if & $J_b$ & $K^{(m)}_b$ & $S$ \\ [0.5ex]
\hline \hline

$1$ & $\ell(w) > 0$ odd & $GL_2(F)$ & $g_m GL_2(\mathfrak{o}_F) g_m^{-1}$ & $\mathbb{A}^{\frac{\ell(w)-1}{2}} \times (\mathbb{P}^1 - \mathbb{P}^1(k))$ \\ 
\hline

$\begin{pmatrix} 1 & 0 \\ 0 & t^{\alpha} \end{pmatrix}, \alpha > 0$ & $\ell(w) - \alpha > 0$ odd & $T(F)$ & $T(\mathfrak{o}_F)$  & $\mathbb{A}^{\frac{\ell(w)-\alpha-1}{2}} \times (\mathbb{P}^1 - \{0,\infty\})$ \\ 
\hline

$b_1$ & $\ell(w)$ even & $D^{\times}$ & $U_D$ & $\mathbb{A}^{\frac{\ell(w)}{2}}$ \\
\hline

\end{tabular}
\label{table:adlv}
\end{table}

\noindent The special cases excluded in this table are:
\begin{itemize}
\item[(i)] $b = 1$, $w=1$; 
\item[(ii)] $b = \begin{pmatrix} t^\alpha & 0 \\ 0 & 1 \end{pmatrix}, \alpha > 0$, $w$ such that $\ell(w) = \alpha$.
\end{itemize}
In all these special cases the variety $X_w(b)$ is a disjoint union of points. 
The precise statements are Propositions \ref{Propb1Var}, \ref{propbdvar}, and \ref{PropbnVar}.

Now we turn to the cohomology. Therefore, we make a base change to $\bar{k}$ and denote $X_w(b) \times_{\Specc k} \Specc \bar{k}$ by $\overline{X}_w(b)$. 
The schemes $\overline{X}_w(b)$ are locally of finite type, but not of finite type. In our case the $\overline{X}_w(b)$ are disjoint unions 
of schemes of finite type. Therefore, we take the cohomology with compact supports which commutes with colimits. Thus the cohomology groups 
of $\overline{X}_w(b)$ are direct sums of the cohomology groups of the connected components. The groups $J_b$ and $\Gamma := \Gall$ 
act on $H^{\ast}_c(\overline{X}_w(b),\cool)$ in a natural way. The group $J_b$ is in every case locally profinite and 
the representations $H^{\ast}_c(\overline{X}_w(b), \cool)$ are smooth. If we consider in the following a 
representation of a locally profinite group, then we mean a smooth representation. 

Consider first the case $b = 1$ and let $1 \neq w \in W_a$ such that $X_w(1) \neq \emptyset$. Put $G := J_1 = GL_2(F)$ and $K := K^{(0)}_1 = GL_2(\mathfrak{o}_F)$. 
Assume for simplicity that  $m \in \{0,1\}$ corresponding to this $w$ equals $0$ (compare (\ref{imtro1})). Consider the representation 
$$\overline{\Stt} = \inff\nolimits_{GL_2(k)}^K \Stt\nolimits_{GL_2(k)}$$
where $\Stt_{GL_2(k)}$ denotes the Steinberg representation of $GL_2(k)$ and we inflate with respect to the natural projection 
$K = GL_2(\mathfrak{o}_F) \twoheadrightarrow GL_2(k)$. We  prove (Proposition \ref{hehe}) that there are the following isomorphisms of $G \times \Gamma$-modules:

$$H^r_c(\overline{X}_w(1), \cool) \cong 
\begin{cases} \cIndd\nolimits^G_K \overline{\Stt}(\frac{\ell(w)-1}{2}) & \text{ if } r = \ell(w), \\ 
\cIndd\nolimits^G_K 1_{\cool}(\frac{\ell(w)-3}{2}) & \text{ if } r = \ell(w) + 1,\\ 
0 & \text{ else.}
\end{cases}$$

\noindent The idea of the proof is to reduce the cohomology of $X_w(1)$ to the cohomology of the Drinfeld upper halfplane 
$\Omega^2_k = \mathbb{P}^1 - \mathbb{P}^1(k)$ and the action of $K^{(m)}_1$ to the action of the finite group $GL_2(k)$ on $\Omega^2_k$.

To investigate these representations, we consider the $\Gamma$-modules $\Homm_G (H^{\ast}_c(\overline{X}_w(1), \cool), \pi)$ 
where $\pi$ ranges over smooth irreducible representations of $G$. As shown in \cite{BH}, there are two types of such representations: 
the cuspidal and the noncuspidal ones. First of all, there are no nonzero morphisms into cuspidal representations. The
noncuspidal (irreducible) representations are of one of the following types: $\Indd^G_B \chi$, where $\chi$ ranges over all character of the diagonal torus 
(with some nonrelevant exceptions) and $B$ is the Borel subgroup of upper triangular matrices; $\phi_G = \phi \circ \dett$ and $\phi \cdot \Stt_G$ 
where $\phi$ ranges over all characters of $F^{\times}$. A character of $T(F)$ resp. of $F^{\times}$ is said to be \textit{ unramified}, 
if it is trivial on $T(\mathfrak{o}_F)$ resp. on $\mathfrak{o}_F^{\times}$.

\begin{thm} Let $1 \neq w \in W_a$ such that $X_w(1) \neq \emptyset$.
\begin{itemize}
\item[(i)] Let $\chi$ be a character of $T(F)$. Then
$$\Homm\nolimits_G(H^{\ell(w)+1}_c(\overline{X}_w(1), \cool), \Indd\nolimits^{G}_{B}\chi) = \begin{cases} \cool(\frac{3-\ell(w)}{2}) & \text{ if } \chi 
\text{ unramified, }\\ 0 & \text{ else. }\end{cases}$$

\item[(ii)] Let $\phi$ be a character of $F^{\times}$. Then 
$$\Homm\nolimits_G(H^{\ell(w)+1}_c(\overline{X}_w(1), \cool), \phi_G) = \begin{cases} \cool(\frac{3-\ell(w)}{2}) & \text{ if } \phi \text{ unramified, }\\ 0 
& \text{ else }\end{cases}$$

\item[(iii)] Let $\phi$ be a character of $F^{\times}$. Then $\Homm\nolimits_G(H^{\ell(w)+1}_c(\overline{X}_w(1), \cool), \phi \cdot \Stt_G) = 0.$
\item[(iv)] Let $\pi$ be a cuspidal representation of $G$. Then $\Homm\nolimits_G(H^{\ell(w)+1}_c(\overline{X}_w(1), \cool), \pi) = 0.$
\end{itemize}
\end{thm}

\begin{thm} Let $1 \neq w \in W_a$ such that $X_w(1) \neq \emptyset$.
\begin{itemize}
\item[(i)] Let $\chi$ be a character of $T(F)$. Then
$$\Homm\nolimits_G(H^{\ell(w)}_c(\overline{X}_w(1), \cool), \Indd\nolimits^{G}_{B}\chi) = \begin{cases} \cool(\frac{1-\ell(w)}{2}) & \text{ if } \chi 
\text{ unramified, }\\ 0 & \text{ else. }\end{cases}$$

\item[(ii)] Let $\phi$ be a character of $F^{\times}$. Then $\Homm\nolimits_G(H^{\ell(w)}_c(\overline{X}_w(1), \cool), \phi_G) = 0$.

\item[(iii)] Let $\phi$ be a character of $F^{\times}$. Then 
$$\Homm\nolimits_G(H^{\ell(w)}_c(\overline{X}_w(1), \cool ), \phi \cdot \Stt\nolimits_G) = \begin{cases} \cool(\frac{1-\ell(w)}{2}) & \text{ if } \phi \text{ unramified, }\\ 0 
& \text{ else. }\end{cases}$$

\item[(iv)] Let $\pi$ be a cuspidal representation of $G$. Then $\Homm_G(H^{\ell(w)}_c(\overline{X}_w(1), \cool), \pi) = 0.$
\end{itemize}
\end{thm}
The essential ingredient in the proof of these Theorems (as well as the corresponding Theorems for $b \neq 1$) is the Frobenius reciprocity 
(\cite{BH} 2.4-5) which one has to apply several times. 

Let now $b =  \begin{pmatrix} 1 & 0 \\ 0 & t^{\alpha} \end{pmatrix}$ with $\alpha > 0$ and $w \in W_a$ with $\ell(w) > \alpha$ 
such that $X_w(b) \neq \emptyset$. We have: $J_b = T(F) = T(\mathfrak{o}_F) \times \mathbb{Z}^2$. The groups $H^r_c(X_w(b),\cool)$ are zero for 
$r \neq \ell(w) -\alpha, \ell(w) - \alpha + 1$. For $r = \ell(w) -\alpha, \ell(w) - \alpha + 1$, we have an isomorphism of $T(F)$-representations: 
$$H^r_c (\overline{X}_w(b),\cool) \cong \cIndd\nolimits_{T(\mathfrak{o}_F)}^{T(F)} 1_{\cool}.$$
In particular, $T(\mathfrak{o}_F)$ acts trivial on $H^r_c (\overline{X}_w(b),\cool)$, and this representation is inflated from the regular representation 
of $T(F)/T(\mathfrak{o}_F) = \mathbb{Z}^2$ (Theorem \ref{carbn}). Thus the only information encoded in these representations is the action of $J_b$ 
on the set of the connected components of $X_w(b)$ by permutation. A similar situation occurs also in the next case. An analogous result is proven for 
the groups $SL_2$, $SL_3$ in \cite{Zb}.

Let now $b = b_1 = \begin{pmatrix} 0 & 1 \\ t & 0 \end{pmatrix}$, and let $w \in W_a$ be such that $X_w(b_1) \neq \emptyset$. 
Then $J_{b_1} = D^{\times}$ is the multiplicative group of the quaternion division algebra $D$ over $F$, and 
$$H^r_c( \overline{X}_w(b_1), \cool) \cong \begin{cases} \cIndd_{U_D}^{D^{\times}} 1_{\cool}(\frac{\ell(w)}{2}) & \text{if } r = \ell(w),\\ 0 & \text{else},\end{cases}$$
are isomorphisms of $D^{\times} \times \Gamma$-modules. 

For any character $\chi$ of $F^{\times}$ let $\chi_D = \chi \circ \dett$ be a character of $D^{\times}$. These are
all one-dimensional representations of $D^{\times}$ (\cite{BH} 53.5). Let further $R \mathbb{Z}$ denote the regular representation of $D^{\times}/U_D = \mathbb{Z}$.

\begin{thm} Let $w \in W_a$ such that $X_w(b_1) \neq \emptyset$. Then
$$H^{\ell(w)}_c (\overline{X}_w(b_1), \cool) \cong \inff\nolimits_{\mathbb{Z}}^{D^{\times}} R \mathbb{Z},$$
 as  $D^{\times}$-representations. 
\begin{itemize}
\item[(i)] Let $\chi$ be a character of $F^{\times}$. Then 

\begin{equation}\nonumber \Homm\nolimits_{D^{\times}} (H^{\ell(w)}_c( \overline{X}_w(b_1), \cool), \chi_D) = 
\begin{cases} \cool(-\frac{\ell(w)}{2}) & \text{ if } \chi \text{ unramified, } \\ 0 & \text{ else.} \end{cases} \end{equation}
\item[(ii)] Let $\pi$ be an  irreducible representation of $D^{\times}$ of dimension $\geq 2$. Then \\
$$\Homm\nolimits_{D^{\times}} (H^{\ell(w)}_c( \overline{X}_w(b_1), \cool), \pi) = 0.$$
\end{itemize}
\end{thm}

\textbf{Acknowledgments:} This paper is my diploma thesis, written at the University of Bonn. I want to thank my advisors M. Rapoport and U. Goertz 
for the interesting thema, for explanations on this subject and for many helpful discussions.
 
\section{Affine Deligne-Lusztig varieties inside the affine flag manifold}

\subsection{Notations} \mbox{}

Let $k$ be a field with $q$ elements, and let $\bar{k}$ be an algebraic closure of $k$. Let $\sigma$ denote the Frobenius morphism of $\bar{k}/k$. 
We set $F = k((t))$ and $L = \bar{k}((t))$. We extend $\sigma$ to the Frobenius morphism of $L/F$. Thus $\sigma(\sum a_n t^n) = \sum \sigma(a_n)t^n$. We write 
$\mathfrak{o} = \bar{k}[[t]]$ and $\mathfrak{o}_F = k[[t]]$ for the valuation rings of $L$ and $F$, and $\mathfrak{p}$ and $\mathfrak{p}_F$ 
for their maximal ideals. Furthermore, we denote the valuation on $L$ by $v_L$. 

\subsection{The Bruhat-Tits building} \mbox{}

In this subsection, we recall the Bruhat-Tits buildings of the groups $SL_2(L)$ and $SL_2(F)$, and prove some facts about them which we will need later on. 
A detailed discussion can be found in \cite{Br}, chapters 4 and 5.

Let $\sim$ be the equivalence relation on the set of all $\mathfrak{o}$-lattices in $L^2$ given by $\mathfrak{L} \sim \mathfrak{L}^{\prime}$ 
if and only if there is a scalar $c \in L$ with $c\mathfrak{L} = \mathfrak{L}^\prime$.

\begin{Def} The Bruhat-Tits building $\build$ of $SL_2(L)$ is the one-dimensional simplicial complex such that 
\begin{itemize}
\item[(i)] its 0-dimensional simplices (vertices) are equivalence classes under $\sim$ of $\mathfrak{o}$-lattices $\mathfrak{L} \subset L^2$; 

\item[(ii)] two vertices are connected by a 1-dimensional simplex (alcove) if and only if there are representatives $\mathfrak{L}_0, \mathfrak{L}_1$ of them, 
such that $t \mathfrak{L}_1 \subsetneq \mathfrak{L}_0 \subsetneq \mathfrak{L}_1$. 
\end{itemize}
\end{Def}

Then $\build$ is a tree (a graph in which every two vertices are connected by exactly one path) with infinitely many vertices and infinitely many 
alcoves containing a fixed vertex. One can accomplish the same construction with $F$ and $\mathfrak{o}_F$ instead of $L$ and $\mathfrak{o}_L$. We denote 
the simplicial complex arising from this construction by $\ratbuild$. It is the Bruhat-Tits building of $SL_2(F)$. It is again a tree with infinitely many 
vertices. For a fixed vertex there are exactly $q+1$ alcoves containing it. We see $\ratbuild$ as a subset of $\build$, by sending 
an $\mathfrak{o}_F$-module $\mathfrak{L} \subseteq F^2$ to $\mathfrak{L}\otimes_{\mathfrak{o}_F} \mathfrak{o} \subseteq L^2$. 
Further $\sigma$ acts on $\build$, and $\ratbuild$ are exactly the fixed points. 

\begin{Def} Let $\mathfrak{L}$ be a $\mathfrak{o}$-lattice in $L^2$, and let $m \in \mathbb{Z}$ be such that $\wedge^2 \mathfrak{L} = t^m\mathfrak{o}$. 
The type of the vertex represented by $\mathfrak{L}$ is 
$$ \begin{cases}0 & \text{ if } m \text{ is even,} \\ 1 & \text{ if } m  \text{ is odd}. \end{cases}$$ \end{Def}

Since scalar multiplication changes this integer by some even number, the Definition is independent of the choice of the representing lattice.

\begin{Def} 
\mbox{}
\begin{itemize}
\item[(i)] A gallery in $\build$ is a sequence $(C_0, C_1, ..., C_n)$ of alcoves such that $C_i$ and $C_{i+1}$ are adjacent (i.e. have a common 
vertex) for every $0\leq i \leq n-1$. A gallery is non-stuttering if $C_i \neq C_{i+1}$ for all $i$. 

\item[(ii)] The length of the gallery $\Gamma = (C_0, C_1, ..., C_n)$ is defined as $$\ell(\Gamma) = n.$$

\item[(iii)] If $\Gamma = (C_0, C_1, ..., C_n)$ and $\Gamma^{\prime} = (C^{\prime}_0, C^{\prime}_1..., C^{\prime}_m)  $ are two galleries such 
that $C_n, C_{0}^{\prime}$ are adjacent, then the composite gallery is 
$$(\Gamma, \Gamma^{\prime}) = (C_0, C_1..., C_n, C^{\prime}_0, C^{\prime}_1..., C_m^{\prime}).$$ 

\item[(iv)] The inverse gallery of $\Gamma = (C_0, C_1, ..., C_n)$ is the gallery $$\Gamma^{-1} = (C_n, C_{n-1}..., C_0).$$
\end{itemize}
\end{Def}
If $\Gamma, \Gamma^{\prime}$ are galleries, then the length of the composed gallery $(\Gamma, \Gamma^{\prime})$ is given by 
$$\ell(\Gamma, \Gamma^{\prime}) = \ell(\Gamma) + \ell(\Gamma^{\prime}) + 1.$$
\noindent For every two alcoves $C, D$ in $\build$ resp. $\ratbuild$ there is a gallery of minimal length containing them both and having $C$ 
as the first alcove. Since $\build$ resp. $\ratbuild$ is a tree, this gallery is unique. 

\begin{Def}\label{distmin} A gallery $\Gamma = (C_0,C_1...,C_n)$ is called minimal if it is the minimal gallery connecting $C_0$ and $C_n$.
\end{Def}

\begin{Def} Let $\Gamma = (C_0, C_1, ..., C_n)$ be a gallery in $\build$. A first vertex of $\Gamma$ is a vertex of $C_0$ 
which satisfies the following condition:
\begin{itemize}
\item[(*)] It is not a vertex of $C_1$ if $n > 0$ and $C_0 \neq C_1$.
\end{itemize}
A last vertex of $\Gamma$ is a first vertex of the inverse gallery $\Gamma^{-1}$.
\end{Def}

This means the following: if $\ell(\Gamma) = 0$ or $\ell(\Gamma) > 0$ and $C_0 = C_1$, then every vertex of its first alcove is 
a first vertex of $\Gamma$. If $\ell(\Gamma) > 0$ and $C_0 \neq C_1$, then $\Gamma$ has a unique fisrt vertex: it is the vertex 
of $C_0$ which is not a vertex of $C_1$. In particular, every gallery has at least one first resp. last vertex. 
If a gallery has more than one alcove, and is minimal, then the first resp. last vertex is unique.
We say that a gallery contains a vertex if it contains an alcove which has this vertex as one of its faces. 
Since $\build$ is a tree, there is a unique gallery with minimal length, containing a vertex $P$ and an alcove $C$ of $\build$, and having $P$ as a first 
vertex. Analogously, there is a unique gallery of minimal length, containing two distinct vertices $P, Q$ of $\build$, and having $P$ as a first vertex.
Such galleries are minimal in the sense of Definition \ref{distmin}. 

If $P,Q$ are some simplices in $\build$, then we say that a gallery $\Gamma$ is \textit{stretched from $P$ to $Q$} if $\Gamma$ is minimal and has $P$ 
resp. $Q$ as a first resp. last vertex or alcove.

\begin{Def} \mbox{}
\begin{itemize}
\item[(i)]  The distance between two alcoves $C$, $D$ in $\build$ is the length of the gallery stretched from $C$ to $D$.
\item[(ii)] The distance between a vertex $P$ and an alcove $D$ in $\build$ is the length of the gallery stretched from $P$ to $D$. 
\end{itemize}
\end{Def}

We will need the following characterisation of minimal galleries:

\begin{Lm}\label{MinimalityLm}
\mbox{} 
\begin{itemize}
\item[(i)] A gallery $\Gamma = (C_0, C_1,...,C_n)$ in $\build$ resp. in $\ratbuild$ is minimal if the following conditions hold:
\begin{itemize}
\item[(a)] $\Gamma$ contains no alcove twice.

\item[(b)] For every alcove $C_i$ with $0 < i < n$ the following holds: if $P$, $Q$ are the two vertices of $C_i$, and $P$ is a vertex of $C_{i-1}$, 
then $Q$ is a vertex of $C_{i+1}$.
\end{itemize}

\item[(ii)] Let $\Gamma = (C_0, C_1,...,C_n), \Gamma^{\prime} = (C^{\prime}_0, C^{\prime}_1..., C^{\prime}_m)$ be two minimal galleries such that 
$C_n$ and $C^{\prime}_0$ are adjacent. Then the composite gallery $(\Gamma, \Gamma^{\prime})$ is minimal if the following conditions hold:
\begin{itemize}
\item[(c)] $\Gamma$, $\Gamma^{\prime}$ have no common alcoves.
 
\item[(d)] A last vertex of $\Gamma$ is a first vertex of $\Gamma^{\prime}$.
\end{itemize}
\end{itemize}
\end{Lm}

\begin{proof} The proofs for $\build$ and $\ratbuild$ are the same, thus we restrict ourselves to the case of $\build$.
At first we prove (i). We proceed by induction. The cases $\ell(\Gamma)= 0$ 
and $\ell(\Gamma) = 1$ are trivial. Assume now that (a) and (b) imply minimality for every gallery of length $\leq n$ with $n \geq 1$. 
Let $\Gamma = (C_0,C_1..., C_{n+1})$ be a gallery satisfying the conditions (a) and (b). Then the gallery $\Gamma_1 = (C_0,C_1, ..., C_n)$ satisfies these conditions, too 
and is minimal by the induction hypothesis. Now, removing the alcove $C_n$ (without vertices) from $\build$, divides $\build$ in two connected components, 
both of which are trees. By (a), the alcoves $C_{n-1}$, $C_n$, $C_{n+1}$ are pairwise distinct and by (b), the alcoves $C_{n-1}$, $C_{n+1}$ are not adjacent 
(otherwise $C_{n-1}$, $C_n$, $C_{n+1}$ would give a non-trivial cycle) and thus $C_{n-1}$, $C_{n+1}$ lie in different connected components, described above. 
By minimality of $\Gamma_1$, the gallery $(C_0,C_1,...,C_{n-1})$ does not contain $C_n$. Thus $C_0$ lies in the same connected component 
as $C_{n-1}$. Thus every gallery connecting $C_0$ with $C_{n+1}$ contains $C_n$. This holds for the unique gallery $\Gamma_{min}$ stretched from $C_0$ to $C_{n+1}$, 
and thus it will have the form $\Gamma_{min} = (C_0,D_1,...,D_r,C_n, D_{r+1}...,D_{s}, C_{n+1})$. Hence $(C_0,D_1,...,D_r,C_n)$ and $(C_n, D_{r+1},...,D_{s}, C_{n+1})$ are minimal.
But $\Gamma_1 = (C_0,C_1,...C_n)$ is minimal and $C_n$, $C_{n+1}$ are adjacent, thus $\Gamma_{min} = (C_0,C_1,...C_n,C_{n+1}) = \Gamma$. 
Hence $\Gamma$ is minimal.

To prove (ii) we notice first that from (c) and minimality of $\Gamma$, $\Gamma^{\prime}$ the condition (a) of part (i) for the composed gallery 
$(\Gamma, \Gamma^{\prime})$ follows. The condition (b) is also clear for all alcoves of $(\Gamma, \Gamma^{\prime})$ except for $C_n$ and $C^{\prime}_0$. If 
$\ell(\Gamma) = 0$, then the condition (b) for $C_n$ is empty. Assume that $\ell(\Gamma) > 0$. Then, since $\Gamma$ is minimal, its last vertex is unique: 
it is the vertex $P$ of $C_n$ which is not a vertex of $C_{n-1}$. Now by (d), $P$ is a common vertex of $C_n$ and $C^{\prime}_0$. This is exactly the condition (b) 
of (i)  for the alcove $C_n$. The verification of the condition (b) for $C_0^{\prime}$ can be done similarly 
(one can also invert all involved galleries and use above considerations again).
\end{proof}

Let $$I = \begin{pmatrix} \mathfrak{o}^{\times} & \mathfrak{o} \\ \mathfrak{p} & \mathfrak{o}^{\times}\end{pmatrix}$$ 
\noindent be the standard Iwahori subgroup of $GL_2(L)$, and $$I^{SL_2} = I \cap SL_2(L)$$ the standard Iwahori subgroup of $SL_2(L)$.

The groups $SL_2(L)$ and $GL_2(L)$ act transitively   on the set of the alcoves of $\build$. Further $SL_2(L)$ acts transitively
on all vertices with the same type $m$ in $\build$ (Lemma \ref{transit}), and $GL_2(L)$ acts transitively on the set of all vertices.
One has an obvious base vertex of type $0$ represented by $\mathfrak{o} \oplus \mathfrak{o}$ with stabilizer $SL_2(\mathfrak{o})$, resp. $GL_2(\mathfrak{o})$.

\begin{Def} 
 Let $C^{0}_M$ denote the base alcove, represented by $\mathfrak{o} \oplus t\mathfrak{o} \subsetneq \mathfrak{o} \oplus \mathfrak{o}$.
\end{Def}

The stabilizer of $C^0_M$ under the action of $SL_2(L)$ is $I^{SL_2}$. Let $T$ be the diagonal 
torus of $GL_2$, and $T^{SL_2}  = T \cap SL_2$ the diagonal torus of $SL_2$. Let further $N_{SL_2}(T^{SL_2})$ denote the normalizer of $T^{SL_2}$ 
in $SL_2$. Then we define:

\begin{Def} 
\mbox{}
\begin{itemize}
\item[(i)] The standard apartment $A_M$ is the minimal full subcomplex of $\build$ whose alcoves lie in the $N_{SL_2}(T^{SL_2})(L)$-orbit of $C^{0}_M$. \\
\item[(ii)] The affine Weyl group of $SL_2(L)$ is the group $$W_a = N_{SL_2}(T^{SL_2})(L)/(T^{SL_2})(\mathfrak{o}).$$
\end{itemize}
\end{Def}

One has $T^{SL_2}(\mathfrak{o}) = N_{SL_2}(T^{SL_2})(L) \cap I^{SL_2}$. Thus $W_a$ acts simply transitively on the set of the alcoves in the standard apartment. 
By the choice of the alcove $C^0_M$, we identify $W_a$ with $A_M$. Furthermore we number the alcoves in $A_M$ by integers 
and call the $i$-th alcove $C_M^i$ such that $C_M^i$ is represented by 
$\mathfrak{o} \oplus t^{i+1}\mathfrak{o} \subsetneq \mathfrak{o}\oplus t^i\mathfrak{o}$.
Then $\begin{pmatrix} t^{-i} & 0 \\ 0 & t^i \end{pmatrix}$ corresponds under the above identification to $C^{2i}_M$ and 
$\begin{pmatrix} 0 & t^{-i} \\ -t^i & 0 \end{pmatrix}$ corresponds to $C_M^{2i-1}$ for $i \in \mathbb{Z}$. 

The group $W_a$ is a Coxeter group on two generators of order two with no further relations. Let $\ell(w)$ denote the length of the element 
$w \in W_a$. If $wC_M^0 = C_M^i$, then $\ell(w) = \abs{i}$.


We have a Bruhat decomposition of $SL_2(L)$: $$SL_2(L) = \bigcup_{w \in W_a} I^{SL_2} w I^{SL_2}.$$ 

\begin{Def} 
The relative position map on the set of the alcoves is $$\inv: SL_2(L)/I^{SL_2} \times SL_2(L)/I^{SL_2} \rightarrow W_a.$$ 
It maps the cosets $xI^{SL_2}$, $yI^{SL_2}$ to the unique element $w$ of the affine Weyl group such that $x^{-1}y \in I^{SL_2}wI^{SL_2}$.  
\end{Def}

Via the above identification of $W_a$ with $A_M$, we see the relative position of two alcoves as an alcove in the standard apartment. 
Let $D, D^{\prime}$ be two alcoves in $\build$ and let $\Gamma$ be the gallery, stretched from $D$ to $D^{\prime}$, with length $i$. Then 
$$\inv(D, D^{\prime}) = \left\{ \begin{array}{cl}C^i_M & \text{ if a first vertex of } \Gamma \text{ has type 0,} \cr C^{-i}_M & \text{ if a 
first vertex of }\Gamma \text{ has type 1.} \end{array}\right.$$

\begin{Def} Let $N_{GL_2}(T)$ be the normalizer of $T$ in $GL_2$. The extended affine Weyl group is $$\tilde{W} = N_{GL_2}(T)(L)/T(\mathfrak{o}).$$
\end{Def}
 
Then a Bruhat-decomposition of $GL_2(L)$ is given by: 
$$GL_2(L) = \bigcup_{\tilde{w} \in \tilde{W}} I\tilde{w} I.$$

\noindent The affine and the extended affine Weyl groups are closely related: the inclusion \\ 
$N_{SL_2}(T^{SL_2})(L) \hookrightarrow N_{GL_2}(T)(L)$ induces a short exact sequence: 
$$ 1 \rightarrow W_a \rightarrow \tilde{W} \rightarrow \mathbb{Z} \rightarrow 0, $$
where the last map is induced by $v_L \circ \dett$. This sequence is split by the map sending $m \in \mathbb{Z}$ to $\begin{pmatrix} 0 & 1 \\ t & 0\end{pmatrix}^m$, 
and we have $\tilde{W} = W_a \rtimes \mathbb{Z}$. Thus an element of $\tilde{W}$ is given by a pair $(w,m)$, where $w \in W_a$
 and $m \in \mathbb{Z}$. 



\subsection{The affine flag manifold}\label{affflmfd} \mbox{}

In this subsection we define the affine flag manifolds of $SL_2$ and $GL_2$ and show that they are Ind-schemes. 
A similar situation in characteristic $0$ with the affine Grassmanian instead of the affine flag manifold is discussed in (\cite{BL} section 2).

If $R$ is a $k$-algebra, then a $R[[t]]$-lattice in $R((t))^2$ is a $R[[t]]$-submodule $\mathfrak{L}$ such that there exists a 
positive integer $N$ with $t^N R[[t]]^2 \subseteq \mathfrak{L} \subseteq t^{-N} R[[t]]^2$ and $t^{-N}R[[t]]^2/\mathfrak{L}$ is a locally on $\Specc(R)$ 
free module of finite rank.

\begin{Def} 
The affine flag manifold $X^{SL_2}$ of $SL_2$ is the functor, whose $R$-valued points are 
$$X^{SL_2}(R) = \left\{\mathfrak{L}_0 \subsetneq \mathfrak{L}_1 \subsetneq R((t))^2 \colon \begin{array}{cl} \mathfrak{L}_0, \mathfrak{L}_1 \text{ }
R[[t]]\text{-lattices, } \wedge^{2} \mathfrak{L}_0 = R[[t]],\cr t\mathfrak{L}_1 \subsetneq \mathfrak{L}_0 \subsetneq \mathfrak{L}_1, 
\mathfrak{L}_0/t\mathfrak{L}_1\text{ is loc.}  \cr \text{ on }\Specc(R) \text{ free of rank }1 \end{array}\right\}.$$ 
\end{Def}

Then $X^{SL_2}$ is an Ind-scheme: set 
$$X^{SL_2,n}(R) = \{ (\mathfrak{L}_0 \subsetneq \mathfrak{L}_1) \in X^{SL_2}(R)\colon t^n R[[t]]^2 \subseteq \mathfrak{L}_0 \subsetneq \mathfrak{L}_1 
\subseteq t^{-n} R[[t]]^2 \}.$$

\noindent Then $X^{SL_2,n}$ is isomorphic to a closed subscheme of a partial flag manifold of \\
$V = t^{-n} k[[t]]^2 /t^n k[[t]]^2$: let $\mathfrak{Flag}_{2n, 2n+1} (V)$ be the partial flag manifold of $V$ 
such that for every $k$-algebra $R$ the $R$-valued points are 

$\mathfrak{Flag}_{2n, 2n+1} (V) (R) = $
$$ = \left\{ L_0 \subsetneq L_1 \subsetneq R \otimes V \colon \begin{array}{cl} L_i, R\otimes V/L_i \text{ are loc. on } \Specc(R) \text{ free } 
R\text{-modules} \cr \text{for }i = 0,1; \text{rk} (L_0) = 2n, \text{rk} (L_1) = 2n + 1 \end{array} \right\}.$$

\noindent Let $\mathfrak{Flag}_{2n, 2n+1}^t (V)$ be the closed subscheme of $\mathfrak{Flag}_{2n,2n+1}(V)$ parametrizing the chains of $t$-stable subspaces of 
$V$ of dimensions $2n, 2n+1$. There is a closed embedding: 
\begin{equation}\label{affflmfdfrml} X^{SL_2,n} \longrightarrow \mathfrak{Flag}_{2n, 2n+1}^t (V),\end{equation}
$$(\mathfrak{L}_0 \subsetneq \mathfrak{L}_1) \mapsto (\mathfrak{L}_0/t^n k[[t]]^2 \subsetneq \mathfrak{L}_1/t^n k[[t]]^2),$$
which is an isomorphism on $K$-valued points for every field $K$ containing $k$ (compare \cite{BL} 2.4. There is also remarked that, 
as Genestier pointed out the scheme on the right hand side is in general not reduced, even in the easiest case). 
Thus we have a filtration of $X^{SL_2}$ by closed subschemes of finite type over $k$: $$X^{SL_2,1} \subsetneq X^{SL_2,2} \subsetneq ... \subsetneq X^{SL_2}.$$


\noindent The group $SL_2(L)$ acts transitively by linear transformations on $X^{SL_2}(\bar{k})$ and the stabilizer of the $\bar{k}$-valued point 
$\mathfrak{o} \oplus t\mathfrak{o} \subsetneq \mathfrak{o} \oplus \mathfrak{o} $ is $I^{SL_2}$, thus $$X^{SL_2}(\bar{k}) = SL_2(L)/I^{SL_2}.$$ 
We  identify the set of $\bar{k}$-valued points of $X^{SL_2}$ with the set of alcoves in $\build$. Under this identification the $k$-valued points 
$X^{SL_2}(k) = SL_2(F)/(I^{SL_2} \cap SL_2(F))$ correspond to the alcoves in $\ratbuild$. 

The Ind-scheme $X^{SL_2}$ is reduced (\cite{PR} 6.1), and connected (\cite{PR} 5.1).   

\begin{Def} The affine flag manifold $X$ of $GL_2$ is the functor whose $R$-valued points are 
$$X(R) = \left\{\mathfrak{L}_0 \subsetneq \mathfrak{L}_1 \subsetneq R((t))^2 \colon \begin{array}{cl} \mathfrak{L}_0, \mathfrak{L}_1 \text{ are }R[[t]]\text{-lattices, } 
t\mathfrak{L}_1 \subsetneq \mathfrak{L}_0 \subsetneq \mathfrak{L}_1, \mathfrak{L}_0/t\mathfrak{L}_1 \cr \text{ is loc. on }\Specc(R)\text{ free of rank }1 \end{array}\right\}.$$ \end{Def}

Then $X$ is again an Ind-scheme, and we have: $X(\bar{k}) = GL_2(L)/I$. Set further

$$X^{(v)}(R) = \left\{ (\mathfrak{L}_0 \subsetneq \mathfrak{L}_1) \in X(R) \colon \begin{array}{cl}\text{ for every morphism} \cr 
\text{Spec}(K) \rightarrow \text{Spec}(R)\text{, with }K \text{ field:}\cr \wedge^2 (\mathfrak{L}_0\otimes K) = t^vK[[t]] \end{array} \right\}.$$ 

\noindent Then $$X = \coprod_v X^{(v)}.$$
The Ind-scheme $X$ is not reduced (\cite{PR} 6.5) in contrast to $X^{SL_2}$. We have:
$$(X^{(v)})_{\redd}(R) = \{ (\mathfrak{L}_0 \subsetneq \mathfrak{L}_1) \in X(R) \colon \wedge^2\mathfrak{L}_0 = t^v R[[t]]\}. $$
If $r_v \in GL_2(L)$ satisfies $v_L(\dett(r_v))=v$, then the left multiplication by $r_v$ gives an isomorphism 
$X^{SL_2} \stackrel{\sim}{\rightarrow} (X^{(v)})_{\redd}$. 

We identify $X^{SL_2}$ with $(X^{(0)})_{\redd}$ by choosing $r_0 = 1$:  
$$X^{SL_2} \stackrel{\sim}{\rightarrow} (X^{(0)})_{\redd} \hookrightarrow X_{\redd}.$$ 
Let $H \subset GL_2(L)$ be the subgroup of all matrices with the valuation of the determinant equal $0$. On the $\bar{k}$-valued points 
the above inclusion is given by 
$$X^{SL_2}(\bar{k}) = SL_2(L)/I^{SL_2} = H/I \hookrightarrow GL_2(L)/I.$$
Thus, on the left hand side stands the set of all alcoves in $\build$. The group $H$ acts on them by linear transformations, and this action corresponds 
on the right hand side to left multiplication on the set of the cosets. Further, the action of $H$ on the vertices of $\build$ is type-preserving.


By the above discussion, $X_{\redd}$ is isomorphic to a disjoint union of $\mathbb{Z}$ copies of $X^{SL_2}$.

\begin{Lm}\label{ReprLemma} Let $R = \{r_v \colon v \in \mathbb{Z}\}$ be a subset of $GL_2(L)$ with $v_L(\dett(r_v)) = v$ for every $v \in \mathbb{Z}$. Then there is an
 isomorphism of Ind-schemes: $$\alpha_R: \coprod_v X^{SL_2} \rightarrow X_{\redd}$$ such that $\alpha_R(x_vI^{SL_2}) = r_vx_vI$ for every $v \in
 \mathbb{Z}$.\end{Lm} 

Since $X^{SL_2}$ is connected we have: $\pi_0(X_{\redd}) = \mathbb{Z}$ and the connected component of the coset $I$ is $X^{(0)}$. 


\subsection{Affine Deligne-Lusztig varieties in the affine flag manifold of $GL_2(L)$}\label{ADLVDef} \mbox{}

\begin{Def} Let $\tilde{w} \in \tilde{W}$ and $b \in GL_2(L)$. The affine Deligne-Lusztig variety for $GL_2(L)$ is the locally closed subset of $X$ 
given by $$X_{\tilde{w}}(b) = \{ xI \in GL_2(L)/I \colon x^{-1}b \sigma(x) \in I\tilde{w}I\},$$ 
provided with the reduced sub-Ind-scheme structure. \end{Def}

That $X_{\tilde{w}}(b)$ is indeed locally closed will follow from Proposition \ref{LmGLSL} and the results of the next two sections.
For any $b \in GL_2(L)$ the $\sigma$-stabilizer of $b$ in $GL_2(L)$ is given by

$$J_b = \{g \in GL_2(L) \colon g^{-1}b\sigma(g) = b\}.$$

\begin{Rem} {\rm The group $J_b$ is the group of $F$-valued points of the functor $\tilde{J}_b$ which associates to an $F$-algebra $R$ the group }
$$\tilde{J}_b (R) = \{ g \in GL_2(R \otimes_F L) \colon g^{-1}b\sigma(g) = b\}.$$
{ \rm This functor is representable by a connected reductive group $\tilde{J}_b$ over $F$, which is an inner form of a certain Levi subgroup 
$M_b$ of $GL_2$ attached to $b$ (compare \cite{RZ} (1.12) for the Witt ring case).}
\end{Rem}

The group $J_b$ acts by left multiplication on $X_{\tilde{w}}(b)$. Now we will give the three examples, which will be relevant for us. In the rest of the paper we will 
use the following notation: $$b_1 := \begin{pmatrix} 0 & 1 \\ t & 0 \end{pmatrix}.$$

\noindent \textbf{Examples.} 
\begin{itemize} 
\item[(i)] For $b=1$ we have $J_1 = GL_2(F)$.

\item[(ii)] For $b = \begin{pmatrix} 1 & 0 \\ 0 & t^{\alpha} \end{pmatrix}$ with $\alpha > 0$ we have $J_b = T(F)$.

\item[(iii)] Let $b = b_1$. Then $J_{b_1}$ is the multiplicative group of the quaternion algebra over $F$. 
Let $F \subsetneq E \subsetneq L$ be the unramified extension of $F$ of degree 2. Then
$$J_{b_1} = \left\{  \begin{pmatrix} a & \sigma(c) \\ tc & \sigma(a) \end{pmatrix} \colon a, c \in E, a\sigma(a) - tc\sigma(c) \neq 0\right\}.$$

\noindent We can also write: $$J_{b_1} \cong E[\pi]^{\ast},$$
where $D = E[\pi]$ is the (non-commutative) quaternion algebra over $F$ defined by the relations $\pi^2 = t$ and $a\pi = \pi \sigma(a)$ for all $a \in E$. 
The isomorphism is given by  sending $a \in E$ to $\begin{pmatrix} a & 0 \\ 0 & \sigma(a) \end{pmatrix}$ and $\pi$ to $b_1$.
\end{itemize}

The determinant of matrices in $I$ always has the valuation $0$ and 

$$v_L(\dett(x^{-1}b\sigma(x))) = v_L(\dett(b)).$$ 
From this follows 

\begin{Lm}\label{WMLemma} If $X_{(w,m)}(b)$ is non-empty, then $v_L(\dett(b)) = m$. \end{Lm}

\noindent \textbf{Convention.} For transparency, we omit $m$ from the notation and write $X_w(b)$ instead of 
$X_{(w, v_L(\dett(b)))}(b)$ for every $w \in W_a$ and $b \in GL_2(L)$. This notation depends on the chosen splitting $\tilde{W} = W_a \rtimes \mathbb{Z}$.

\begin{Lm}\label{lemma3} 
\mbox{}
\begin{itemize}
\item[(i)] The varieties $X_w(b)$ and $X_w(g^{-1}b\sigma(g))$ are isomorphic. 
\item[(ii)] Let $c = \begin{pmatrix}  t^m & 0 \\ 0 & t^m \end{pmatrix} \in N_{GL_2}(T)(L)$. 
Then $X_w(b)$ and $X_w(cb)$ are equal as subvarieties of $X_{\redd}$.
\end{itemize}
\end{Lm} 

\begin{proof}
If $g \in GL_2(L)$, then $x \mapsto g^{-1}x$ is an isomorphism $X_w(b) \rightarrow X_w(g^{-1}b\sigma(g))$. This proves (i). 
To prove (ii), we notice that $c$ is central and the image of $c$ in $\tilde{W}$ is the pair $(1,2m)$. Let $v = v_L(\dett(b))$. 
Now we have: $xI \in X_w(b) \Leftrightarrow x^{-1}b\sigma(x) \in I(w, v)I \Leftrightarrow x^{-1}cb\sigma(x) \in 
I(1,2m)(w,v)I = I(w, v+2m)I$. But $v + 2m = v_L(\dett(cb))$, and thus $xI \in X_w(b)$ is equivalent 
to $xI \in X_{w}(cb)$.
\end{proof}

Now we show the connection between $X_w(b)$ and affine Deligne-Lusztig varieties in the affine flag manifold of $SL_2(L)$. We use the 
following 

\begin{Def}\label{xsld} Let $w \in W_a$, $b \in GL_2(L)$ and $\tilde{w} = (w,v_L(\dett(b))) \in \tilde{W}$. The affine Deligne-Lusztig variety for $SL_2(L)$ is the 
locally closed subset of $X^{SL_2}$ defined by $$X_w^{SL_2}(b) = \{ xI^{SL_2} \in SL_2(L)/I^{SL_2} \colon x^{-1}b \sigma(x) \in I\tilde{w}I\},$$ 
provided with its reduced subscheme structure. \end{Def}

That $X^{SL_2}_w(b)$ is indeed locally closed will follow from  the results of the next two sections.
Now, we want to understand Definition \ref{xsld} better.  We see $X^{SL_2}_w(b)(\bar{k})$ as a set of alcoves 
in $\build$. 

\begin{Lm}\label{SLGLident} Consider the situation as in the Definition \ref{xsld}. An alcove $D$ lies 
in $X^{SL_2}_w(b)(\bar{k})$ if and only if $\inv(D, b\sigma D) = w$.
\end{Lm}

\begin{proof} 
For better readability we assume that $\dett(b) = t^m$ for some $m \in \mathbb{Z}$. This is in fact the only interesting case for the future. 
The proof without this assumption is very similar. We fix a representative of $w$ in $N_{SL_2}(T^{SL_2})(L) \subseteq N_{GL_2}(T)(L)$ and denote 
it again by $w$. Let $D = xC^0_M$ with $x \in SL_2(L)$. Since $b_1 C_M^0 = C_M^0$ we have: 
$b\sigma D = b\sigma(x)b_1^{-m} C^0_M$ and $b\sigma(x)b_1^{-m} \in SL_2(L)$. Since $b_1 I^{SL_2}b_1^{-1} = I^{SL_2}$, we have:
\begin{eqnarray*} \inv(D, b\sigma D) = w &\Leftrightarrow& x^{-1}b\sigma(x) b_1^{-m} \in I^{SL_2}wI^{SL_2} \Leftrightarrow x^{-1}b\sigma(x) \in I^{SL_2}w b_1^m I^{SL_2} \Rightarrow \\ 
 &\Rightarrow& x^{-1}b\sigma(x) \in Iw b_1^m I = I \tilde{w} I \Leftrightarrow D \in X^{SL_2}_w(b)(\bar{k}).\end{eqnarray*}
Now we have to prove the converse of the third inclusion. Since $b_1$ normalizes $I$ and $I^{SL_2}$, it is enough to show that 
$IwI \cap SL_2(L) \subseteq I^{SL_2} w I^{SL_2}$. If $iwj \in IwI \cap SL_2(L)$, then $\dett(i) = \dett(j)^{-1} \in \mathfrak{o}^{\times}$ 
(since $\dett(w)=1$ and $i,j \in I$). Let $r = \begin{pmatrix} \dett(i) & 0 \\ 0 & 1\end{pmatrix}$. Now, $w$ normalizes $T(\mathfrak{o})$, and thus 
$r w = w r^{\prime}$ for some $r^{\prime} \in T(\mathfrak{o})$. Hence $iwj = ir^{-1} w r^{\prime} j$ where $\dett(ir^{-1}) = \dett(w) = 1$ and hence 
$\dett(r^{\prime} j) = 1$. Moreover, since $r,r^{\prime} \in T(\mathfrak{o})$, we have $ir^{-1}, r^{\prime} j \in I$. 
Thus $ir^{-1} , r^{\prime}j \in I^{SL_2}$ and $iwj = ir^{-1} w r^{\prime} j \in I^{SL_2}wI^{SL_2}$.
\end{proof}

Thus if $b \in SL_2(L)$, then $X^{SL_2}_w(b)$ is the usual affine Deligne-Lusztig variety inside the flag manifold of $SL_2$, attached to $b$ and $w$. 
If $b, b^{\prime} \in GL_2(L)$ are $\sigma$-conjugate under $SL_2(L)$ and $w \in W_a$, then the varieties $X_w^{SL_2}(b)$ and $X_w^{SL_2}(b^{\prime})$ 
are isomorphic as subvarieties of $X^{SL_2}$. Indeed if $g \in SL_2(L)$, then $x \mapsto g^{-1}x$ is an isomorphism $X_w^{SL_2}(b) 
\rightarrow X_w^{SL_2}(g^{-1}b\sigma(g))$.

The group $$J^{SL_2}_b = J_b \cap SL_2(L)$$ acts on $X^{SL_2}_w(b)$ by left multiplication. To express $X_w(b)$ in terms of $X_w^{SL_2}(b)$, we need 
the following

\begin{Lm}\label{LmSurj} Let $b \in GL_2(L)$. The restriction of $v_L \circ \dett: GL_2(L) \rightarrow \mathbb{Z}$ to $J_b$ is surjective. \end{Lm}

\begin{proof} If $b,c \in GL_2(L)$ are $\sigma$-conjugate, then the groups $J_b$ and $J_c$ are conjugate in $GL_2(L)$. Since conjugation
 does not change the determinant, it is enough to prove the statement for some representatives of the $\sigma$-conjugacy classes of $GL_2(L)$. Those are
 given by 
\begin{equation}\label{conjclasses} \left\{ \begin{pmatrix} t^\alpha & 0 \\ 0 & t^\beta \end{pmatrix} \colon \alpha \leq \beta \right\} \cup \left\{ \begin{pmatrix}  0 &
 t^{\alpha - 1} \\ t^\alpha & 0 \end{pmatrix} \colon \alpha \text{ odd} \right\}, \end{equation}
which follows from (\cite{RR} 1.10).

The group $J_b$ stays unchanged if we multiply $b$ by some central element of $GL_2(L)$. Therefore, we have only two cases, in which we can prove
 the statement explicitly.

First case: $b = \begin{pmatrix} 1 & 0 \\ 0 & t^{\alpha} \end{pmatrix}$ and $\alpha \geq 0$. Then for all $v \in \mathbb{Z}$ we have: $\begin{pmatrix} 
1 & 0 \\ 0 & t^v \end{pmatrix} \in J_b$ and $v_L(\dett\begin{pmatrix} 1 & 0 \\ 0 & t^v \end{pmatrix}) = v$.

Second case: $b = b_1$. In this case we have $b_1^v \in J_b$, and $v_L(\dett(b_1^v)) = v$ for all $v \in \mathbb{Z}$.
\end{proof}

We proved the surjectivity of the homomorphism $v_L \circ \dett : J_b \longrightarrow \mathbb{Z}$ for any $b \in GL_2(L)$. 
Let us now introduce its kernel.

\begin{Def} For $b \in GL_2(L)$ set $$H_b = \Kerr(v_L \circ \dett : J_b \longrightarrow \mathbb{Z}).$$ \end{Def}

Now $H_b \subseteq H$ where $H \subseteq GL_2(L)$ is the subgroup of all matrices with valuation of the determinant equal zero. $H_b$ is exactly 
the stabilizer of $X_w(b) \cap (X^{(0)})_{\redd}$ under the action of $J_b$ on $X_w(b)$. Recall that in the last subsection we identified 
$X^{SL_2}$ with $(X^{(0)})_{\redd}$. Via this identification $H_b$ acts on $X^{SL_2}_w(b) = X_w(b) \cap (X^{(0)})_{\redd}$, and this action 
is the restriction of the action of $H$ on $X^{SL_2}$.

\begin{Prop}\label{LmGLSL} We have 
$$X_w(b) \cong \coprod_{J_b/H_b} X^{SL_2}_w(b)$$
as $k$-varieties, and the $J_b$-action on the set of these components is given by left multiplication on the index set. 
\end{Prop}

\begin{proof} The scheme structure on $X_w(b)$ is the reduced one, thus the inclusion $X_w(b) \subset X$ factorizes through $X_{\redd} \rightarrow X$. 
By Lemma \ref{LmSurj}, we can choose $R = (r_v) \subseteq J_b$ with $r_0 = 1$ and $v_L(\dett(r_v)) = v$ for all $v \in \mathbb{Z}$, 
and $\alpha_R$ from Lemma \ref{ReprLemma} restricts to the isomorphism: 
$$X_w(b) \cong \coprod_{v \in \mathbb{Z}} X^{SL_2}_w(b).$$
Now the action of $J_b$ permutes these components and $H_b$ is exactly the stabilizer of the component corresponding to $r_0 = 1$. Again, by Lemma 
\ref{LmSurj}, $J_b$ acts transitively on the set of these components. 
\end{proof}

In particular, $X_w(b)$ is non-empty if and only if $X^{SL_2}_w(b)$ is. We clearly have $J^{SL_2}_b \subseteq H_b$. 

By Lemma \ref{lemma3}(i), to determine $X_w(b)$ for all $b \in GL_2(L)$ and all $w \in W_a$, it is enough to do so for all $b$ lying in a fixed set 
of representatives of the $\sigma$-conjugacy classes of $GL_2(L)$ (compare (\ref{conjclasses})). By Lemma \ref{lemma3}(ii) it is enough to consider 
the following three cases:
\begin{itemize}
\item[(i)] $b = 1$;
\item[(ii)] $b = \begin{pmatrix} 1 & 0 \\ 0 & t^{\alpha} \end{pmatrix}$ and $\alpha > 0$;
\item[(iii)] $b = b_1$. 
\end{itemize}

\noindent By Proposition \ref{LmGLSL}, $X_w(b)$ is the disjoint union of $\mathbb{Z}$ copies of $X^{SL_2}_w(b)$. In Lemma \ref{SLGLident}  we showed that 
 $X^{SL_2}_w(b)(\bar{k})$ is the set of all alcoves in $\build$ with $\inv(D, b \sigma D) = w$. 
In the following, we will determine $X^{SL_2}_w(b)(\bar{k})$.

\section{The sets $X^{SL_2}_w(b)(\bar{k})$}
\subsection{The vertex of departure} \mbox{}

\begin{Lm}\label{MinGal}
Let $\mathfrak{C}$ be a full connected subcomplex of $\build$. Let $D$ be an alcove 
in $\build$, which is not contained in $\mathfrak{C}$. Then there is a unique gallery $\Gamma_{D, \mathfrak{C}}$ with minimal length in $\build$, 
containing a vertex $P_D$ in $\mathfrak{C}$, whose first alcove is $D$. This vertex $P_D$ is uniquely determined by $D$.
\end{Lm}

\begin{proof}
The existence of such a gallery follows from the connectedness of $\build$. A gallery with given properties clearly contains exactly one 
vertex lying in $\mathfrak{C}$. Such a gallery is uniquely determined by this vertex since $\build$ is a tree.

To prove the uniqueness of $\Gamma_{D,\mathfrak{C}}$, assume $\Gamma_{D, \mathfrak{C}}^{\prime}$ is an other gallery with such properties stretched from $D$ to a vertex 
$P_D^{\prime}$ in $\mathfrak{C}$. By the connectedness of $\mathfrak{C}$, there would be the minimal gallery $\Gamma_{PP^\prime}$ 
stretched from $P_D$ to $P_D^{\prime}$, lying in $\mathfrak{C}$. Consider the composite gallery $(\Gamma_{D,\mathfrak{C}}, \Gamma_{PP^\prime})$. 
It is minimal by Lemma \ref{MinimalityLm}. In fact, $\Gamma_{D,\mathfrak{C}}$ and $\Gamma_{PP^{\prime}}$ are minimal, $\Gamma_{D,\mathfrak{C}}$ has no alcoves 
in $\mathfrak{C}$ and all alcoves of $\Gamma_{PP^{\prime}}$ lie in $\mathfrak{C}$: thus they have no common alcoves and $P_D$ is a last vertex of 
$\Gamma_{D,\mathfrak{C}}$ and a first vertex of $\Gamma_{PP^{\prime}}$.

Thus there are two minimal galleries, $\Gamma_{D, \mathfrak{C}}^{\prime}$ and $(\Gamma_{D,\mathfrak{C}}, \Gamma_{PP^\prime})$ stretched from $D$ to $P^{\prime}_D$. 
Since $\build$ is a tree, these galleries coincide. Since $\Gamma_{D, \mathfrak{C}}^{\prime}$ contains exactly one vertex in $\mathfrak{C}$, this implies
 $P_D = P^{\prime}_D$. Therefore $\Gamma_{D,\mathfrak{C}} = \Gamma_{D, \mathfrak{C}}^{\prime}$.
\end{proof}

\begin{Def} 
Let $\mathfrak{C}$ be a full connected subcomplex of $\build$. Let $D$ be an alcove in $\build$, which is not contained in $\mathfrak{C}$. 
Let $\Gamma_{D, \mathfrak{C}}$ and $P_D$ be as in Lemma \ref{MinGal}. We call $P_D$ the vertex of departure for $D$ from $\mathfrak{C}$. 
Set further $$d_{\mathfrak{C}}(D) = 1 + \ell(\Gamma_{D,\mathfrak{C}}).$$
\end{Def} 

Now we introduce some notations which we will need in the following.

\begin{Nott}\label{notb1} Let $\mathfrak{C}$ be a full connected subcomplex of $\build$.

\begin{itemize}
\item[(i)] For $m \in \{0, 1\}$ let $\mathfrak{C}^{(m)}$ denote the set of all vertices in $\mathfrak{C}$ with type $m$.

\item[(ii)] For a vertex $P$ in $\mathfrak{C}$ and $n > 0$, set 
$$D^n_{\mathfrak{C}}(P) = \left\{ D \colon \begin{array}{cl} D\text{ is an alcove in }\build \text{ with }P \text{ as }\cr \text{vertex of departure from } \mathfrak{C} \text{ and } 
d_\mathfrak{C}(D) = n\end{array} \right\}.$$ 
\end{itemize}
\end{Nott}

\begin{Lm}\label{miini} Let $\mathfrak{C}$ be a full connected subcomplex of $\build$ and let $D$ be an alcove in $\build$, not lying in $\mathfrak{C}$. Let further 
$\gamma$ be an automorphism of $\build$ (as a simplicial complex), and assume that $\gamma$ stabilizes the subcomplex $\mathfrak{C}$. Then 
$\Gamma_{\gamma D, \mathfrak{C}} = \gamma \Gamma_{D,\mathfrak{C}}$.
\end{Lm}
\begin{proof} As an automorphism of a simplicial complex $\gamma$ inherits adjacency, and thus takes galleries to galleries. Since $\gamma$ is invertible,
it also inherits minimality of galleries. Thus $\gamma \Gamma_{D,\mathfrak{C}}$ is minimal. The first alcove of $\gamma \Gamma_{D,\mathfrak{C}}$ 
is $\gamma D$. Further, $\gamma$ stabilizes the set of vertices lying in $\mathfrak{C}$ and thus also the set of vertices not lying in $\mathfrak{C}$. 
Hence $\gamma \Gamma_{D,\mathfrak{C}}$ contains exactly one vertex in $\mathfrak{C}$. This vertex is the image of a last vertex of 
$\Gamma_{D,\mathfrak{C}}$, and thus itself a last vertex of $\gamma \Gamma_{D,\mathfrak{C}}$. Thus $\gamma \Gamma_{D,\mathfrak{C}}$ has all properties 
of Lemma \ref{MinGal} which uniquely charcterize $\Gamma_{\gamma D,\mathfrak{C}}$. The Lemma follows.
\end{proof}


\subsection{The first case: the sets $X^{SL_2}_w(1)(\bar{k})$}\label{firstcasecount} \mbox{}

Let $D$ be an alcove in $\build$ which is not contained in $\ratbuild$. Apply Lemma \ref{MinGal} to the full connected subcomplex 
$\mathfrak{C} = \ratbuild$ and $D$. Thus there is a unique minimal gallery $\Gamma_{D,\ratbuild}$ stretched from $D$ to a uniquely determined vertex 
$P_D$ in $\ratbuild$ which is the vertex of departure for $D$ from $\ratbuild$. 

\begin{Lm}\label{sigmalm} We have: $\Gamma_{\sigma D, \ratbuild} = \sigma \Gamma_{D,\ratbuild}$.
\end{Lm}
\begin{proof} Follows from Lemma \ref{miini}, applied to the automorphism given by $\sigma$.  
\end{proof}

Now we construct the minimal gallery stretched from $D$ to $\sigma D$.
Observe that the vertices of departure for $D$ and for $\sigma D$ from $\ratbuild$ are equal (the one is the image under $\sigma$ of the other 
and both are in $\ratbuild$, thus stable under $\sigma$). Thus we get a gallery connecting $D$ and $\sigma D$, which consists of two parts: the first part 
is $\Gamma_{D, \ratbuild}$ and the second part is the gallery stretched from $P_D$ to $\sigma D$ i.e.  
$\Gamma_{\sigma D, \ratbuild}^{-1} = \sigma \Gamma_{D,\ratbuild}^{-1}$. 
We denote the composite gallery by $$\Gamma_D := (\Gamma_{D,\ratbuild}, \sigma \Gamma_{D,\ratbuild}^{-1}).$$ 

\begin{Lm}\label{CompGal} The gallery $\Gamma_D$ is minimal and $\ell(\Gamma_D) = 2d_{\ratbuild}(D) - 1 $.\end{Lm} 

\begin{proof} $\Gamma_{D,\ratbuild}$ and $\sigma \Gamma_{D,\ratbuild}^{-1}$ are minimal. $P_D$ is a last vertex of $\Gamma_{D,\ratbuild}$
 and a first vertex of $\sigma \Gamma_{D,\ratbuild}^{-1}$. Thus by Lemma \ref{MinimalityLm} we have to prove that they have no common 
alcoves. If an alcove $C$ would lie in  $\Gamma_{D,\ratbuild}$ and in $\sigma \Gamma_{D,\ratbuild}^{-1}$, then $\sigma C$ would too, and 
$\Gamma_{D,\ratbuild}$ would contain $C$ and $\sigma C$. Further $d_{\ratbuild}(C) = d_{\ratbuild}(\sigma C)$, since 
$\Gamma_{\sigma C, \ratbuild} = \sigma \Gamma_{C,\ratbuild}$ by Lemma \ref{sigmalm}. But all alcoves of $\Gamma_{D,\ratbuild}$ 
have different distances to $\ratbuild$ and hence $C = \sigma C$. This is equivalent to $C$ lying in $\ratbuild$ which leads to a contradiction, since 
$\Gamma_{D,\ratbuild}$ has no alcoves lying in $\ratbuild$. 

The length of $\Gamma_D$ is: 
\[\ell(\Gamma_D) = \ell(\Gamma_{D,\ratbuild}) + \ell(\sigma\Gamma_{D,\ratbuild}^{-1}) + 1 = (d_{\ratbuild}(D) - 1) + (d_{\ratbuild}(D) - 1) + 1 = 2d_{\ratbuild}(D) - 1. \qedhere \]
\end{proof}

The gallery $\Gamma_D$ is minimal and $\ell(\Gamma_D) > 0$, thus $\Gamma_D$ has a unique first vertex. 
If its type is $m \in \{0,1\}$, then $$\inv(D, \sigma D) = C_M^{(-1)^{m}(2d_{\ratbuild}(D) - 1)}.$$  

\noindent The set $X^{SL_2}_w(1)(\bar{k})$ is non-empty exactly for $w = 1$ and $w \in W_a$ with odd length (Proposition \ref{propb1} below or \cite{Re} 
Proposition 2.1.2). The following picture illustrates this via the identification of $W_a$ and $A_M$ (the fat alcoves are those 
for which $X^{SL_2}_w(1)(\bar{k})$  is non-empty). 
\\ 
\\ \includegraphics[bb = 80 85 0 0, scale = 1.5]{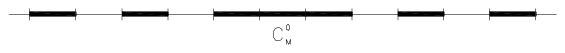}
\\
\begin{Prop}\label{propb1} The set $X^{SL_2}_w(1)(\bar{k})$ is non-empty if and only if $w = 1$ or $\ell(w)$ is odd.
Let now $w \in W_a$ such that $X^{SL_2}_w(1)(\bar{k}) \neq \emptyset$.
\begin{itemize}
\item[(i)] The set $X^{SL_2}_1(1)(\bar{k})$ is the set of all alcoves in $\ratbuild$.
\item[(ii)] If $\ell(w) = 2i-1, i > 0$ (i.e. $w C^0_M = C_{M}^{2i-1}$ or $w C^0_M = C_{M}^{-2i+1}$), let 
$$m = \left\{ \begin{array}{cl}0 & \text{ if i is odd and }wC^0_M = C_M^{-2i+1}\text{ or i is even and }wC^0_M = C_M^{2i-1}, 
\cr 1 &\text{ if i is odd and }wC^0_M = C_M^{2i-1}\text{ or i is even and }wC^0_M = C_M^{-2i+1}. \end{array} \right. $$ 
Then $$X^{SL_2}_w(1)(\bar{k}) = \coprod_{P \in \ratbuild^{(m)}}D^i_{\ratbuild}(P).$$ 
\end{itemize}
\end{Prop}
\begin{Lm}\label{Nonemptyb1} If $X^{SL_2}_w(1)(\bar{k}) \neq \emptyset$, then $w = 1$ or $\ell(w)$ is odd. \end{Lm}

\begin{proof} In fact, let $w \in W_a$ and $D \in X^{SL_2}_w(1)(\bar{k})$ be any alcove. If $D \in \ratbuild$, then $\inv(D, \sigma D) = C_M^0$ 
and thus $w=1$. Otherwise, the length of the minimal gallery $\Gamma_D$ constructed above is odd, and thus $\ell(w)$ is odd. \end{proof}

\begin{proof}[Proof of Proposition]
(i) follows from the fact that $\ratbuild = \build^{<\sigma>}$.
To prove (ii) we consider the case $i = \frac{\ell(w) + 1}{2}$ odd and $w C_M^0 = C_{M}^{2i-1}$. Thus $m = 1$. If an alcove $D$ lies in $X_w^{SL_2}(1)(\bar{k})$, then $D$ 
is not contained in $\ratbuild$ and the length of the gallery $\Gamma_D$ defined above must be equal to the length of the gallery stretched from $C_M^0$ 
to $C_M^{2i-1}$: $2d_{\ratbuild}(D)-1 = 2i-1$, so we must have $d_{\ratbuild}(D) = i$.  The type of the first vertex of $\Gamma_D$ must coincide 
with the type of the first vertex of the gallery stretched from $C_M^0$ to $C_M^{2i-1}$, i.e. $0$, so the type of the vertex of departure for $D$ 
from $\ratbuild$ must be $1$ (by parity of $i$). So $X^{SL_2}_w(1)(\bar{k}) \subseteq \bigcup_{P \in \ratbuild^{(1)}} D_{\ratbuild}^i(P)$. 

If conversely, $D \in D^i_{\ratbuild}(P)$ with $P \in \ratbuild^{(1)}$, then the relative position of $D$ and $\sigma D$ will be $C_M^{2i-1}$ 
(this is clear by construction of $\Gamma_D$). Hence $X_w^{SL_2}(1)(\bar{k}) = \bigcup_{P \in \ratbuild^{(1)}} D_{\ratbuild}^i(P)$ and thus is non-empty. 
The sets $D^i_{\ratbuild}(P)$ are disjoint for different $P \in \ratbuild^{(1)}$, by uniqueness of the vertex of departure. The other three cases 
($i$ is odd, $w C_M^0 = C_{M}^{-2i+1}$ and the two cases where $i$ is even) have similar proofs.
\end{proof}


\subsection{The second case: the sets $X^{SL_2}_w(b)(\bar{k})$ for diagonal $b \neq 1$}\label{seccase} \mbox{} 

Let now $b = \begin{pmatrix} 1 & 0 \\ 0 & t^\alpha \end{pmatrix}$ with $\alpha > 0$. First of all, $b$ acts on $A_M$ by translation by $\alpha$ alcoves 
to the right: $b$ sends the alcove $C_M^i$ to the alcove $C_M^{i + \alpha}$. 
The distance from the main apartment will play the analogous role, which in the previous case was played by the distance to $\ratbuild$. 
Let $D$ be an alcove in $\build$ which is not contained in $A_M$. Apply Lemma \ref{MinGal} to the full connected subcomplex $\mathfrak{C} = A_M$ 
of $\build$ and $D$. Thus there is a unique minimal gallery $\Gamma_{D,A_M}$ stretched from $D$ to a uniquely determined vertex $P_D$ in $A_M$, 
which is the vertex of departure for $D$ from $A_M$. Then $P_D$ is a last vertex of $\Gamma_{D,A_M}$.
\begin{Lm}\label{bsigmalm} We have: $\Gamma_{b \sigma D,A_M} = b \sigma \Gamma_{D,A_M}$.
\end{Lm}
\begin{proof} Follows from Lemma \ref{miini}, applied to the automorphism given by $b \sigma$.  
\end{proof}

Like in the previous case, we want to construct a gallery stretched from $D$ to $b \sigma D$. Let $\Gamma_{D,tr}$ be the gallery stretched from
$P_D$ to $b \sigma P_D$. Thus $\Gamma_{D,tr}$ has $P_D$ as a first vertex and $b \sigma P_D$ as a last vertex. 
All alcoves of $\Gamma_{D, tr}$ lie in $\maap$. The length of $\Gamma_{D,tr}$ is $\alpha - 1$. 
From Lemma \ref{bsigmalm} follows: $\Gamma_{b \sigma D,A_M}^{-1} = b \sigma \Gamma_{D,A_M}^{-1}$. This gallery has 
$b \sigma P_D$ as a first vertex and $b \sigma D$ as the last alcove. We set $$\Gamma_D := (\Gamma_{D,A_M},\Gamma_{D,tr}, b \sigma \Gamma_{D,A_M}^{-1}).$$ 
It has $D$ as the first and $b \sigma D$ as the last alcove. 

\begin{Lm} \label{CompGalDiag} The gallery $\Gamma_D$ is minimal and $\ell(\Gamma_D) = 2d_{\maap}(D) + \alpha - 1 $.\end{Lm}

\begin{proof} 
The galleries $\Gamma_{D,A_M}$, $\Gamma_{D,tr}$ and $b \sigma \Gamma_{D,A_M}^{-1}$ are minimal. We use Lemma \ref{MinimalityLm}(ii). At first, 
we prove that $\Gamma_{D,A_M}$, $\Gamma_{D,tr}$ and $b \sigma \Gamma_{D,A_M}^{-1}$ pairwise have no common alcoves. In fact, $\Gamma_{D,A_M}$ and 
$b\sigma \Gamma_{D,A_M}^{-1} = \Gamma_{b\sigma D,A_M}^{-1}$ have no alcoves lying in $A_M$. Hence $\Gamma_{D,A_M}$ and $b \sigma \Gamma_{D,A_M}^{-1}$ 
have no common alcoves with $\Gamma_{D,tr}$ (whose alcoves all lie in $A_M$). All alcoves in $\Gamma_{D,A_M}$ have the same vertex $P_D$ of departure 
from $A_M$.  The vertex $b\sigma P_D$  is the vertex of departure from $A_M$ for every alcove in 
$\Gamma_{b \sigma D,A_M}^{-1} = b \sigma \Gamma_{D,A_M}^{-1}$. But $P_D \neq b \sigma P_D$. By the uniqueness of the vertex of departure, 
$\Gamma_{D,A_M}$ and $b \sigma \Gamma_{D,A_M}^{-1}$ contain no common alcoves.

Now the condition (d) of Lemma \ref{MinimalityLm}(ii) for $\Gamma_{D,A_M}$, $\Gamma_{D,tr}$ is clear. 
Thus the composite gallery $(\Gamma_{D,A_M}, \Gamma_{D,tr})$ is minimal.

Now we have to verify the condition (d) of Lemma \ref{MinimalityLm}(ii) for $(\Gamma_{D,A_M}, \Gamma_{D,tr})$ and $b \sigma \Gamma_{D,A_M}^{-1}$. The vertex 
$b \sigma P_D$ is a first vertex of $b \sigma \Gamma_{D,A_M}^{-1}$. Further, $b \sigma P_D$ is a last vertex of $\Gamma_{D,tr}$. The only vertex contained 
in $\Gamma_{D,A_M}$ and in $\Gamma_{D,tr}$ is $P_D$ (it is the unique vertex of $\Gamma_{D,A_M}$ contained in $A_M$). But $P_D \neq b \sigma P_D$, 
and thus $b \sigma P_D$ is also a last vertex of the composite gallery $(\Gamma_{D,A_M}, \Gamma_{D,tr})$. 
Thus by Lemma \ref{MinimalityLm}(ii) the composite gallery $\Gamma_D = (\Gamma_{D,A_M}, \Gamma_{D,tr}, b \sigma \Gamma_{D,A_M}^{-1})$ is minimal.

The length of $\Gamma_D$ is: 
\begin{eqnarray*} \ell(\Gamma_D) &=& \ell(\Gamma_{D,A_M}) + \ell(\Gamma_{D, tr}) + \ell(b \sigma \Gamma_{D,A_M}^{-1}) + 2 \\ &=& 
2\ell(\Gamma_{D,A_M}) + \ell(\Gamma_{D, tr}) + 2 = 2(d_{A_M}(D) - 1) + (\alpha - 1) + 2 \\ &=& 
2d_{A_M}(D) + \alpha - 1. \qedhere \end{eqnarray*}
\end{proof}

The gallery $\Gamma_D$ is minimal and $\ell(\Gamma_D) > 0$, hence it has a unique first vertex. If its type is $m \in \{0,1\}$, then 
$$\inv(D, b\sigma D) = C_M^{(-1)^m(\alpha + 2d_{A_M}(D) - 1)}.$$  

\noindent Now, $X^{SL_2}_w(b)(\bar{k})$ is non-empty if and only if $w$ has length $\alpha$ or $\alpha + 2i-1$ for some $i > 0$ (see Proposition \ref{propbd} below or 
\cite{Re} 2.1.4 and 2.2). The following picture illustrates this in the case $\alpha = 4$ (the fat alcoves are those for which $X^{SL_2}_w(b)(\bar{k})$  
is non-empty).
\\
\\ \includegraphics[bb = 100 190 0 0]{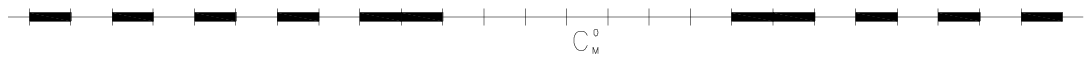}
\\
\\

\begin{Prop}\label{propbd} Let $b = \begin{pmatrix} 1 & 0 \\ 0 & t^{\alpha} \end{pmatrix}$ with $\alpha > 0$, and $w \in W_a$. Then $X^{SL_2}_w(b)(\bar{k})$ is non-empty if and only 
if $\ell(w) = \alpha$ or $\ell(w) = \alpha + 2i-1$ with $i>0$.

Let now $w \in W_a$ such that $X^{SL_2}_w(b)(\bar{k}) \neq \emptyset$.
\begin{itemize}
\item[(i)] If $wC^0_M = C_M^{\alpha}$, then $X^{SL_2}_w(b)(\bar{k}) = \coprod\limits_{j \in \mathbb{Z}} \{C_M^{2j}\}$. 
\item[(ii)] If $wC^0_M = C_M^{-\alpha}$, then $X^{SL_2}_w(b)(\bar{k}) = \coprod\limits_{j \in \mathbb{Z}} \{C_M^{2j+1}\} $.
\item[(iii)] If $\ell(w) = \alpha + (2i-1)$ for $i > 0$ (i.e. $wC_M^0 = C_M^{\alpha + (2i-1)}$ or $wC_M^0 = C_M^{ - \alpha - (2i-1)}$), let 
$$m = \left\{ \begin{array}{cl} 0 & \text{if }i \text{ is odd and } wC_M^0 = C_M^{ - \alpha - (2i-1)} \text{ or if }i \text{ is even and } wC_M^0 = 
C_M^{\alpha + (2i-1)}, \cr 1 &  \text{if }i \text{ is even and } wC_M^0 = C_M^{ - \alpha - (2i-1)} \text{ or if }i \text{ is odd and } wC_M^0 = C_M^{\alpha 
+ (2i-1)}. \end{array} \right.$$ 
Then $$X^{SL_2}_w(b)(\bar{k}) = \coprod_{P \in A_M^{(m)}} D^i_{A_M}(P).$$
\end{itemize}
\end{Prop}

\begin{Lm}\label{NonemptybDiag} If $X^{SL_2}_w(b)(\bar{k}) \neq \emptyset$, then $\ell(w) = \alpha$ or $\ell(w) = \alpha + 2i - 1$ with $i>0$. \end{Lm}  

\begin{proof} In fact, let $w \in W_a$ and $D \in X^{SL_2}_w(b)(\bar{k})$ be any alcove. If $D = C^{2j}_M$ resp. $D = C^{2j+1}_M$ lies in $A_M$, then 
$\inv(D, b\sigma D) = C^{\alpha}_M$ resp. $C^{-\alpha}_M$ and $\ell(w) = \alpha$. Otherwise, the length of the minimal gallery $\Gamma_D$ 
constructed above is $\alpha + 2i - 1$ for some $i > 0$. Thus $\ell(w) = \alpha + 2i - 1$. \end{proof}

\begin{proof}[Proof of Proposition] To prove (i) and (ii), we notice that if $\ell(w) = \alpha$, then $X^{SL_2}_w(b)(\bar{k})$ is contained in $\maap$: 
in fact if $D$ is not in $A_M$, then the length of the gallery $\Gamma_D$ constructed above is $\alpha + 2i - 1$ for some $i \in \mathbb{Z}$, which differs 
by an odd integer from $\alpha$. Now $b$ acts on the alcoves in $A_M$ by shifting 
by $\alpha$ alcoves to the right. Thus if $w C^0_M = C^{\alpha}_M$, then $X^{SL_2}_w(b)(\bar{k}) = \{C_M^{2j} \colon j \in \mathbb{Z}\} $. If $w C^0_M = C^{-\alpha}_M$, then 
$X^{SL_2}_w(b)(\bar{k}) = \{C_M^{2j-1} \colon j \in \mathbb{Z}\}$.

Now we prove (iii) for $i > 0$ odd and $wC^0_M = C_M^{\alpha + 2i-1}$ (the other cases can be proven similarly). In this case $m=1$. If $D$ lies in the 
set $X^{SL_2}_w(b)(\bar{k})$, then $D$ does not lie in $A_M$, by the above considerations, and the length of the gallery $\Gamma_D$ constructed above must 
be equal to $\ell(w) = \alpha + 2i - 1$. Therefore, we must have: $\alpha + 2d_{\maap}(D) = \alpha + 2i$, which implies $d_{\maap}(D) = i$. The first vertex 
of $\Gamma_D$ must have the same type as the first vertex of the gallery stretched from $C_M^0$ to $C_M^{\alpha + (2i-1)}$, i.e. $0$. So, by parity of $i$, the vertex of 
departure for $D$ must have type 1. Thus, $\bigcup_{P \in A_M^{(1)}}D^i_{A_M}(P) \supseteq X^{SL_2}_w(b)(\bar{k})$.

Conversely, if $D \in D^i_{A_M}(P)$ for some $P \in A_M^{(1)}$, then the corresponding gallery $\Gamma_D$ has the length 
$\alpha  + (2i-1)$ and the type of its first vertex is 0, thus it can be folded into the gallery stretched from $C_M^0$ to $C_M^{\alpha + (2i-1)}$. Hence 
$X^{SL_2}_w(b)(\bar{k}) = \bigcup_{P \in A_M^{(1)}}D^i_{A_M}(P)$. This union is disjoint, since the vertex of departure is uniquely determined. Thus (iii) follows.
Now, the first part of the Proposition follows from Lemma \ref{NonemptybDiag} and (i)-(iii). 
\end{proof} 


\subsection {The third case: the sets $X^{SL_2}_w(b_1)(\bar{k})$} \mbox{}

Recall that $b_1 = \begin{pmatrix} 0 & 1 \\ t & 0 \end{pmatrix}$. The action of $b_1$ on the main apartment is given by reflection about the midpoint of 
$C_M^0$, thus it sends $C_M^{\alpha}$ to $C_M^{-\alpha}$. The 
vertex of $C_M^0$ represented by $\mathfrak{o} \oplus \mathfrak{o}$ goes to the vertex represented by $\mathfrak{o} \oplus t\mathfrak{o}$, and conversely.

In this case, $C_M^0$ plays the same role as $\ratbuild$ and $A_M$ in the previous cases. Let $D \neq C_M^0$ be an alcove in $\build$. By Lemma 
\ref{MinGal}, applied  to the full connected subcomplex $\overline{C^0_M}$ of $\build$ containing only the alcove $C_M^0$ and its vertices, there is a 
unique minimal gallery $\Gamma_{D, \overline{C_M^0}}$ in $\build$ stretched from $D$ to its vertex $P_D$ of departure from $\overline{C_M^0}$. 

\begin{Lm}
We have: $b_1\sigma\Gamma_{D,\overline{C^0_M}} = \Gamma_{b_1\sigma D,\overline{C^0_M}}$.
\end{Lm}
\begin{proof} Follows from Lemma \ref{miini}, applied to the automorphism given by $b_1 \sigma$.
\end{proof}

Like above, we construct the minimal gallery $\Gamma_D$ stretched from $D$ to $b_1 \sigma D$.
Let $\Gamma_{D,tr}$ be the gallery consisting of the single alcove $C_M^0$. Now, $P_D$ is a last vertex of $\Gamma_{D,\overline{C^0_M}}$ and a first vertex 
of $\Gamma_{D,tr}$. Further, $b_1\sigma P_D$ is a first vertex of $b_1 \sigma \Gamma_{D,\overline{C^0_M}}^{-1} = \Gamma_{b_1\sigma D,\overline{C^0_M}}^{-1}$ and a last vertex of $\Gamma_{D, tr}$. 
Hence $$\Gamma_D := (\Gamma_{D,\overline{C^0_M}}, \Gamma_{D,tr}, b_1\sigma\Gamma_{D,\overline{C^0_M}}^{-1})$$ 
is a gallery. Its first alcove is $D$ and its last alcove is $b_1\sigma D$. 

\begin{Lm}\label{dCM} The gallery $\Gamma_D$ is minimal and $\ell(\Gamma_D) = 2d_{\overline{C^0_M}}(D)$. \end{Lm} 

\begin{proof} The galleries $\Gamma_{D,\overline{C^0_M}}$, $\Gamma_{D,tr}$ and $b_1\sigma\Gamma_{D,\overline{C^0_M}}^{-1}$ are minimal. We use Lemma \ref{MinimalityLm}(ii).
At first, we prove that $\Gamma_{D,\overline{C^0_M}}$, $\Gamma_{D,tr}$ and $b_1 \sigma \Gamma_{D,\overline{C^0_M}}^{-1}$ have pairwise no common alcoves. 
$\Gamma_{D,tr}$ clearly has no common alcoves with $\Gamma_{D,\overline{C^0_M}}$ and $b_1\sigma\Gamma_{D,\overline{C^0_M}}^{-1}$. Further, every 
alcove of $\Gamma_{D,\overline{C^0_M}}$ has $P_D$ as its vertex of departure from $\overline{C^0_M}$ and every alcove 
of $b_1\sigma\Gamma_{D,\overline{C^0_M}}^{-1} = \Gamma_{b_1\sigma D,\overline{C^0_M}}^{-1}$ has $b_1\sigma P_D \neq P_D$ as its vertex of departure from $\overline{C^0_M}$. 
The vertex of departure from $\overline{C^0_M}$ is uniquely determined for every alcove, thus $\Gamma_{D,\overline{C^0_M}}$ and $b_1\sigma\Gamma_{D,\overline{C^0_M}}^{-1}$
have no common alcoves. 

The condition (d) of Lemma \ref{MinimalityLm}(ii) for $\Gamma_{D,\overline{C^0_M}}$, $\Gamma_{D,tr}$ is clear.  Thus the composite gallery 
$(\Gamma_{D,\overline{C^0_M}}, \Gamma_{D,tr})$ is minimal.

Now we have to verify the condition (d) of Lemma \ref{MinimalityLm}(ii) for $(\Gamma_{D,\overline{C^0_M}}, \Gamma_{D,tr})$ and $b_1 \sigma \Gamma_{D,\overline{C^0_M}}^{-1}$. 
The vertex $b_1 \sigma P_D$ is a first vertex of $b_1 \sigma \Gamma_{D,\overline{C^0_M}}^{-1}$. Further $b_1 \sigma P_D$ is a last vertex of $\Gamma_{D,tr}$. 
The only vertex contained in $\Gamma_{D,\overline{C^0_M}}$ and in $\Gamma_{D,tr}$ is $P_D$. 
But $P_D \neq b_1 \sigma P_D$. Thus $b_1 \sigma P_D$ will also be a last vertex of the composite gallery $(\Gamma_{D,\overline{C^0_M}}, \Gamma_{D,tr})$. Thus by Lemma 
\ref{MinimalityLm}(ii) the composite gallery $\Gamma_D = (\Gamma_{D,\overline{C^0_M}}, \Gamma_{D,tr}, b_1 \sigma \Gamma_{D,\overline{C^0_M}}^{-1})$ 
is minimal.

The  length of $\Gamma_D$ is: 
\[\ell(\Gamma_D) = \ell(\Gamma_{D,\overline{C^0_M}}) + \ell(\Gamma_{D,tr}) + \ell(b_1\sigma\Gamma_{D,\overline{C^0_M}}^{-1}) + 2 = 2\ell(\Gamma_{D,\overline{C^0_M}}) + 0 + 2 = 2d_{\overline{C^0_M}}(D).\qedhere \]
\end{proof}

The gallery $\Gamma_D$ is minimal and $\ell(\Gamma_D) > 0$, thus it has a unique first vertex. If its type is $m \in \{0,1\}$, then 
$$\inv(D, b_1\sigma D) = C_M^{(-1)^m 2d_{\overline{C^0_M}}(D)}.$$ 

\noindent Now, $X^{SL_2}_w(b_1)(\bar{k})$ is non-empty if and only if $w$ has even length (Proposition \ref{propbn} below or \cite{Re} 2.2). 
The following picture illustrates this (the fat alcoves are those for which $X^{SL_2}_w(b_1)(\bar{k})$  is non-empty).
\\
\\
\\ \includegraphics[bb = 1 225 0 0 scale = 0.6]{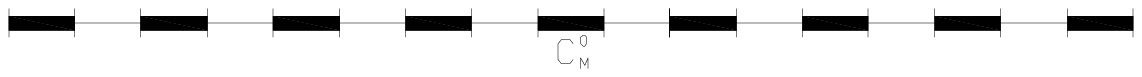}
\\
For $m\in \{0,1\}$, let $P_m$ denote the vertex of $\build$ represented by $\mathfrak{o} \oplus t^m\mathfrak{o}$. 
\begin{Prop}\label{propbn} Let $w \in W_a$. Then $X^{SL_2}_w(b_1)(\bar{k})$ is non-empty if and only if $\ell(w)$ is even.

Let now $w \in W_a$ such that $X^{SL_2}_w(b_1)(\bar{k}) \neq \emptyset$.
\begin{itemize}
\item[(i)] We have: $X^{SL_2}_1(b_1)(\bar{k}) = \{C_M^{0}\}$.
\item[(ii)] If $\ell(w) = 2i$ for $i > 0$ (i.e. $wC_M^0 = C_M^{2i}$ or $wC_M^0 = C_M^{-2i}$), let 
$$m = \left\{ \begin{array}{cl} 0 & \text{if }i \text{ is odd and } wC_M^0 = C_M^{-2i} \text{ or if }i \text{ is even and } wC_M^0 = 
C_M^{2i}, \cr 1 &  \text{if }i \text{ is odd and } wC_M^0 = C_M^{2i} \text{ or if }i \text{ is even and } wC_M^0 = C_M^{-2i}. \end{array} \right.$$ 

Then $$X^{SL_2}_w(b_1)(\bar{k}) = D^i_{\overline{C^0_M}}(P_m).$$

\end{itemize}
\end{Prop}
\begin{Lm}\label{NonemptybLast} If $X^{SL_2}_w(b_1)(\bar{k}) \neq \emptyset$, then $\ell(w)$ is even. \end{Lm}  

\begin{proof} In fact, let $w \in W_a$ and $D \in X^{SL_2}_w(b_1)(\bar{k})$ be any alcove. If $D = C^0_M$, then $w = 1$ has the length $0$. 
Otherwise, the length of the minimal gallery $\Gamma_D$ constructed above is even and thus $\ell(w) = \ell(\Gamma_D)$ is even.
\end{proof}

\begin{proof}[Proof of Proposition] For (i) observe that for any alcove $D \neq C^0_M$ the gallery $\Gamma_D$ never has the length $0$, 
thus $X^{SL_2}_1(b_1)(\bar{k}) \subseteq \{C^0_M\}$. The inverse implication is clear. 

Now we prove (ii) for $i > 0$ odd and $wC_M^0 = C_M^{2i}$ (the other cases can be proven similarly). In this case $m=1$. If $D$ lies in the set 
$X^{SL_2}_w(b_1)(\bar{k})$, 
then we have: $2d_{\overline{C_M^0}}(D) = \ell(\Gamma_D) = \ell(w) = 2i$, which 
implies $d_{\overline{C_M^0}}(D) = i$. The first vertex of $\Gamma_D$ must 
have the same type as the first vertex of the gallery stretched from $C_M^0$ to $C_M^{2i}$, i.e. $0$. So by parity of $i$, the vertex of departure for $D$ 
 from $\overline{C_M^0}$ must have type $1$. Thus $D^i_{\overline{C^0_M}}(P_1) \supseteq X^{SL_2}_w(b_1)(\bar{k})$.

Conversely, if $D \in D^i_{\overline{C^0_M}}(P_1)$, then the corresponding gallery $\Gamma_D$ has 
length $2i$ and the type of its first vertex is $0$. Thus $\Gamma_D$ can be folded into the gallery stretched from  
$C_M^0$ to $C_M^{2i}$. So $X^{SL_2}_w(b_1)(\bar{k}) = D^i_{\overline{C^0_M}}(P_1)$. The first part of the Proposition follows 
from Lemma \ref{NonemptybLast} and (i), (ii). 
\end{proof}

\section{The variety structure on $X_w(b)$}  \mbox{}

In this section we choose the functorial point of view on a scheme, and work mainly with the set of its $\bar{k}$-valued points. 
In particular, a locally closed reduced sub-Ind-scheme of the affine flag manifold is uniquely determined by the set of its $\bar{k}$-valued points. 
In the last section we have determined the set-theoretical structure of $X^{SL_2}_w(b)(\bar{k})$. 
Now we determine the scheme-structure on $X_w(b)$ and on $X^{SL_2}_w(b)$. From Proposition \ref{LmGLSL} we have 
$$X_w(b) \neq \emptyset \quad \Leftrightarrow \quad X^{SL_2}_w(b) \neq \emptyset \quad \Leftrightarrow \quad X^{SL_2}_w(b)(\bar{k}) \neq \emptyset.$$ 

\noindent In all cases $X_w^{SL_2}(b)(\bar{k})$ were (disjoint unions of) sets of very similar types. At first we prove 
a general fact which shows that all these sets are locally closed subsets of $X^{SL_2}(\bar{k})$. 

\subsection{The crucial result} 

\begin{Lm}\label{LemmaCycl} Let $l \geq 0$ and $P \neq Q$ two vertices in $\build$, represented by lattices $\mathfrak{L}_1, \mathfrak{L}_2$, respectively. 
Assume that $\mathfrak{L}_1 \supseteq \mathfrak{L}_2$ and $\dimm_{\bar{k}} \mathfrak{L}_1/\mathfrak{L}_2 = l+1$. Then the following are equivalent:
\begin{itemize}
\item[(i)] The length of the minimal gallery in $\build$ stretched from $P$ to $Q$ is $l$.

\item[(ii)] The $\bar{k}[t]$-module $ \mathfrak{L}_1/\mathfrak{L}_2$ is cyclic.
\end{itemize}
The same holds for $\ratbuild$ instead of $\build$.
\end{Lm}

\begin{proof} Choose by the elementary divisor theorem a $\bar{k}[[t]]$-basis $\{v_1, v_2\}$ of $\mathfrak{L}_1$ such that \\ 
$\{t^{a_1}v_1, t^{a_2}v_2\}$ is a $\bar{k}[[t]]$-basis of $\mathfrak{L}_2$. Then $0 \leq a_1, a_2 \leq l+1$ and $a_1 + a_2 = l+1$. 
We can assume $a_1 \leq a_2$. The $\bar{k}[t]$-module $ \mathfrak{L}_1/\mathfrak{L}_2$ is cyclic if and only if $a_1 = 0$. 
Consider the following lattice chain:

$$\langle t^{a_1}v_1, t^{a_1}v_2\rangle_{\bar{k}[[t]]} \supsetneq \langle t^{a_1}v_1, t^{a_1+1}v_2\rangle_{\bar{k}[[t]]} \supsetneq ... 
\supsetneq \langle t^{a_1}v_1, t^{a_2}v_2\rangle_{\bar{k}[[t]]}, $$

\noindent which represents a minimal gallery between the vertex $P$, represented by $\langle t^{a_1}v_1, t^{a_1}v_2\rangle_{\bar{k}[[t]]}$ and $Q$, 
represented by $\langle t^{a_1}v_1, t^{a_2}v_2\rangle_{\bar{k}[[t]]}$. It is minimal by Lemma \ref{MinimalityLm}(i) (verification of the two conditions 
is straightforward). The length of this gallery is $a_2 - a_1 - 1$. This number is equal $l$ if and only if $a_1 = 0$, or equivalently 
if and only if $\mathfrak{L}_1/\mathfrak{L}_2$ is cyclic. The proof for $\ratbuild$ instead of $\build$ is the same.
\end{proof}

It is a well known fact that the alcoves around a vertex $P$ in $\build$ form a closed subset of $X^{SL_2}(\bar{k})$. There is a unique reduced closed subscheme 
of $X^{SL_2}$ such that its $\bar{k}$-valued points are exactly the alcoves around $P$. This subscheme is defined over $k$ and isomorphic 
to $\mathbb{P}^1_k$ as a $k$-scheme.

\begin{Def}\label{DefFunk} Let $P$ be a vertex in $\ratbuild$. Let $V$ be a subset of $\mathbb{P}^1_k(k)$. Thus $V$ corresponds to a finite set of 
alcoves in $\ratbuild$, having distance 0 to $P$. Let $l \geq 0$. We set 
$$ \mathfrak{F}_{P,V,l}(\bar{k}) = \left\{ D \colon \begin{array}{cl} D \text{ is an alcove in } \build \text{ which has distance } l 
\text{ to } P, \cr \text{ and the minimal gallery stretched from } P \cr \text{to } D \text{ does not pass through } V\end{array}\right\},$$
and let $\mathfrak{F}_{P,V,l}(k)$ be the intersection of $\mathfrak{F}_{P,V,l}(\bar{k})$ with the set of alcoves lying in $\ratbuild$.
\end{Def}

\begin{Prop} Let $P, V, l$ be as in Definition \ref{DefFunk}. Then $\mathfrak{F}_{P,V,l}(\bar{k})$ is a locally closed subset of 
$X^{SL_2}(\bar{k})$. Denote by $\mathfrak{F}_{\bar{k}}$ the corresponding induced reduced sub-Ind-scheme of $X^{SL_2}_{\bar{k}}$. 
The Ind-scheme $\mathfrak{F}_{\bar{k}}$ is a scheme, it is defined over $k$ by a scheme $\mathfrak{F}$, and there is an isomorphism of $k$-schemes: 
$$\mathfrak{F} \cong \mathbb{A}_k^{l} \times (\mathbb{P}^1_k - V).$$ \end{Prop}

\begin{proof} By homogeneity under the action of $GL_2(F)$, we can assume that $P$ is represented by 
$\mathfrak{L}_0 = \StGitF$. Let us first prove the Proposition for $V$ consisting of only one element. In this case we can assume that this element is 
$C_M^0$ and its vertices are $P$ and $Q$, where $Q$ is represented by $\mathfrak{L}_1 = \tStGitF$.

Before going on with the proof, we make the following
\begin{Rem} {\rm For $V$ and $P$ as described above, the set $\mathfrak{F}_{P,V,l}(\bar{k})$ is nothing other than the open Schubert 
cell for the element $v \in W_a$ with $vC^0_M = C^{-(l+1)}_M$. There is another proof that the Schubert cell for $v \in W_a$ is isomorphic
 to the affine space of dimension $\ell(v)$. This proof works similarly in the affine case, as in the finite case (which is given in \cite{Bo}, 14.12, the 
first theorem). Nevertheless, we will give another proof below and then compare the two proofs. }
\end{Rem}

\noindent \textbf{The sketch of the proof.} Consider $N = \mathfrak{L}_0/t^{l+1}\mathfrak{L}_0$ as a $k$-vector space carrying a nilpotent action of $t$. 
Let $\mathfrak{Flag}^t_{l+1,l+2}(N) \subseteq \mathfrak{Flag}_{l+1,l+2}(N)$ be the closed subscheme of the partial flag manifold of $N$,
 parametrizing the chains of $t$-stable subspaces of $N$ of dimensions $l+1, l+2$. Let also $$S := \mathfrak{Flag}^t_{l+1, l+2}(N)_{\redd}.$$ 
\noindent In particular $S$ has the same $k$-valued points as $\mathfrak{Flag}^t_{l+1,l+2}(N)$. We have the following two closed immersions 
(the vertical one is given by the inverse of (\ref{affflmfdfrml})):

\begin{equation}\nonumber \xymatrix{ X^{SL_2} & \\ S \ar[u] \ar[r] & \mathfrak{Flag}_{l+1,l+2}(N)} \end{equation}
In the proof we consider the following commutative diagram of inclusions:

\begin{equation} \nonumber \xymatrix{\mathfrak{F}_{P,V,l}(k) \ar[r] \ar[dr] \ar[ddr]& X^{SL_2}(k) & \\ 
& S(k)  \ar[u] \ar[r] & \mathfrak{Flag}_{l+1, l+2}(N)(k) \\ & Y(k) \ar[ur] & }  \end{equation}

\noindent where $Y$ will later be defined as some open subscheme of $\mathfrak{Flag}_{l+1,l+2}(N)$ which is isomorphic to the affine space over $k$ of 
appropriate dimension. We will show that the two diagonal arrows from $\mathfrak{F}_{P,V,l}(k)$ exist and that there is a unique closed reduced subscheme 
of $Y$, defined by vanishing of linear polynomials whose $k$-rational points are $\mathfrak{F}_{P,V,l}(k)$.

\noindent \textbf{First step.} We prove that $\mathfrak{F}_{P,V,l}(k) \rightarrow X^{SL_2}(k)$ factorizes through $S(k) \rightarrow X^{SL_2}(k)$. 

Observe that every alcove with distance $l$ to $P$ has a unique representing lattice chain 
$\mathfrak{L}^{\prime} \supsetneq \mathfrak{L} \supsetneq t\mathfrak{L}^{\prime}$ such that $\mathfrak{L}, \mathfrak{L}^{\prime} \subseteq \mathfrak{L}_0$, the
vertex represented by $\mathfrak{L}$ has greater distance to $P$, and dim$_k \mathfrak{L}_0/ \mathfrak{L} = l+1$ 
($\mathfrak{L}$ is uniquely determined for reasons of dimension and then the uniqueness of $\mathfrak{L}^{\prime}$ follows from the uniqueness of minimal 
galleries). After these choices $\mathfrak{L}_0/ \mathfrak{L}$ is a cyclic $k[t]$-module by Lemma \ref{LemmaCycl}. 
We have $\mathfrak{L}, \mathfrak{L}^{\prime} \supseteq t^{l+1}\mathfrak{L}_0$. 

The projections $\mathfrak{L} \rightarrow L := \mathfrak{L}/(t^{l+1}\mathfrak{L}_0)$, $\mathfrak{L}^{\prime} \rightarrow L^{\prime} := 
\mathfrak{L}^{\prime}/(t^{l+1}\mathfrak{L}_0)$ give the embedding $\mathfrak{F}_{P,V,l}(k) \subset S(k)$ (which is injective, since the maps 
$\mathfrak{F}_{P,V,l}(k) \rightarrow X^{SL_2}(k)$, $S(k) \rightarrow X^{SL_2}$ are). Note that $N/L \cong \mathfrak{L}_0/\mathfrak{L}$ and 
$N/L^{\prime} \cong \mathfrak{L}_0/\mathfrak{L}^{\prime}$ are $k[t]$-cyclic too. Thus:  
\begin{equation}\label{FPVl_1}\mathfrak{F}_{P,V,l}(k) \subseteq \left\{ (L \subsetneq L^{\prime}) \in S(k) \colon N/L \text{ and } N/L^{\prime} 
\text{ are cyclic as } k[t]\text{-modules} \right\},\end{equation} 

\noindent \textbf{Second step.} In this step we reformulate the condition that the minimal gallery stretched from the vertex, represented by $\mathfrak{L}$, to $P$ 
does not pass through $Q$. 

Let $\{e_1, e_2\}$ be the image of the canonical $\mathfrak{o}_F$-basis of $\mathfrak{L}_0$ under the projection $\mathfrak{L}_0 \twoheadrightarrow N$. Then $\{ 
e_1,te_1, ..., t^{l}e_1, e_2, te_2, ..., t^{l}e_2\}$ is a $k$-basis of $N$. Let $N_1 = \mathfrak{L}_1 / t^{l+1}\mathfrak{L}_0$. Then $N_1$ is a $(2l+1)$-dimensional 
subspace of $N$ and $\{ e_1,te_1, ..., t^{l}e_1, te_2, ..., t^{l}e_2\}$ is its basis. 

\begin{Lm} \label{LN_1}
Let $\mathfrak{L}^{\prime} \supsetneq \mathfrak{L} \supsetneq t\mathfrak{L}^{\prime}$ be a representative of an alcove $D$, having distance $l$ to $P$,
chosen as in step one. Let $L:= \mathfrak{L}/t^{l+1}\mathfrak{L}_0$. The following are equivalent:
\begin{itemize}
\item[(i)] The minimal gallery from $P$ to the vertex represented by $\mathfrak{L}$ does not pass through $Q$.
\item[(ii)] $L \not\subseteq N_1$.
\end{itemize}
\end{Lm}

\begin{proof}[Proof of Lemma] We have $t^{-1}\mathfrak{L}_1 \supsetneq \mathfrak{L}_0 \supsetneq \mathfrak{L}$ and 
$\dimm_k((t^{-1}\mathfrak{L}_1)/ \mathfrak{L}_0) = 1, \dimm_k(\mathfrak{L}_0)/ \mathfrak{L}) = l+1$. Hence: 
$\dimm_k((t^{-1}\mathfrak{L}_1)/ \mathfrak{L}) = l+2$. Further we have: $\mathfrak{L}_1 \supseteq t\mathfrak{L}$. By the elementary divisor theorem, 
choose a $k[[t]]$-basis $\{ v_1, v_2\}$ of $\mathfrak{L}_1$ such that $\{t^{a_1}v_1, t^{a_2}v_2\}$ is a $k[[t]]$-basis of $t\mathfrak{L}$. 

Now, (i) is equivalent to the statement that the minimal gallery 
from $Q$ to $D$ has length $l+1$. By Lemma \ref{LemmaCycl} and the above dimension counting, this is equivalent to 
$(t^{-1}\mathfrak{L}_1)/\mathfrak{L}$ being a cyclic $k[t]$-module. This is equivalent to $\mathfrak{L}_1/t\mathfrak{L}$ being  a cyclic $k[t]$-module.

From this follows (ii): By the cyclicity requirement we must have $a_1=0$ or $a_2 = 0$. Assume $a_1 = 0$. 
Then $t^{-1}v_1 \in \mathfrak{L} - \mathfrak{L}_1$ and its image modulo $t^{l+1}\mathfrak{L}_0$ lies in $L - N_1$.

To prove the converse, assume that $\mathfrak{L}_1/t\mathfrak{L}$ is not cyclic. We must 
have $a_1, a_2 \neq 0$ (otherwise $v_1$ or $v_2$  would induce a cyclic generator of $\mathfrak{L}_1/t\mathfrak{L}$). 
Thus $\langle tv_1, tv_2\rangle_{k[[t]]} \supseteq t\mathfrak{L}$. This implies $\mathfrak{L}_1 = \langle v_1, v_2~\rangle_{k[[t]]} \supseteq \mathfrak{L}$, 
which is equivalent to $N_1 \supseteq L$. \end{proof}

Thus, by Definition, the set $\mathfrak{F}_{P,V,l}(k)$ is exactly the subset of the set on the right hand side in 
(\ref{FPVl_1}), consisting of all chains $(L \subsetneq L^{\prime}) \in S(k)$ with $L \subsetneq N_1$:
\begin{equation}\label{beschrFD} \noindent \mathfrak{F}_{P,V,l}(k) = \left\{ (L \subsetneq L^{\prime}) \in S(k) \colon N/L, N/L^{\prime} \text{ are cyclic as } k[t]\text{-modules}, 
L \not\subseteq N_1 \right\}.\end{equation}

\noindent \textbf{Third step.} We have the following 

\begin{Lm}\label{LemmaNeq} $$\mathfrak{F}_{P,V,l}(k) = \left\{ L\subsetneq L^{\prime} \subsetneq N \colon \begin{array}{cl}\exists v = e_2 + a_0e_1 + a_1te_1 + ... + a_{l}t^{l}e_1, 
 \cr L = \langle v \rangle_{k[t]}, L^{\prime} = L \oplus \langle t^le_1\rangle_k, a_0, ..., a_l \in k \end{array}\right\}.$$ 
The vector $v$ is uniquely determined by the chain $L \subsetneq L^{\prime}$.
\end{Lm}

\begin{proof}[Proof of Lemma]
To prove the inclusion `$\subseteq$' let $L\subsetneq L^{\prime} \subsetneq N$ be an element of $\mathfrak{F}_{P,V,l}(k)$ as described in 
(\ref{beschrFD}). In particular 
$L \not\subseteq N_1$. Then, there exists a $v = e_2 + w \in L$ with $w \in N_1$. Since $L$ is $t$-invariant, we can by successive replacing of $v$ 
by $v - c_it^iv$ with a suitable 
$c_i \in k$ for $1 \leq i \leq l$ assume that $v$ is of the following form: $v = e_2 + a_0e_1 + a_1te_1 + ... + a_{l}t^{l}e_1$. This vector $v$ 
is unique in $L$: this follows from $\dimm_k L = l+1$ and linear independence of $v, tv, ..., t^lv$. Further $L^{\prime}$ is uniquely determined by $L$: 
$L^{\prime}/L$ is a one-dimensional $t$-invariant submodule of $N/L = \langle \bar{e_1}, ...,	 t^l\bar{e_1} \rangle_k$, so 
$L^{\prime} = L \oplus \langle t^{l}e_1 \rangle_k$.

The inverse inclusion is easy: $v \in L - N_1$ and $e_1$ generates the $k[t]$-modules $N/L$ and $N/L^{\prime}$.\end{proof}

\noindent \textbf{Fourth step.} Finally, we prove the existence of a locally closed subscheme $\mathfrak{F}$ of $S$, isomorphic to $\mathbb{A}^{l+1}_k$ such that 
$\mathfrak{F}(\bar{k}) = \mathfrak{F}_{P,V,l}(\bar{k})$.

The partial flag manifold $\mathfrak{Flag}_{l+1,l+2}(N)$ has (\cite{Bo} 10.3) the open subscheme $Y$, defined over $k$, whose $\bar{k}$-valued points are:

$$Y(\bar{k}) = \left\{ E_0 \subsetneq E_1 \subsetneq N \colon \begin{array}{cl}\dimm_k E_0 = l+1, \dimm_k E_1 = l+2, 
E_0 \cap \langle e_1 \rangle_{\bar{k}[t]} = 0, \cr E_1 \cap \langle e_1, te_1, ..., t^{l-1}e_1 \rangle_{\bar{k}} = 0 \end{array} \right\}.$$

\noindent As a $k$-scheme, it is isomorphic to $\mathbb{A}^{(l+1)^2+l}_k$ with the affine coordinate ring 

\noindent $k[a_{ij}, b_{p} \colon 0 \leq i,j < l+1, 0 \leq p < l]$, where a parametrization is given by 

$E_0 = \langle t^i e_2 + \sum_j a_{ij}t^{j}e_1 \colon i = 0, ..., l \rangle_k$, 

$E_1 = E_0 \oplus \langle t^{l}e_1 + b_0e_1 + ... + b_{l-1}t^{l-1}e_1 \rangle_k$.

\noindent All the arguments in the steps one, two and three will also work if we replace $k$ by $\bar{k}$. 
Therefore from Lemma  \ref{LemmaNeq} immediately follows that $\mathfrak{F}_{P,V,l}(\bar{k}) \subset Y(\bar{k}) $.
Moreover, the closed subscheme $\mathfrak{F}$ of $Y$ defined by the ideal 
$$( a_{ij} - a_{0,j-i}, a_{i^{\prime}j^{\prime}}, b_p \colon 0 \leq i \leq j < l+1, 0 \leq j^{\prime} < i^{\prime} < l+1, 0 \leq p < l),$$ 
\noindent of $k[a_{i,j}, b_{p} \colon 0 \leq i,j < l+1, 0 \leq p < l]$ satisfies $\mathfrak{F}(\bar{k}) = \mathfrak{F}_{P,V,l}(\bar{k})$
and hence $\mathfrak{F}_{P,V,l}(\bar{k})$ is a locally closed subset of $X^{SL_2}(\bar{k})$. 
Further, we have: $\mathfrak{F} \cong \mathbb{A}^{l+1}_k$. The uniqueness of $\mathfrak{F}$ follows from the fact that a reduced $k$-subscheme 
 of $X^{SL_2}$ is uniquely determined by its $\bar{k}$-valued points.

In the following, we denote the scheme $\mathfrak{F}$ by $\mathfrak{F}_{P,V,l}$. Thus we have:
$$\mathfrak{F}_{P,V,l} = \Specc k[a_{00}, ..., a_{0l}] \cong \mathbb{A}^{l+1}_k.$$

\noindent \textbf{Fifth step.} Now let $V$ be arbitrary. We can assume that $C^0_M \in V$. 
Then $\mathfrak{F}_{P, V, l}(\bar{k}) \subseteq \mathfrak{F}_{P, C^0_M, l}(\bar{k})$. Consider the morphism of $k$-schemes 
$$\beta: \mathfrak{F}_{P, C^0_M, l} = \Specc k[a_{00}, ..., a_{0l}] \rightarrow \mathfrak{F}_{P, C^0_M, 0} = \Specc k[a^{\prime}_{00}]$$ 
which is defined by using the coordinate rings from the last step: 
$$\beta^0:k[a^{\prime}_{00}] \rightarrow k[a_{00}, ..., a_{0l}]\text{, } a^{\prime}_{00} \mapsto a_{00}.$$

\noindent One sees easily that this morphism sends an alcove $D \in \mathfrak{F}_{P,C^0_M,l}(\bar{k})$ to the first alcove of the minimal gallery, 
stretched from $P$ to $D$. Now an easy computation shows that there is an (unique) open subscheme of $\mathfrak{F}_{P,C^0_M,0}$, isomorphic  
to $\mathbb{P}^1_k - V$ whose $\bar{k}$-rational points are $\mathfrak{F}_{P,V,0}(\bar{k})$.  

We clearly have $$\mathfrak{F}_{P, V, l}(\bar{k}) = \beta^{-1} (\mathfrak{F}_{P,V,0} (\bar{k})),$$ 
(as sets) and thus $\mathfrak{F}_{P, V, l}(\bar{k})$ is a locally closed subset of $X^{SL_2}(\bar{k})$ and there is a unique open subscheme  of 
$\mathfrak{F}_{P,C^0_M, l}$, whose $\bar{k}$-valued points are $\mathfrak{F}_{P, V, l}(\bar{k})$. 
We denote this subscheme by $\mathfrak{F}_{P, V, l}$. It fits into the following Cartesian diagram:
\begin{equation} \xymatrix{ \mathfrak{F}_{P, V, l} \ar[r] \ar[d] & \mathfrak{F}_{P, V, 0} \ar[d] \\ \mathfrak{F}_{P,C^0_M,l} \ar[r]& \mathfrak{F}_{P,C^0_M,0}, } \end{equation} 
where the lower map is $\beta$. Now $\mathfrak{F}_{P,C^0_M,0} \cong \mathbb{A}_k^1$, $\mathfrak{F}_{P, V, 0} \cong \mathbb{P}^1_k - V$, $\mathfrak{F}_{P,C^0_M,l}  \cong \mathbb{A}_k^{l+1}$, 
and since $\beta$ is just the projection on the last factor and $\mathfrak{F}_{P, V, 0} \hookrightarrow \mathfrak{F}_{P,C^0_M,0}$ is the inclusion
$\mathbb{P}^1_k - V \hookrightarrow \mathbb{P}^1_k - \{pt\} = \mathbb{A}^1_k$, we have:
\[ \mathfrak{F}_{P, V, l} \cong \mathbb{A}_k^{l} \times (\mathbb{P}^1_k - V).\qedhere \] 
\end{proof}

\noindent \textbf{Comparison of the two proofs.} We go back to the assumption that $P$ is represented by $\mathfrak{L}_0 = \mathfrak{o} \oplus \mathfrak{o}, 
V = \{C^0_M\}$. In the first four steps we essentially proved that the open Schubert cell associated to $v \in W_a$ with $vC^0_M = C^{-(l+1)}_M$ 
is isomorphic to $\mathbb{A}^{l+1}_k$. There exists also an other proof of this fact which works similarly as in the finite case 
(for the finite case compare \cite{Bo}, 14.12, the first theorem). This proof uses affine root subgroups of $SL_2$. 
They are given by ($n \in \mathbb{Z}$):  

$$U_{(\alpha,n)} = \{\begin{pmatrix} 1 & ct^n\\ 0 & 1\end{pmatrix} \colon c \in \bar{k}\}, \quad
U_{(-\alpha,n)} = \{\begin{pmatrix} 1 & 0\\ ct^n & 1\end{pmatrix} \colon c \in \bar{k}\},$$
where $\alpha$ denotes the unique positive (finite) root of $SL_2$. We have:
$$(\beta, n) \text{ is positive if and only if } \begin{cases} \beta = \alpha &\text{ and } n \geq 0, \text{ or }\\ 
\beta = -\alpha &\text{ and } n > 0. \end{cases}$$
Thus $(\beta,n)$ is positive if and only if $U_{(\beta,n)} \subset I^{SL_2}$. The affine Weyl group acts on the set of all affine roots.

Consider for every $v \in W_a$ the morphism of varieties (the left hand side carries a natural variety structure):
$$\psi \colon \prod_{(\beta, n) > 0, v^{-1}.(\beta, n) < 0}U_{(\beta,n)} \rightarrow I^{SL_2}vI^{SL_2}/I^{SL_2}, (x_n)_n \mapsto (\prod_{n}x_n) vI^{SL_2}.$$
In the other proof one shows that $\psi$ is an isomorphism.

If $l+1 = 2s > 0$ is even (the other cases are essentially the same) and $v = \begin{pmatrix}t^s & 0 \\ 0 & t^{-s}\end{pmatrix}$ then 
$$\mathfrak{F}_{P, C^0_M, l}(\bar{k}) = I^{SL_2}vI^{SL_2}/I^{SL_2}$$  

\noindent and we have the following morphisms of varieties:
\begin{equation}\label{compo} \prod_{0 \leq n < l+1}U_{(\alpha,n)} \stackrel{\psi}{\longrightarrow} I^{SL_2}vI^{SL_2}/I^{SL_2} = \mathfrak{F}_{P,C^0_M, l}(\bar{k}) 
\longrightarrow \Specc \bar{k}[a_{00},...,a_{0l}],\end{equation} 
where the variety on the right hand side is the closed subvariety of $Y$ constructed in our proof. We also  constructed the map on the right and proved 
it to be an isomorphism. 

The image of a $\bar{k}$-valued point $(c_n)_{n=0}^l$ on the left,  under the composition of the two morphisms, is
 a chain of subspaces of $N \otimes \bar{k}$ (where $N = \mathfrak{L}_0/t^{l+1}\mathfrak{L}_0$) which is determined by the alcove
 
\noindent $\psi((c_i)_{n=0}^l) = \begin{pmatrix}t^s & t^{-s}c_0 + ... + t^{s-1} c_l \\ 0 & t^{-s}\end{pmatrix}C^0_M$. With notations as in the proof, 
choosing representatives inside $\mathfrak{L}_0$ of appropriate codimensions and dividing out $t^{l+1}\mathfrak{L}_0$, gives the representing chain 
$$\langle e_2 + c_0e_1 +c_1te_1 + ... + c_lt^le_1\rangle_{\bar{k}[t]} \subsetneq \langle e_2 + c_0e_1 +c_1te_1 + ... + c_lt^le_1\rangle_{\bar{k}[t]} 
\oplus \langle t^le_1\rangle_{\bar{k}}$$
inside $N \otimes \bar{k}$.

The vector $v = e_2 + c_0e_1 +c_1te_1 + ... + c_lt^le_1$ is exactly the same as in Lemma \ref{LemmaNeq}. In the fourth step $\mathfrak{F}_{P,V,l}$ 
was parametrized by the coordinates of this unique vector $v$. Thus the composition of the two morphisms in (\ref{compo}) is given 
on $\bar{k}$-valued points by $(c_n)_{n = 0}^l \mapsto (c_n)_{n=0}^l$. 

\begin{Cor} \label{lmdp} \mbox{}
\begin{itemize}
\item[(i)] Let $P$ be a vertex in $\ratbuild, n \geq 0$ and $V = \mathbb{P}^1_k(k)$ be the set of all alcoves in $\ratbuild$, having $P$ 
as a vertex, considered as a finite variety over $k$. Then there is a locally closed reduced subscheme of $X^{SL_2}$, defined over $k$, whose 
$\bar{k}$-valued points are exactly $D^n_{\ratbuild}(P) = \mathfrak{F}_{P, V, n-1}(\bar{k})$. We denote this scheme again by $D^n_{\ratbuild}(P)$. There is 
an isomorphism of $k$-schemes: $$D^n_{\ratbuild}(P) \cong \mathbb{A}_k^{n-1} \times (\mathbb{P}^1_k - \mathbb{P}^1_k(k)).$$

\item[(ii)] Let $P$ be a vertex in  $A_M, n \geq 0$ and $V$ be the set of all alcoves in $A_M$, having $P$ 
as a vertex (thus $V$ has two elements). Then there is a locally closed reduced subscheme of $X^{SL_2}$, defined over $k$, whose 
$\bar{k}$-valued points are exactly $D^n_{A_M}(P) = \mathfrak{F}_{P, V, n-1}(\bar{k})$. We denote this scheme again by $D^n_{A_M}(P)$. There is 
an isomorphism of $k$-schemes: $$D^n_{A_M}(P) \cong \mathbb{A}_k^{n-1} \times (\mathbb{P}^1_k - \{ 0, \infty\}).$$

\item[(iii)] Let $m \in \{0,1\}$ and $P_m$ be the vertex represented by $\mathfrak{o} \oplus t^m\mathfrak{o}$. Then there is a locally closed reduced subscheme 
of $X^{SL_2}$, defined over $k$, whose $\bar{k}$-valued points are exactly $D^n_{\overline{C_M^0}}(P_m) = \mathfrak{F}_{P_m, C^0_M, n-1}(\bar{k})$. We denote 
this scheme by $D^n_{C_M^0}(P_m)$. There is an isomorphism of $k$-schemes: $$D^n_{C_M^0}(P_m) \cong \mathbb{A}_k^n.$$ 
\end{itemize}
\end{Cor}

\begin{proof} It follows directly from the Proposition.\end{proof}

In particular, the varieties $D^n_{\ratbuild}(P), D^n_{A_M}(P), D^n_{C_M^0}(P_i)$ are all smooth and irreducible.

\subsection{Some preliminaries before stating the results}\label{SomeNOtts} \mbox{}

Recall that $H_b = \Kerr(v_L \circ \dett : J_b \longrightarrow \mathbb{Z})$ and $J^{SL_2}_b = J_b \cap SL_2(L)$. We introduce the following notation:

\begin{Nott}\label{DefK_m} Let $b \in GL_2(L)$. Let $m \in \{ 0, 1\}$, and let $P_m$ be the vertex of $\build$ represented by 
$\mathfrak{o} \oplus t^m\mathfrak{o}$. Set 
\begin{eqnarray*} K^{SL_2,(m)}_b &=& \Stabb\nolimits_{J_b^{SL_2}} (P_m) \quad \text{ and } \\
 K^{(m)}_b &=& \Stabb\nolimits_{H_b} (P_m), \end{eqnarray*}
where $\Stabb\nolimits_{J_b^{SL_2}} (P_m)$ resp. $\Stabb\nolimits_{H_b} (P_m)$ denotes the stabilizer of the vertex $P_m$ under the action of $J^{SL_2}_b$ 
resp. $H_b$. \end{Nott}

For the rest of the work we also use the following notation: for $m \in \{0,1\}$ we set 
$$g_m = \begin{pmatrix} 1 & 0 \\ 0 & t^m\end{pmatrix}$$ (i.e. $g_0 = 1$). Easy computations give:
\begin{Rem} \mbox{}
\begin{itemize}
\item[(i)] Let $b = 1$. Then for $m \in \{0,1\}$:
\begin{eqnarray*} K^{SL_2,(m)}_1 &=& g_m SL_2(\mathfrak{o}_F) g_m^{-1} \quad \text{ and } \\
 K^{(m)}_1 &=& g_m GL_2(\mathfrak{o}_F) g_m^{-1}. \end{eqnarray*}
\item[(ii)] Let $b = \begin{pmatrix} 1 & 0 \\ 0 & t^{\alpha} \end{pmatrix}$ with $\alpha > 0$. Then for $m \in \{0,1\}$: 

\begin{eqnarray*} K^{SL_2,(m)}_b &=& T^{SL_2}(\mathfrak{o}_F) = \{ \begin{pmatrix} a & 0 \\ 0 & a^{-1}\end{pmatrix} \colon a \in \mathfrak{o}_F^{\times} \} \quad 
\text{ and } \\
 K^{(m)}_b &=& T(\mathfrak{o}_F) = \{ \begin{pmatrix} a & 0 \\ 0 & b \end{pmatrix} \colon a, b \in \mathfrak{o}_F^{\times} \}. \end{eqnarray*}
\end{itemize}
\end{Rem}

For every $b \in GL_2(L)$ and $m \in \{0,1\}$, we have:
\begin{eqnarray*} K^{SL_2,(m)}_b \subseteq J^{SL_2}_b &\subseteq& H_b \subseteq J_b, \quad \text{ and } \\ 
K^{(m)}_b &\subseteq& H_b.\end{eqnarray*}

\begin{Lm}\label{transit} \mbox{}
\begin{itemize}
\item[(i)] Let $b=1$. The group $H_1$ resp. $J^{SL_2}_1$ acts transitively on $\ratbuild^{(m)}$, and the stabilizer of $P_m$ 
is the subgroup $K^{(m)}_1$ resp. $K^{SL_2,(m)}_1$.

\item[(ii)] Let $b=\begin{pmatrix} 1 & 0 \\ 0 & t^{\alpha}\end{pmatrix}$. The group $H_b$ resp. $J^{SL_2}_b$ acts transitively on $A_M^{(m)}$, 
and the stabilizer of $P_m$ is the subgroup $K^{(m)}_b$ resp. $K^{SL_2,(m)}_b$.

\end{itemize}
\end{Lm}

\begin{proof}
That $K^{(m)}_b$ resp. $K^{SL_2,(m)}_b$ is the stabilizer of $P_m$ under these actions, follows directly from Definition \ref{DefK_m}. 
The only thing to show is the transitivity. Since $H_b \supseteq J^{SL_2}_b$, it is enough to show this for $J^{SL_2}_b$.

In the first case ($b=1$) we have $J^{SL_2}_1 = SL_2(F)$. Let $P \in \ratbuild^{(m)}$ be represented by a $\mathfrak{o}_F$-lattice $\mathfrak{L}$ in $F^2$ 
with $\wedge^2\mathfrak{L} = t^{m}\mathfrak{o}_F$ and let $\{v_1,v_2\}$ be an $\mathfrak{o}_F$-basis of it such that the determinant of 
the matrix, having $v_1$, $v_2$ as columns, is $t^m$. If now $\{e_1, e_2\}$ is the standard basis of $F^2$, then $P_m$ is represented by the lattice 
with basis $\{e_1,t^m e_2\}$. Now the matrix $x$ sending $\{e_1,t^m e_2\}$ in $\{v_1,v_2\}$ has determinant $1$. We have $xP_m = P$ and thus $J^{SL_2}_1$ 
acts transitively on $\ratbuild^{(m)}$.
In the second case the proof is a similar straightforward computation. 
\end{proof}

\begin{Nott}\label{Imdef} For $m \in \{0,1\}$ we set: 
\begin{equation}\nonumber I^{(m)} = \Stabb\nolimits_{H_1}(C^m_M). \end{equation}
\end{Nott}

Thus, $$I^{(m)} = g_m (I \cap GL_2(F)) g_m^{-1}$$ (in particular, $I^{(0)} = I \cap GL_2(F)$) or more explicitly:
$$I^{(0)} = \begin{pmatrix} \mathfrak{o}^{\times}_F & \mathfrak{o}_F \\ \mathfrak{p}_F & \mathfrak{o}^{\times}_F \end{pmatrix}, \qquad
I^{(1)} = \begin{pmatrix} \mathfrak{o}^{\times}_F & \mathfrak{p}_F^{-1} \\ \mathfrak{p}_F^2 & \mathfrak{o}^{\times}_F \end{pmatrix}. $$

\subsection{The first case: $X^{SL_2}_w(1)$ and $X_w(1)$}  \mbox{}

We have $J_1^{SL_2} = SL_2(F)$ and $J_1 = GL_2(F)$. 

\begin{Prop}\label{Propb1Var} 
\mbox{}
\begin{itemize}
\item[(i)] We have: $$X^{SL_2}_1(1) \cong \coprod_{J_1^{SL_2} / (I^{SL_2} \cap SL_2(F))}\{pt\} \quad \text{ and } \quad 
X_1(1) \cong \coprod_{J_1 / I^{(0)}} \{pt\}$$
as $k$-varieties, and $J^{SL_2}_1$ resp. $J_1$ acts on the set of these connected components by left multiplication on the index set.

\item[(ii)] Let now $1 \neq w \in W_a$ such that $X_w(1) \neq \emptyset$ and $m \in \{0,1\}$ as in Proposition \ref{propb1}(ii). Then
\begin{eqnarray}\nonumber X_w^{SL_2}(1) &\cong& \coprod_{J_1^{SL_2}/K^{SL_2,(m)}_1} \mathbb{A}_k^{\frac{\ell(w)-1}{2}} \times (\mathbb{P}^1_k - \mathbb{P}^1_k(k)) \quad \text{ and }\\
\nonumber X_w(1) &\cong& \coprod_{J_1/K^{(m)}_1} \mathbb{A}_k^{\frac{\ell(w)-1}{2}} \times (\mathbb{P}^1_k - \mathbb{P}^1_k(k)) 
\end{eqnarray}
\noindent as $k$-varieties, and $J^{SL_2}_1$ resp. $J_1$ acts on the set of these connected components by left multiplication on the index set.
\end{itemize}
\end{Prop}

\begin{proof} We prove (i). From Proposition \ref{propb1}(i) follows that the variety $X^{SL_2}_1(1)$ is a disjoint union of points, 
$J^{SL_2}_1$ acts transitively on the set of these points, and the stabilizer of one of them (we take $C^0_M$) is $I^{SL_2} \cap SL_2(F)$. 
Thus the assertion about $X^{SL_2}_1(1)$ follows. For $X_1(1)$ we have to replace $J^{SL_2}_1$ by $H_1$. The stabilizer of $C^0_M$ inside $H_1$ is 
$I^{(0)}$. Thus $X^{SL_2}_1(1) = \coprod_{H_1/I^{(0)}} \{pt\}$.  The assertion about $X_1(1)$ follows now from Proposition \ref{LmGLSL}:
$$X_1(1) = \coprod_{J_1/H_1} \coprod_{H_1 / I^{(0)}} \{pt\} = \coprod_{J_1/ I^{(0)}}\{pt\},$$ 
\noindent where the $J_1$-action on the set of the connected components is given by left multiplication on the index set.

To prove (ii), let $w \neq 1$ such that $X^{SL_2}_w(1) \neq \emptyset$ (see \ref{propb1}(ii)) and let $m \in \{0,1\}$ be as in Proposition \ref{propb1}(ii). Let $i = \frac{\ell(w)+1}{2}$. 
For every $P \in \ratbuild^{(m)}$ the variety $D_{\ratbuild}^{i}(P)$ is connected (Corollary \ref{lmdp}). $J^{SL_2}_1$ acts on the set of the connected components 
of $X^{SL_2}_w(1) = \coprod_{P \in \ratbuild^{(m)}} D_{\ratbuild}^{i}(P)$ by translating them: thus the action is given by permuting
the vertices of departure. They lie in $\ratbuild^{(m)}$ and $J_1^{SL_2}$ acts transitively on them (Lemma \ref{transit}). 
The vertex of departure of the connected component $D_{\ratbuild}^{i}(P_m)$ of $X^{SL_2}_w(1)$ is $P_m$, and its stabilizer is $K^{SL_2,(m)}_1$. 
The assertion about $X^{SL_2}_w(1)$ follows thus from Proposition \ref{propb1}(ii) and Corollary \ref{lmdp}.

For $H_1$ instead of $J^{SL_2}_1$ we have completely analogously: 
$$X^{SL_2}_w(1) = \coprod_{H_1/K^{(m)}_1} \mathbb{A}_k^{\frac{\ell(w)-1}{2}} \times (\mathbb{P}^1_k - \mathbb{P}^1_k(k)), $$
\noindent and thus from Proposition \ref{LmGLSL} follows
$$X_w(1) = \coprod_{J_1/H_1} X^{SL_2}_w(1) = \coprod_{J_1/H_1} \coprod_{H_1/K^{(m)}_1} \mathbb{A}_k^{\frac{\ell(w)-1}{2}} 
\times (\mathbb{P}^1_k - \mathbb{P}^1_k(k)) = \coprod_{J_1/K^{(m)}_1} \mathbb{A}_k^{\frac{\ell(w)-1}{2}} \times (\mathbb{P}^1_k - \mathbb{P}^1_k(k)),$$

\noindent where $J_1$ acts on the set of the connected components by left multiplication on the index set.
\end{proof}


\subsection{The second case: $X^{SL_2}_w(b)$ and $X_w(b)$ for diagonal $b \neq 1$}  \mbox{}

Let  $b = \begin{pmatrix} 1 & 0 \\ 0 & t^{\alpha} \end{pmatrix}$ with $\alpha > 0$. Then $J_b = T(F)$, $J_b^{SL_2} = T^{SL_2}(F)$. Recall that we have 
$K^{(m)}_b = T(\mathfrak{o}_F)$ and $K^{SL_2,(m)}_b = T^{SL_2}(\mathfrak{o}_F)$.

\begin{Prop}\label{propbdvar} 
\mbox{}
\begin{itemize}
\item[(i)] Let $w \in W_a$ with $\ell(w) = \alpha$. We have: 
$$X^{SL_2}_w(b) \cong \coprod_{J_b^{SL_2} / K^{SL_2,(0)}_b}\{pt\} \quad \text{ and } \quad X_w(b) \cong \coprod_{J_b / K^{(0)}_b} \{pt\}$$
as $k$-varieties, and $J^{SL_2}_b$ resp. $J_b$ acts on the set of these connected components by left multiplication on the index set.

\item[(ii)] Let $w \in W_a$ such that $\ell(w) > \alpha$ and $X_w(1) \neq \emptyset$. Let $m \in \{0,1\}$ as in Proposition \ref{propbd}(iii). Then

\begin{eqnarray}\nonumber X^{SL_2}_w(b) &\cong& \coprod_{J_b^{SL_2}/K^{SL_2,(m)}_b} \mathbb{A}_k^{\frac{\ell(w)-\alpha-1}{2}} \times (\mathbb{P}^1_k - \{0, \infty\}) \text{ and } \\
\nonumber X_w(b) &\cong& \coprod_{J_b/K^{(m)}_b} \mathbb{A}_k^{\frac{\ell(w)-\alpha-1}{2}} \times (\mathbb{P}^1_k - \{0, \infty\}) 
\end{eqnarray}
\noindent as $k$-varieties, and $J^{SL_2}_b$ resp. $J_b$ acts on the set of these connected components by left multiplication on the index set.
\end{itemize}
\end{Prop}

\begin{proof} 

To prove (i), notice that from Proposition \ref{propbd}, the variety $X^{SL_2}_w(b)$ is a disjoint union of points if $\ell(w) = \alpha$.
If $wC^0_M = C^{\alpha}_M$ resp. $wC^0_M = C^{-\alpha}_M$, the group $J^{SL_2}_1$ acts transitively on the set of these points, $C^0_M$ resp. $C^1_M$ is 
one of them and the stabilizer of $C^0_M$ resp. $C^1_M$ is $K^{SL_2,(0)}_b(F)$. Thus the assertion about $X^{SL_2}_w(b)$ follows. The assertion 
about $X_w(b)$ follows analogously (as in Proposition \ref{Propb1Var} (i)) from Proposition \ref{LmGLSL}, since $K^{(0)}_b$ is the stabilizer of $C^0_M$ resp. $C^1_M$ under the action of $H_b$.

The proof of (ii) is the same as the proof of Proposition \ref{Propb1Var} (ii) (one has to replace $\ratbuild^{(m)}$ by $A_M^{(m)}$ and use Proposition 
\ref{propbd} instead of Proposition \ref{propb1}).
\end{proof}

\begin{Rem}
{ \rm Recall that $J_b/K^{(m)}_b =  T(F) / T(\mathfrak{o}_F) \cong \mathbb{Z}^2$. Thus if $X_w(b)$ is non-empty, then $\pi_0(X_w(b)) = \mathbb{Z}^2$. }
\end{Rem}

\subsection{The third case: $X^{SL_2}_w(b_1)$ and $X_w(b_1)$}  \mbox{}

Let $E$ denote the unramified quadratic extension of $F$ contained in $L$. Then (compare Example (iii) in \ref{ADLVDef}):

\begin{eqnarray*} J_{b_1} &=& \left\{ \begin{pmatrix} a & \sigma(c) \\ tc & \sigma(a) \end{pmatrix} \colon a, c \in E, a\sigma(a) - c\sigma(c)t \in F^{\times}. \right\} 
\end{eqnarray*}  

\noindent We understand the affine space of dimension zero to be a point. Recall from Proposition \ref{propbn} that $X_w(b_1)$ is non-empty if and only 
if $\ell(w)$ is even.

\begin{Prop}\label{PropbnVar} Let $w \in W_a$ such that $X_w(b_1) \neq \emptyset$. Then
$$X_w^{SL_2}(b_1) \cong \mathbb{A}_k^{\frac{\ell(w)}{2}}$$
as $k$-varieties. Further: 

$$X_w(b_1) \cong \coprod_{J_{b_1}/H_{b_1}} \mathbb{A}_k^{\frac{\ell(w)}{2}}$$ 
as $k$-varieties, and $J_{b_1}$ acts on the set of these connected components by left multiplication on the index set.
\end{Prop}

\begin{proof} The claim about $X^{SL_2}_w(b_1)$ follows directly from Proposition \ref{propbn} and Corollary \ref{lmdp}. The claim about $X_w(b_1)$ 
follows now from Proposition \ref{LmGLSL}.
\end{proof}

\begin{Rem}
\mbox{} 
\begin{itemize}
\item[(i)] {\rm Recall that $J_{b_1} / H_{b_1} \cong \mathbb{Z}$. Thus if $X_w(b_1)$ is non-empty, then $\pi_0(X_w(b_1)) = \mathbb{Z}$.}

\item[(ii)] { \rm The next Lemma says that $H_{b_1} = K^{(m)}_{b_1}$. Thus we can write $X_w(b_1)$ more similarly to the two previous cases. 
In fact, let $w \in W_a$ such that $X_w(b_1)$ is non-empty and let $m \in \{0,1\}$ be as in Proposition \ref{propbn}. Then: }

$$X_w(b_1) \cong \coprod_{J_{b_1}/K^{(m)}_{b_1}} \mathbb{A}_k^{\frac{\ell(w)}{2}}.$$ 
\end{itemize}
\end{Rem}

\begin{Lm}\label{JKLm} We have $H_{b_1} = K^{(m)}_{b_1}$.\end{Lm} 
   
\begin{proof} 
By Definition we have $K^{(m)}_{b_1} \subseteq H_{b_1}$, hence it is enough to prove that $H_{b_1} \subseteq K^{(m)}_{b_1}$. 
This is equivalent to saying that $H_{b_1}$ stabilizes $P_0$ and $P_1$. Thus we have to prove that $H_{b_1} \subseteq I$.
Let $E$ be the unramified quadratic extension of $F$, contained in $L$, and assume 
$x = \begin{pmatrix}a & \sigma(c) \\ tc & \sigma(a) \end{pmatrix} \in H_{b_1}$ with $a,c \in E$ (compare Example (iii) in \ref{ADLVDef}).

Since $x \in H_{b_1}$, we have $0= v_L(\dett(x)) = v_L(a\sigma(a) - tc\sigma(c))$. But $v_L(a\sigma(a))$ is even 
and $v_L(tc\sigma(c))$ is odd. Thus we have $0 = \minn\{v_L(a\sigma(a)), v_L(tc\sigma(c))\}$. From this follows $a \in \mathfrak{o}^{\times}$ and 
$c \in \mathfrak{o}$. Hence $tc \in \mathfrak{p}$ and $\sigma(c) \in \mathfrak{o}$. From all these follows $x \in I$.  
\end{proof}


\section{The cohomology of $X_w(b)$}

Extend the scalars to $\bar{k}$. Let 
\[ \overline{X}_w(b) = X_w(b) \times_{\Specc k} \Specc \bar{k}.\]
The group $J_b$ acts on $X_w(b)$ and thus on $\overline{X}_w(b)$ and this action induces an action on the cohomology groups with 
compact support $H^{\ast}_c(\overline{X}_w(b), \cool)$ where $l \neq p$ is a prime. In this section we will compute 
the cohomology groups of $\overline{X}_w(b)$ and the induced representations of $J_b$ on them. All representations we consider are $\cool$-vector spaces. 
The group $J_b$ is in every case locally profinite and the representations we get are smooth (in the sense of \cite{BH} 2.1).

Beside $J_b$, also the Galois group $$\Gamma = \Gall$$ of the extension $\bar{k}/k$ acts on the cohomology $H_c^{\ast}(\overline{X}_w(b), \cool)$. 
If $X$ is a representation of a group $G$ and $n \in \mathbb{Z}$, then we write $X(n)$ for the twisted representation of $G \times \Gamma$, which is 
isomorphic to $X$ as $G$-representation, and on which the topological generator $\sigma$ of $\Gamma$ acts by multiplication with $q^n$. 
By abuse of notation we write sometimes $X$ instead of $X(0)$.

If $G$ is a locally profinite group, then the (compact) induction of smooth representations from an open subgroup 
$K$ of $G$ (see sect. \ref{AHelpfulResult} below or \cite{BH} 2.4-5 for Definitions) commutes by construction with this twisting: 
if $X$ is a smooth representation of $K$, then 
$$\cIndd\nolimits_K^G X(n) = (\cIndd\nolimits_K^G X)(n) \quad \text{ and } \quad \Indd\nolimits_K^G X(n) = (\Indd\nolimits_K^G X)(n).$$
\noindent We will identify these terms, and only use the left side.

To simplify notation we write in the future $\mathbb{P}^n$ and $\mathbb{A}^n$ for $\mathbb{P}^n_{\bar{k}}$ and $\mathbb{A}^n_{\bar{k}}$. Further, 
we will write $\mathbb{P}^1 - \mathbb{P}^1(k)$ for the scheme $(\mathbb{P}^1_k - \mathbb{P}^1_k(k)) \times_{\Specc k} \Specc \bar{k}$, which is 
the projective line over $\bar{k}$ with the $q+1$ points defined already over $k$ deleted.

\subsection{Three Lemmas}\label{AHelpfulResult}  \mbox{}

Here we will prove three easy facts which we use later on.

\begin{Lm}\label{reptriv} 
Let $X$ be a connected smooth projective curve over $\bar{k}$, and let $G$ be a group acting on $X$. Then the induced action of $G$ on $H^r_c(X, \cool)$ for $r = 0,2$ 
is trivial.
\end{Lm}

\begin{proof} 
This follows from the fact, that the induced action of $G$ on the constant sheaf is trivial and from Poincar\'e duality. 
\end{proof}

Now we will prove two group-theoretic Lemmata. If $G \rightarrow H$ is a homomorphism, then $\inff_H^G$ denotes 
the inflation of representations from $H$ to $G$. If $\pi$ is a representation of a group $G$, and $K \subseteq G$ is a subgroup, 
then we also write $\pi$ for the restriction of $\pi$ to $K$ if it causes no ambiguity. 

The induced representation $\Indd_K^G \pi$ of a smooth representation $(\pi, V)$ of a closed subgroup $K$ of a locally profinite group $G$ is 
the set of functions $f \colon G \rightarrow V$, satisfying the conditions
\begin{itemize}
\item[(1)] $f(kg) = \pi(k)f(g)$ for all $g \in G$; 

\item[(2)] there exists a compact open subgroup $P$ of $G$ such that $f(gp) = f(g)$ for all $g \in G, p \in P$,
\end{itemize}

\noindent and on which $G$ operates by $(gf)(x) := f(xg)$. The compactly induced representation is the subspace $\cIndd_K^G \pi$ of $\Indd_K^G \pi$ which 
consists of functions $f \colon G \rightarrow V$ with the additional property
\begin{itemize}
\item[(3)] the image of the support of $f$ in the set of right cosets of $G$ modulo $K$ is compact.
\end{itemize}

\begin{Lm}\label{infINd} Let $\alpha: G \rightarrow K$ be a continuous, surjective, open homomorphism of locally profinite groups, $N$ an open subgroup 
of $K$ and $R = \alpha^{-1}(N)$. Let $(\pi, V)$ be a smooth representation of $N$. Then we have an isomorphism of smooth $G$-representations:
$$\inff\nolimits_K^G \Indd\nolimits_N^K \pi \cong \Indd\nolimits_R^G \inff\nolimits_N^R \pi.$$
\end{Lm}

\begin{proof} Induction from open subgroups and  inflation respect smoothness. Consider the two maps:
\begin{eqnarray}\nonumber \inff\nolimits_K^G \Indd\nolimits_N^K \pi &\rightarrow& \Indd\nolimits_R^G \inff\nolimits_N^R \pi, \quad f \mapsto (\tilde{f}: G \rightarrow V), \tilde{f}(g) = f(\alpha (g))  \quad \text{ and } \\
\nonumber \Indd\nolimits_R^G \inff\nolimits_N^R \pi &\rightarrow& \inff\nolimits_K^G \Indd\nolimits_N^K \pi, \quad \tilde{f} \mapsto (f: K \rightarrow V), f(k) = \tilde{f}(g), \text{ for some } g \in \alpha^{-1}(k).
\end{eqnarray}

\noindent It is straightforward to show now that these maps are well-defined (i.e., the functions in the image satisfy the conditions (1) and (2) 
and the second map is independent of the choice of the preimage of $k$), inverse to each other and both $G$-equivariant.
\end{proof}

\begin{Lm}\label{cIndIso} 
Let $H$ be a locally profinite group, and $K$, $N$ two open subgroups such that $H = K \cdot N$ where $K \cdot N = \{kn \colon k \in K, n \in N\}$. 
Let $(\pi,V)$ be a smooth $K$-representation. Then there is the following isomorphism of smooth $N$-representations:
$$\Indd\nolimits_K^H \pi \cong \Indd\nolimits^N_{K \cap N} \pi.$$
\end{Lm}

\begin{proof} Induction from an open subgroup respects smoothness. Consider the two maps:
\begin{eqnarray}\nonumber \Indd\nolimits_K^H \pi &\rightarrow& \Indd\nolimits_{K \cap N}^N \pi, \quad \tilde{f} \mapsto \tilde{f}|_N, \quad \text{ and } \\
\nonumber \Indd\nolimits_{K \cap N}^N \pi &\rightarrow& \Indd\nolimits_K^H \pi, \quad f \mapsto \tilde{f}, \text{ where } \tilde{f}(kn) = \pi(k)f(n), 
\text{ for all } k \in K, n \in N. 
\end{eqnarray}
\noindent It is straightforward to show now that these maps are well-defined (i.e., the functions in the image satisfy the conditions (1) and (2), 
and the second map is independent of the choice of $k$ and $n$), inverse to each other and both $N$-equivariant.
\end{proof}


\subsection{The first case: cohomology of $\overline{X}_w(1)$}  \mbox{}

Let $b=1$. We will use the following notation: $$G := GL_2(F) = J_1.$$
Recall that for $m \in \{0,1\}$, we denoted by $P_m$ the vertex in $\build$, represented by the lattice chain $\mathfrak{o} \oplus t^m\mathfrak{o}$, 
by $g_m$ the matrix $\begin{pmatrix} 1 & 0 \\ 0 & t^m \end{pmatrix},$ and finally we had the stabilizer of $P_m$ under the action of $H_1$: 
$$K^{(m)}_1 = g_m GL_2(\mathfrak{o}_F) g_m^{-1}.$$
 
For any group $N$, let $1_{\cool}$ denote the trivial representation of $N$ on the one-dimensional $\cool$-vector space. 
Let $\overline{B}$ denote the subgroup of upper triangular matrices of $GL_2(k)$. Then the Steinberg representation $\Stt_{GL_2(k)}$ of $GL_2(k)$ is 
defined by the following exact sequence: 

\begin{equation}\nonumber \xymatrix{0 \ar[r] & 1_{\cool} \ar[r] & \Indd\nolimits_{\overline{B}}^{GL_2(k)} 1_{\cool} \ar[r] & \Stt_{GL_2(k)} \ar[r] & 0}
\end{equation}

\noindent where the image of the first map is the set of the constant functions. This sequence is split and we 
have $$1_{\cool} \oplus \Stt\nolimits_{GL_2(k)} \cong \Indd\nolimits_{\overline{B}}^{GL_2(k)} 1_{\cool}.$$ The representation $\Stt_{GL_2(k)}$ is irreducible. 
For all these facts about $\Stt_{GL_2(k)}$, we refer to  \cite{BH} 6.3.

For $m \in \{0,1\}$ the map 
$$\pi_m: K^{(m)}_1 \longrightarrow GL_2(k), x \mapsto g_m^{-1} x g_m \modd \left( 1 + \begin{pmatrix} \mathfrak{p}_F &
 \mathfrak{p}_F \\ \mathfrak{p}_F & \mathfrak{p}_F \end{pmatrix} \right) $$
\noindent is continuous, surjective (for $m=0$ it is just the projection on $GL_2(k)$ induced from the projection $\mathfrak{o}_F \twoheadrightarrow
 \mathfrak{o}_F/\mathfrak{p}_F = k$ and for $m=1$ first conjugation by $g_1^{-1}$ and then this projection) and open ($GL_2(k)$ is finite). 
Recall that $I^{(m)} = \Stabb_{H_1} (C^m_M)$ (Notation \ref{Imdef}). We have $I^{(m)} = \pi_m^{-1}(\overline{B})$.  For $m \in \{0,1\}$ set: 

$$ \overline{\Stt} := \inff\nolimits_{GL_2(k)}^{K^{(m)}_1} \Stt\nolimits_{GL_2(k)},$$

\noindent where we inflate with respect to $\pi_m$. Since $\Stt_{GL_2(k)}$ is irreducible, $\overline{\Stt}$ is as well. As an inflation 
of a representation from a finite group, $\overline{\Stt}$ clearly is smooth. From Lemma \ref{infINd}, we have the exact sequence for $m \in \{0,1\}$:
\begin{equation}\label{eqnIndVtilde} \xymatrix{0 \ar[r] & 1_{\cool} \ar[r] & \Indd\nolimits_{I^{(m)}}^{K^{(m)}_1} 1_{\cool}
 \ar[r] & \overline{\Stt} \ar[r] & 0}, \end{equation}

\noindent where the image of the map on the left side is the set of the constant functions. Notice that the index of $I^{(m)}$ in $K^{(m)}_1$ is 
$q+1$, thus $\overline{\Stt}$ is $q$-dimensional, and that $1_{\cool} \oplus \overline{\Stt} = \Indd\nolimits_{I^{(m)}}^{K^{(m)}_1} 1_{\cool}$.


\begin{Prop}\label{hehe} Let $w \in W_a$ such that $X_w(1) \neq \emptyset$ (Proposition \ref{propb1}) and if $w\neq 1$ let $m$ be as in Proposition \ref{propb1}. 
There are the following isomorphisms of $G \times \Gamma$-modules:
\begin{itemize}
\item[(i)] if $w=1$, $$H^r_c(\overline{X}_1(1), \cool) \cong \begin{cases} \cIndd^G_{I^{(0)}} 1_{\cool}& \text{ if } r = 0, \\ 
0 & \text{ else.  } \end{cases}$$

\item[(ii)] If $w \neq 1$, 
$$H^r_c(\overline{X}_w(1), \cool) \cong 
\begin{cases} \cIndd\nolimits^G_{K^{(m)}_1} \overline{\Stt}(\frac{\ell(w)-1}{2}) & \text{ if } r = \ell(w), \\ 
\cIndd\nolimits^G_{K^{(m)}_1} 1_{\cool}(\frac{\ell(w)-3}{2}) & \text{ if } r = \ell(w) + 1,\\ 
0 & \text{ else.}
\end{cases}$$
\end{itemize}
\end{Prop}

\begin{proof} Here the cohomology with compact support of a disjoint union of schemes of finite type over $k$ is defined as the direct sum of the cohomology 
with compact support of these schemes. With this definition, the cohomology with compact support commutes with colimits. 
For $w = 1$, we get from Proposition \ref{Propb1Var}:

\begin{equation}\nonumber H^r_c(\overline{X}_1(1), \cool) = H^r_c(\coprod_{G/I^{(0)}} \{ pt \}, \cool) = \cIndd\nolimits^G_{I^{(0)}} (H^r_c(\{ pt \}, \cool)).
\end{equation}

\noindent The cohomology of a point in positive degrees vanishes and for $r = 0$ we have: $H^0_c(\{ pt \}, \cool) = 1_{\cool}$. 
The action of $\Gamma$ on the zero cohomology group with coefficients in a constant sheaf is trivial. Hence (i) follows.

Assume now $w \neq 1$ such that $X_w(1) \neq \emptyset$ and $m \in \{0,1\}$ as in Proposition \ref{propb1}. Then from Proposition \ref{Propb1Var} follows:

\begin{eqnarray*} H^r_c(\overline{X}_w(1), \cool) &=& H^r_c(\coprod_{G/K^{(m)}_1} \mathbb{A}^{\frac{\ell(w)-1}{2}} \times 
(\mathbb{P}^1 - \mathbb{P}^1(k)), \cool) \\
&=& \cIndd\nolimits^G_{K^{(m)}_1} (H^r_c(\mathbb{A}^{\frac{\ell(w)-1}{2}} \times (\mathbb{P}^1 - \mathbb{P}^1(k)), \cool)).
\end{eqnarray*}

\noindent The group $K^{(m)}_1$ acts trivially on the first factor of the K\"{u}nneth-formula of the right hand side, thus the $K^{(m)}_1$-actions 
on both sides in the following equation (applied to $Y = \mathbb{P}^1 - \mathbb{P}^1(k)$) coinside. For any scheme $Y$ over $\bar{k}$, we have: 
\begin{equation}\label{compdual} H^r_c(\mathbb{A}^{n} \times Y, \cool) = H^{r-2n}_c(Y, \cool(n)).\end{equation}

\noindent From this follows
\begin{equation}\label{CindSL} H^r_c(\overline{X}_w(1), \cool) = \cIndd\nolimits_{K^{(m)}_1}^G (H^{r-\ell(w) + 1}_c(\mathbb{P}^1 - \mathbb{P}^1(k), \cool(\frac{\ell(w)-1}{2}))).\end{equation}
\noindent Now it suffices to determine the action of $K^{(m)}_1$ and of $\Gamma$ on the cohomology groups of this variety. Thus we have reduced 
the computation of the cohomology of $\overline{X}_w(1)$ to the cohomology of the Drinfeld upper halfplane $\Omega_{k}^2 = \mathbb{P}^1 - \mathbb{P}^1(k)$, which 
occurs in the theory of finite Deligne-Lusztig varieties: $\Omega_{k}^2$ is isomorphic to the Deligne-Lusztig variety corresponding to $SL_2(k)$ and
 to the unique nontrivial element of the Weyl group of $SL_2(k)$. 

In fact, $\mathbb{P}^1 - \mathbb{P}^1(k)$ is one-dimensional, thus all 
cohomology groups in degrees $r > 2$ vanish. Recall that 
\begin{equation}\label{cohomP1} H^r_c(\mathbb{P}^1, \cool) = \begin{cases} \cool & \text{ if } r = 0, \\ \cool(-1) & \text{ if } r = 2, \\ 0 & \text{else}.
\end{cases}\end{equation}

\noindent Consider now the Mayer-Vietoris long exact cohomology sequence for the cohomology with compact supports arising from the decomposition 
$\mathbb{P}^1 = (\mathbb{P}^1 - \mathbb{P}^1(k)) \cup \mathbb{P}^1(k)$:

\begin{equation}\nonumber \xymatrix{0 \ar[r] & H^0_c(\mathbb{P}^1 - \mathbb{P}^1(k), \cool) \ar[r] & \cool \ar[r] & H^0_c(\mathbb{P}^1(k), \cool) \\
\ar[r] & H^1_c(\mathbb{P}^1 - \mathbb{P}^1(k), \cool) \ar[r] & 0 \ar[r] & 0 \\
\ar[r] & H^2_c(\mathbb{P}^1 - \mathbb{P}^1(k), \cool) \ar[r] &\cool(-1) \ar[r] & 0.}
\end{equation}

\noindent Since $\mathbb{P}^1 - \mathbb{P}^1(k)$ and $\mathbb{P}^1(k)$ are both stable under $K^{(m)}_1$, the $\cool$-vector spaces in the above sequence 
are $K^{(m)}_1$-representations and the morphisms are $K^{(m)}_1$-equivariant. $\mathbb{P}^1(k)$ represents the alcoves of $\ratbuild$ around $P_m$, 
and $K^{(m)}_1$ acts transitively on them. Hence the group $H^0_c(\mathbb{P}^1(k), \cool)$ as $K^{(m)}_1$-representation is the induced representation 
from the trivial representation of the stabilizer of one alcove (we take $C^{(m)}_M$). From Lemma \ref{reptriv} follows: 
$H_c^0(\mathbb{P}^1, \cool) = 1_{\cool}$ as $K_1^{(m)}$-representation. The map on the right in the first line is the inclusion in (\ref{eqnIndVtilde}). 
Thus
\begin{eqnarray*} H^0_c(\mathbb{P}^1 - \mathbb{P}^1(k), \cool) &=& 0, \quad \text{ and} \\ 
 H^1_c(\mathbb{P}^1- \mathbb{P}^1(k), \cool) &=& \overline{\Stt}. \end{eqnarray*}

\noindent as $K^{(m)}_1$-representations. As a representation of $\Gamma$ the latter is trivial. Further, 
$$H^2_c(\mathbb{P}^1 - \mathbb{P}^1(k), \cool) = H_c^2(\mathbb{P}^1, \cool) = 1_{\cool}(-1)$$ 
\noindent is trivial as $K^{(m)}_1$-representation by Lemma \ref{reptriv}. For any scheme $Y$ over $\bar{k}$ and for any $n$ we have: 
\begin{equation}\label{twistegal} H^r_c(Y,\cool(n)) = H^r_c(Y,\cool) \otimes \cool(n).\end{equation}
Hence

\[ H^r_c(\mathbb{P}^1 - \mathbb{P}^1(k), \cool(\frac{\ell(w)-1}{2})) = \begin{cases} \overline{\Stt}(\frac{\ell(w)-1}{2}) & \text{ if } r = 1, \\
1_{\cool}(\frac{\ell(w)-3}{2}) & \text{ if } r = 2, \\ 0 & \text{ else}.
\end{cases}\]
From this and (\ref{CindSL}) follows: 
\begin{eqnarray*} H^{\ell(w)}_c(\overline{X}_w(1),\cool) &=& \cIndd\nolimits_{K^{(m)}_1}^G \overline{\Stt}(\frac{\ell(w)-1}{2}), \\
 H^{\ell(w)+1}_c(\overline{X}_w(1),\cool) &=& \cIndd\nolimits_{K^{(m)}_1}^G 1_{\cool}(\frac{\ell(w)-3}{2}), \end{eqnarray*}
and $H^r_c(\overline{X}_w(1),\cool) = 0$, for $r \neq \ell(w), \ell(w) + 1$.
\end{proof}

In particular the representations  $H^{\ast}_c(\overline{X}_w(1), \cool)$ are smooth. We will further investigate them in subsection \ref{repsb1}.


\subsection{Some representation theory of $GL_2(F)$}  \mbox{}

Here we recall briefly some aspects of the representation theory of the locally profinite group $G := GL_2(F)$. 
For a more detailed discussion we refer to \cite{BH}, paragraph 9. We assume here all representations to be smooth. By a character of a locally profinite group $N$ 
we always mean a continuous homomorphism $N \rightarrow \cool^{\times}$. To give such a character is equivalent to giving a (smooth) one-dimensional 
representation of $N$. For $m \in \{0,1\}$ the subgroups $K^{(m)}_1$ of $G$ are compact and open.

Let $B$ denote the Borel subgroup of upper triangular matrices in $G$. 
Recall further that $T(F)$ is the diagonal torus of $G$ contained in $B$. Thus 
$$B  = \{ \begin{pmatrix} a & b \\ 0 & c \end{pmatrix} \in G \colon a,c \in F^{\times}, b \in F\} \quad \text{ and } \quad 
T(F) = \{ \begin{pmatrix} a & 0 \\ 0 & c \end{pmatrix} \in G \colon a,c \in F^{\times}\}.$$
The map $\begin{pmatrix} a & b \\ 0 & c \end{pmatrix} \mapsto \begin{pmatrix} a & 0 \\ 0 & c \end{pmatrix}$ gives 
a projection from $B$ to $T(F)$. Since $T(F)$ is abelian, any irreducible representation $\chi$ of $T(F)$ is a character. We can inflate 
it to a character of $B$ by using this projection $B \rightarrow T(F)$. This inflated character of $B$ will also be denoted by $\chi$.

If $\phi$ is a character of $F^{\times}$, then set $$\phi_G := \phi \circ \dett.$$
Then $\phi_G$ is a character of $G$. If $\phi$ ranges over all characters of $F^{\times}$, then $\phi_G$ ranges over all characters of $G$ 
(\cite{BH} 9.2 Proposition).

We can also associate to $\phi$ the character $\phi_T$ of $T(F)$, defined by 
$$\phi_T(\begin{pmatrix} a & 0 \\ 0 & c \end{pmatrix}) := \phi(a)\phi(b).$$

\noindent In other words $\phi_T = \phi \otimes \phi$. The (twisted) Steinberg representation $\phi \cdot \Stt_G$ of $G$ is defined by the following exact 
sequence:

\begin{equation}\label{steinbergdefG} \xymatrix{0 \ar[r] & \phi_G \ar[r] & \Indd\nolimits_B^G \phi_T \ar[r] & \Stt\nolimits_G \ar[r] & 0}. \end{equation}

\noindent The class of all irreducible smooth representations of $G$ can be divided in two disjoint subclasses: the principal series 
(or noncuspidal) representations and the cuspidal representations.
Now, following (\cite{BH}, 9.11) we list all isomorphism classes of irreducible noncuspidal representations of $G$:

\begin{itemize} 
\item[(1)] the irreducible induced representations $\Indd^G_B \chi$ where $\chi$ ranges over the characters of $T(F)$;
\item[(2)] the one-dimensional representations $\phi_G$, where $\phi$ ranges over the characters of $F^{\times}$; 
\item[(3)] the special representations $\phi \cdot \Stt_G$, where $\phi$ ranges over the characters of $F^{\times}$.
\end{itemize} 

\begin{Def} A character $\chi$ of $T(F)$ resp. of $F^{\times}$ is called unramified if $\chi$ is trivial on the subgroup 
$T(\mathfrak{o}_F)$ resp. on $\mathfrak{o}_F^{\times}$.
\end{Def}

\begin{Prop}[Frobenius Reciprocity] Let $H$ be a locally profinite group and $K$ an open subgroup. Let $\pi$ be a smooth representation of $H$, and $\chi$ a 
smooth representation of $K$. Then there are the following isomorphisms which are functorial in both variables:
\begin{itemize}
\item[(i)] $\Homm_H(\cIndd_K^H \chi, \pi) = \Homm_K(\chi, \pi).$ 
\item[(ii)] $\Homm_H(\pi, \Indd_K^H \chi) = \Homm_K(\pi, \chi).$
\end{itemize}
\end{Prop}
\begin{proof} For the proof we refer to \cite{BH} 2.4-5.
\end{proof}


\subsection{The representations $H^{\ast}_c(\overline{X}_w(1), \cool)$}\label{repsb1}  \mbox{}

If $V, W$ are two $G \times \Gamma$-representations then the $\cool$-vector space $\Homm_G(V,W)$ is in a natural way a $\Gamma$-representation: 
$(\gamma f)(v) := \gamma f(\gamma^{-1}v)$. We see all representations $\Indd\nolimits^{G}_{B}\chi, \phi_G, \phi \cdot \Stt_G$ of the last section 
as trivial $\Gamma$-representations. Thus the $\cool$-vector spaces \\
$\Homm\nolimits_G(H^{r}_c(\overline{X}_w(1), \cool), \Indd\nolimits^{G}_{B}\chi)$, ... 
are $\Gamma$-representations.

\begin{thm}\label{thmb1} Let $1 \neq w \in W_a$ such that $X_w(1) \neq \emptyset$.
\begin{itemize}
\item[(i)] Let $\chi$ be a character of $T(F)$. Then
$$\Homm\nolimits_G(H^{\ell(w)+1}_c(\overline{X}_w(1), \cool), \Indd\nolimits^{G}_{B}\chi) = \begin{cases} \cool(\frac{3-\ell(w)}{2}) & \text{ if } \chi 
\text{ unramified, }\\ 0 & \text{ else. }\end{cases}$$

\item[(ii)] Let $\phi$ be a character of $F^{\times}$. Then 
$$\Homm\nolimits_G(H^{\ell(w)+1}_c(\overline{X}_w(1), \cool), \phi_G) = \begin{cases} \cool(\frac{3-\ell(w)}{2}) & \text{ if } \phi \text{ unramified, }\\ 0 
& \text{ else. }\end{cases}$$

\item[(iii)] Let $\phi$ be a character of $F^{\times}$. Then $\Homm\nolimits_G(H^{\ell(w)+1}_c(\overline{X}_w(1), \cool), \phi \cdot \Stt_G) = 0.$
\item[(iv)] Let $\pi$ be a cuspidal representation of $G$. Then \\ $\Homm\nolimits_G(H^{\ell(w)+1}_c(\overline{X}_w(1), \cool), \pi) = 0.$
\end{itemize}
\end{thm}

\begin{proof} Fix a $w \in W_a$ such that $X_w(1) \neq \emptyset$. The Frobenius morphism $\sigma \in \Gamma$ acts on 
$H^{\ell(w)+1}_c (X_w(1) ,\cool)$ by multiplication with $q^{\frac{\ell(w)-3}{2}}$ (see Proposition \ref{hehe}), and trivial on all representations standing 
on the right side in $\Homm_G(\cdot,\cdot)$. Thus it acts by multiplication with $q^{\frac{3-\ell(w)}{2}}$ on the Hom-spaces.

Let now $m \in \{0,1\}$ be as in Proposition \ref{propb1}. Then $B$ and $K^{(m)}_1$ are open subgroups of $G$, and they satisfy 
$B \cdot K^{(m)}_1 = G$. Hence from Proposition \ref{hehe} we get with Frobenius reciprocity and Lemma \ref{cIndIso}: 

\begin{eqnarray*} \Homm\nolimits_G(H^{\ell(w)+1}_c(\overline{X}_w(1), \cool), \Indd\nolimits^{G}_{B}\chi) = \Homm\nolimits_G(\cIndd\nolimits^G_{K^{(m)}_1} 1_{\cool}, \Indd\nolimits^G_B\chi) = \\
\nonumber = \Homm\nolimits_{K^{(m)}_1}(1_{\cool}, \Indd\nolimits^{K^{(m)}_1}_{B \cap K^{(m)}_1} \chi) = 
\Homm\nolimits_{B \cap K^{(m)}_1} (1_{\cool}, \chi) = \begin{cases} \cool & \text{ if }  
1_{\cool} = \chi|_{B \cap K^{(m)}_1}, \\ 0 & \text{ else. }\end{cases} \end{eqnarray*}

\noindent But 
\begin{equation}\label{BKm} 
B \cap K^{(m)}_1 = \begin{pmatrix} \mathfrak{o}_F^{\times} & t^{-m}\mathfrak{o}_F \\ 0 & \mathfrak{o}_F^{\times} \end{pmatrix}
\end{equation} 

\noindent is mapped onto $T(\mathfrak{o}_F)$ under the projection from $B$ to $T(F)$, and hence $1_{\cool} = \chi|_{B \cap K^{(m)}_1}$ if and only if $\chi$ 
is unramified. Hence (i) follows. We prove (ii) similarly:

\begin{eqnarray*} \Homm\nolimits_G(H^{\ell(w)+1}_c(\overline{X}_w(1), \cool), \phi_G) = \Homm\nolimits_G(\cIndd\nolimits^G_{K^{(m)}_1} 1_{\cool}, \phi_G) = \\
 = \Homm\nolimits_{K^{(m)}_1} (1_{\cool}, \phi_G) = \begin{cases} \cool & \text{ if }  1_{\cool} = \phi \circ \dett |_{K^{(m)}_1}, \\ 0 & \text{ else. }\end{cases}\end{eqnarray*}

\noindent But $\dett (K^{(m)}_1) = \mathfrak{o}_F^{\times}$. Thus $\phi_G|_{K^{(m)}_1} = 1_{\cool}$ if and only if $\phi$ is unramified.

To prove (iii) and (iv), notice that if $\pi$ is some smooth representation of $G$ with \\
$\Homm_G(\cIndd_{K^{(m)}_1}^{G} 1_{\cool}, \pi) \neq 0$ then $\pi$ has a $K^{(m)}_1$-stable vector:

\begin{equation}\label{stabKvec}\pi^{K^{(m)}_1} = \Homm\nolimits_{K^{(m)}_1}(1_{\cool}, \pi) = \Homm\nolimits_G(\cIndd\nolimits_{K^{(m)}_1}^{G} 1_{\cool}, \pi) \neq 0.\end{equation} 

\noindent To prove (iv), notice that if $\pi neq 0$ contains the trivial character on $I^{(0)} \subseteq K^{(m)}_1$ and is not cuspidal 
by \cite{BH} 14.3 Proposition. To prove (iii), it is enough to show that $(\phi \cdot \Stt_G)^{K^{(m)}_1} = 0$. 
Every smooth $K^{(m)}_1$-representation is semisimple since $K^{(m)}_1$ is compact (\cite{BH} 2.2 Lemma). 
Thus we have from (\ref{steinbergdefG}):
\begin{equation}\label{stsplit} \phi_G|_{K^{(m)}_1} \oplus \phi \cdot \Stt\nolimits_G|_{K^{(m)}_1}  = (\Indd\nolimits_B^G \phi_{T})|_{K^{(m)}_1}.
\end{equation}
Hence: 
$$\Homm\nolimits_{K^{(m)}_1} (1_{\cool}, \phi_G) \oplus \Homm\nolimits_{K^{(m)}_1} (1_{\cool}, \phi \cdot \Stt\nolimits_G) 
= \Homm\nolimits_{K^{(m)}_1} (1_{\cool}, \Indd\nolimits_B^G \phi_T).$$ 

\noindent The proofs of (i) and (ii) give:

\begin{equation}\nonumber \Homm\nolimits_{K^{(m)}_1} (1_{\cool}, \Indd\nolimits_B^G \phi_T) = \Homm\nolimits_{K^{(m)}_1} (1_{\cool}, \phi_G) = 
\begin{cases} \cool & \text{if } \phi \text{ unramified,} \\ 0 & \text{else.} \end{cases} \end{equation}
In both cases ($\phi$ unramified and $\phi$ not unramified), by dimension counting, it follows 
$$(\phi \cdot \Stt\nolimits_G)^{K^{(m)}_1} = \Homm\nolimits_{K^{(m)}_1} (1_{\cool}, \phi \cdot \Stt\nolimits_G) = 0. \qedhere$$
\end{proof}

\begin{thm}\label{thm2b1} Let $1 \neq w \in W_a$ such that $X_w(1) \neq \emptyset$.
\begin{itemize}
\item[(i)] Let $\chi$ be a character of $T(F)$. Then
$$\Homm\nolimits_G(H^{\ell(w)}_c(\overline{X}_w(1), \cool), \Indd\nolimits^{G}_{B}\chi) = \begin{cases} \cool(\frac{1-\ell(w)}{2}) & \text{ if } \chi 
\text{ unramified, }\\ 0 & \text{ else. }\end{cases}$$

\item[(ii)] Let $\phi$ be a character of $F^{\times}$. Then $\Homm\nolimits_G(H^{\ell(w)}_c(\overline{X}_w(1), \cool), \phi_G) = 0$.

\item[(iii)] Let $\phi$ be a character of $F^{\times}$. Then 
$$\Homm\nolimits_G(H^{\ell(w)}_c(\overline{X}_w(1), \cool ), \phi \cdot \Stt\nolimits_G) = \begin{cases} \cool(\frac{1-\ell(w)}{2}) & \text{ if } \phi \text{ unramified, }\\ 0 
& \text{ else. }\end{cases}$$

\item[(iv)] Let $\pi$ be a cuspidal representation of $G$. Then $\Homm\nolimits_G(H^{\ell(w)}_c(\overline{X}_w(1), \cool), \pi) = 0.$
\end{itemize}
\end{thm}

\begin{proof} Fix a $w \in W_a$ such that $X_w(1) \neq \emptyset$. Analogous to the last theorem, $\sigma \in \Gamma$ 
acts on $H^{\ell(w)}_c (X_w(1) ,\cool)$ by multiplication with $q^{\frac{\ell(w)-1}{2}}$, and trivially on all representations occuring on the right hand 
side in $\Homm_G(\cdot,\cdot)$. Thus it acts by multiplication with $q^{\frac{1-\ell(w)}{2}}$ on the Hom-spaces.

Let now $m \in \{0,1\}$ be as in Proposition \ref{propb1}. Set: 

$$\omega^{(m)} := \begin{pmatrix} 0 & t^{-m} \\ t^m & 0 \end{pmatrix} \in G, \qquad U^{(m)} := \omega^{(m)} I^{(m)} (\omega^{(m)})^{-1} \cap I^{(m)}.$$

\noindent Let further $\overline{T} = \begin{pmatrix} k^{\times} & 0 \\ 0 & k^{\times} \end{pmatrix}$ be the diagonal torus in $GL_2(k)$. Then $U^{(m)}$ is exactly the 
preimage of $\overline{T}$ in $K^{(m)}_1$ under $\pi_m$ and we have: 

\begin{equation}\label{Ubest} U^{(0)} = \begin{pmatrix} \mathfrak{o}_F^{\times} & t\mathfrak{o}_F \\ t\mathfrak{o}_F & \mathfrak{o}_F^{\times} \end{pmatrix}, 
\qquad U^{(1)} = \begin{pmatrix} \mathfrak{o}_F^{\times} & \mathfrak{o}_F \\ t^2\mathfrak{o}_F & \mathfrak{o}_F^{\times} \end{pmatrix}.\end{equation}

\begin{Lm} We have the following isomorphism of $\overline{B}$-representations: 
$$\Stt\nolimits_{GL_2(k)}|_{\overline{B}} \cong \Indd\nolimits_{\overline{T}}^{\overline{B}} 1_{\cool}.$$  
\end{Lm} 

\begin{proof}[Proof of Lemma] Let $\omega = \begin{pmatrix} 0 & 1 \\ 1 & 0\end{pmatrix}$. The decomposition of $GL_2(k)$ in double cosets 
modulo $\overline{B}$ is given by $$GL_2(k) = \overline{B} \cup \overline{B}\omega\overline{B},$$ 
compare \cite{BH} 5.2. Further, $\omega \overline{B} \omega^{-1}$ is the subgroup of lower triangular matrices in $GL_2(k)$ and $\overline{B} \cap \omega \overline{B} 
\omega^{-1} = \overline{T}$. Hence the Mackey formula, \cite{We} 4.2 implies:
$$1_{\cool} \oplus \Stt\nolimits_{GL_2(k)}|_{\overline{B}} = (\Indd\nolimits^{GL_2(k)}_{\overline{B}} 1_{\cool})|_{\overline{B}} \cong 1_{\cool}  \oplus
\Indd\nolimits_{\overline{B} \cap \omega \overline{B} \omega^{-1}}^{\overline{B}} 1_{\cool} = 1_{\cool} \oplus \Indd\nolimits_{\overline{T}}^{\overline{B}} 1_{\cool}.$$
\noindent From this follows $\Stt_{GL_2(k)}|_{\overline{B}} \cong \Indd_{\overline{T}}^{\overline{B}} 1_{\cool}$.
\end{proof}

By commutativity of induction and inflation (Lemma \ref{infINd}), we get from the last Lemma: 
\begin{equation}\label{StU} \overline{\Stt}|_{I^{(m)}} \cong \Indd\nolimits_{U^{(m)}}^{I^{(m)}} 1_{\cool}. \end{equation}

\noindent Now we prove (i). The subgroups $B,K^{(m)}_1$ of $G$ are open, and they satisfy $B \cdot K^{(m)}_1 = G$. Hence we get from Proposition \ref{hehe} 
with Frobenius reciprocity and Lemma \ref{cIndIso}: 

\begin{eqnarray*} \Homm\nolimits_G(H^{\ell(w)}_c (\overline{X}_w(1), \cool), \Indd\nolimits^{G}_{B}\chi) = 
\Homm\nolimits_G(\cIndd\nolimits^G_{K^{(m)}_1} \overline{\Stt}, \Indd\nolimits^G_B\chi) = \\
 = \Homm\nolimits_{K^{(m)}_1}(\overline{\Stt}, \Indd\nolimits^{K^{(m)}_1}_{B \cap K^{(m)}_1} \chi) = \Homm\nolimits_{B \cap K^{(m)}_1} 
(\overline{\Stt}, \chi). \end{eqnarray*}

\noindent But $B \cap K^{(m)}_1 = \begin{pmatrix} \mathfrak{o}_F^{\times} & t^{-m}\mathfrak{o}_F \\ 0 & \mathfrak{o}_F^{\times} \end{pmatrix} = 
B \cap I^{(m)}$, and $U^{(m)} \cdot (B \cap I^{(m)}) = I^{(m)}$, hence from (\ref{StU}) follows with Frobenius reciprocity and Lemma \ref{cIndIso}:

\begin{eqnarray*} \Homm\nolimits_{B \cap K^{(m)}_1} (\overline{\Stt}, \chi) &=& \Homm\nolimits_{B \cap I^{(m)}} (\Indd\nolimits_{U^{(m)}}^{I^{(m)}} 
1_{\cool}, \chi) = \Homm\nolimits_{B \cap I^{(m)}} (\Indd\nolimits_{B \cap U^{(m)}}^{B \cap I^{(m)}} 1_{\cool}, \chi) \\
 &=& \Homm\nolimits_{B \cap U^{(m)}} ( 1_{\cool}, \chi) =
\begin{cases} \cool & \text{ if }  \chi \text{ unramified,} \\ 0 & \text{ else, }\end{cases} \end{eqnarray*}
since $B \cap U^{(m)}$ is mapped onto $T(\mathfrak{o}_F)$ under the projection from $B$ to $T(F)$ (compare (\ref{Ubest})).

Now we prove (ii): let $\phi$ be a character of $F^{\times}$. Consider

\begin{equation}\nonumber 
\Homm\nolimits_G(\cIndd\nolimits_{K^{(m)}_1}^{G} \overline{\Stt}, \phi_G) = \Homm\nolimits_{K^{(m)}_1}(\overline{\Stt}, \phi_G). 
\end{equation}

\noindent Now $\overline{\Stt}$ and $\phi_G|_{K^{(m)}_1}$ are both irreducible and $\dimm_{\cool} \overline{\Stt} = q$, $\dimm_{\cool} \phi_G = 1$. 
Thus there are no morphisms between them. Thus (ii) follows.

In particular, $\Homm\nolimits_{K^{(m)}_1}(\overline{\Stt}, \phi_G) = 0$ and thus from (\ref{stsplit})  follows: 
$$\Homm\nolimits_{K^{(m)}_1}(\overline{\Stt}, \phi \cdot \Stt\nolimits_G) = \Homm\nolimits_{K^{(m)}_1}(\overline{\Stt}, \phi_G) \oplus \Homm\nolimits_{K^{(m)}_1}(\overline{\Stt}, \phi \cdot \Stt\nolimits_G) = 
\Homm\nolimits_{K^{(m)}_1}(\overline{\Stt}, \Indd\nolimits_B^G \phi_{T}).$$
From Frobenius reciprocity and the proof of part (i) follows now:
\begin{eqnarray*}
          \Homm\nolimits_G( \cIndd\nolimits_{K^{(m)}_1}^G \overline{\Stt}, \phi \cdot \Stt\nolimits_G) &=& 
          \Homm\nolimits_{K^{(m)}_1}(\overline{\Stt}, \phi \cdot \Stt\nolimits_G) = 
          \Homm\nolimits_{K^{(m)}_1}(\overline{\Stt}, \Indd\nolimits_B^G \phi_T)  \\
        &=& \begin{cases} \cool & \text{ if } \phi \text{ unramified, } \\ 0 & \text{else,} \end{cases}
\end{eqnarray*}
since $\phi_T$ is unramified if and only if $\phi$ is.

Now it remains to prove (iv). 

\begin{Lm} We have: $\Indd_{I^{(m)}}^{K^{(m)}_1}  1_{\cool} = \Indd_{I^{(0)}}^{K^{(m)}_1} 1_{\cool}$.
\end{Lm}
\begin{proof}
We recall from the proof of Proposition \ref{hehe} that $\Indd_{I^{(m)}}^{K_1^{(m)}} 1_{\cool}$ is the $K^{(m)}_1$-representation, which one obtains by 
letting $K^{(m)}_1$ act (transitively) on the alcoves in $\ratbuild$ having $P_m$ as a vertex, and then taking its action on the zero cohomology 
group of $\mathbb{P}^1(k)$: it is the induced representation of the trivial representation of the stabilizer of one of these alcoves. But $C^0_M$ has both 
$P_0$ and $P_1$ as a vertex and thus $\Indd_{I^{(m)}}^{K^{(m)}_1}  1_{\cool} = \Indd_{I^{(0)}}^{K^{(m)}_1} 1_{\cool}$.
\end{proof}

Thus we have: $1_{\cool} \oplus \overline{\Stt} = \Indd_{I^{(m)}}^{K_1^{(m)}} 1_{\cool} = \Indd_{I^{(0)}}^{K^{(m)}_1} 1_{\cool}$. 
Now consider any smooth representation $\pi$ of $G$ with $\Homm\nolimits_G(\cIndd_{K^{(m)}_1}^{G} \overline{\Stt}, \pi) \neq 0$. Then:

\begin{equation}\nonumber 0 \neq \Homm\nolimits_G(\cIndd\nolimits_{K^{(m)}_1}^{G} \overline{\Stt}, \pi) = \Homm\nolimits_{K^{(m)}_1}(\overline{\Stt}, \pi)  
\subseteq \Homm\nolimits_{K^{(m)}_1}(\Indd\nolimits_{I^{(0)}}^{K_1^{(m)}} 1_{\cool}, \pi) = \Homm\nolimits_{I^{(0)}}( 1_{\cool}, \pi). \end{equation}

\noindent Thus $\pi$ would contain the trivial character of $I^{(0)} = I \cap GL_2(F)$. Then $\pi$ is not cuspidal by \cite{BH} 14.3 Proposition.
\end{proof}

The next Corollary shows that $H^0_c(X_1(1), \cool)$ contains no new information.

\begin{Cor} As $G \times \Gamma$-modules we have: 
$$H^0_c(X_1(1), \cool) \cong \cIndd\nolimits_{K^{(0)}_1}^{G} 1_{\cool} \oplus \cIndd\nolimits_{K^{(0)}_1}^{G}\overline{\Stt}.$$
(i.e. the $\Gamma$-action is trivial).
\begin{itemize}
\item[(i)]Let $\chi$ be a character of $T(F)$. Then
$$\Homm\nolimits_G(H^0_c(\overline{X}_1(1), \cool), \Indd\nolimits^{G}_{B}\chi) = \begin{cases} \cool^2 & \text{ if } \chi 
\text{ unramified, }\\ 0 & \text{ else. }\end{cases}$$
\item[(ii)] Let $\phi$ be a character of $F^{\times}$. Then 
$$\Homm\nolimits_G(H^0_c(\overline{X}_1(1), \cool), \phi_G) = \begin{cases} \cool & \text{ if } \phi \text{ unramified, }\\ 0 
& \text{ else, }\end{cases}$$
\item[(iii)] Let $\phi$ be a character of $F^{\times}$. Then 
$$\Homm\nolimits_G(H^0_c(\overline{X}_1(1), \cool), \phi \cdot \Stt\nolimits_G) = \begin{cases} \cool & \text{ if } \phi \text{ unramified, }\\ 0 
& \text{ else. }\end{cases}$$

\item[(iv)] Let $\pi$ is a cuspidal representation of $G$, then $\Homm\nolimits_G(H^0_c(\overline{X}_1(1), \cool), \pi) = 0.$
\end{itemize}
\end{Cor}

\begin{proof} 
The triviality of $\Gamma$-action follows directly from Proposition \ref{hehe}(i). We have $1_{\cool} \oplus \overline{\Stt} = \Indd_{I^{(0)}}^{K^{(0)}_1} 1_{\cool} = 
\cIndd_{I^{(0)}}^{K^{(0)}_1} 1_{\cool}$, 
since $K^{(0)}_1/I^{(0)}$ is finite. Now from Proposition \ref{hehe}(i), from additivity (compare \cite{BH} 2.4-5) and transitivity (see \cite{We} 4.1.15  for the case of finite groups) 
of the compact induction follow
\begin{eqnarray*} H^0_c(X_1(1), \cool) &=& \cIndd\nolimits_{I^{(0)}}^{G} 1_{\cool} \\ &=& \cIndd\nolimits_{K^{(0)}_1}^{G} \cIndd\nolimits_{I^{(0)}}^{K^{(0)}_1} 1_{\cool} \\&=&
\cIndd\nolimits_{K^{(0)}_1}^{G} (1_{\cool} \oplus \overline{\Stt}) = \cIndd\nolimits_{K^{(0)}_1}^{G}1_{\cool} \oplus \cIndd\nolimits_{K^{(0)}_1}^{G} \overline{\Stt}. 
\end{eqnarray*}
as $G$-representations.  Now (i)-(iv) follow from the Theorems \ref{thmb1} and \ref{thm2b1}.
\end{proof}


\subsection{The second case: cohomology of $\overline{X}_w(b)$ with diagonal $b \neq 1$}  \mbox{}

Let $b = \begin{pmatrix}  1 & 0 \\ 0 & t^{\alpha} \end{pmatrix}$ with $\alpha > 0$. Recall that $J_b = T(F)$ and 
$K^{(m)}_b = T(\mathfrak{o}_F)$ for $m \in \{0,1\}$. In particular all these groups are abelian.

\begin{Prop}\label{cohobdiag} 
Let $w \in W_a$ such that $X_w(b) \neq \emptyset$ (compare Proposition \ref{propbd}), and if $\ell(w) > \alpha$, let $m$ be as in Proposition \ref{propbd}. 
There are the following isomorphisms of $J_b \times \Gamma$-modules:
\begin{itemize}
\item[(i)] if $\ell(w) = \alpha$, 
$$H^r_c(\overline{X}_w(b), \cool) = \begin{cases} \cIndd\nolimits^{J_b}_{K^{(0)}_b} 1_{\cool}& 
\text{if } r = 0, \\ 0 & \text{else. } \end{cases}$$

\item[(ii)] If $\ell(w) > \alpha$, 
\begin{equation}\nonumber H^r_c(\overline{X}_w(b), \cool) = \begin{cases} \cIndd\nolimits^{J_b}_{K^{(m)}_b} 1_{\cool}(\frac{\ell(w)-\alpha-1}{2}) & \text{ if } r = \ell(w) - \alpha, \\ 
\cIndd\nolimits^{J_b}_{K^{(m)}_b} 1_{\cool}(\frac{\ell(w)-\alpha-3}{2})& \text{ if } r = \ell(w)-\alpha+1, \\
0 & \text{ else.} \end{cases} \end{equation}
\end{itemize}
\end{Prop}

\begin{proof} We prove (i). If $\ell(w) = \alpha$, then we get from Proposition \ref{propbdvar}:

$$H^r_c(\overline{X}_w(b), \cool) = H^r_c(\coprod_{J_b/K^{(0)}_b} \{pt\}, \cool) = \cIndd\nolimits_{K^{(0)}_b}^{J_b} (H^r_c(\{pt\}, \cool)).$$

\noindent The cohomology of a point in positive degrees vanishes and for $r = 0$ we have: $H^0_c(\{ pt \}, \cool) = 1_{\cool}$, 
as $K^{(m)}_b$-representation. The action of $\Gamma$ on the zero cohomology group with coefficients in $\cool$ is trivial. Hence (i) follows.

Consider now a $w \in W_a$ with $\ell(w) > \alpha$ such that $X_w(b) \neq \emptyset$, and let $m \in \{0,1\}$ be as in Proposition \ref{propbd}. 
Then from Proposition \ref{propbdvar} follows:
\begin{eqnarray}\label{cIndbd} H^r_c(\overline{X}_w(b), \cool) &=& H^r_c(\coprod_{J_b/K^{(m)}_b} \mathbb{A}^{\frac{\ell(w)-\alpha-1}{2}} \times 
(\mathbb{P}^1 - \{0,\infty\}), \cool) \\
\nonumber &=& \cIndd\nolimits^{J_b}_{K^{(m)}_b} (H^r_c(\mathbb{A}^{\frac{\ell(w)-\alpha-1}{2}} \times (\mathbb{P}^1 - \{0, \infty \}), \cool)).
\end{eqnarray}

\noindent Now (\ref{compdual}) implies:
\begin{equation}\label{neucINdbd} H^r_c(\mathbb{A}^{\frac{\ell(w)-\alpha-1}{2}} \times (\mathbb{P}^1 - \{0, \infty \}), \cool) = 
H^{r - \ell(w) + \alpha + 1}_c (\mathbb{P}^1 - \{0, \infty \}, \cool(\frac{\ell(w) - \alpha - 1}{2})), \end{equation}
\noindent Hence it suffices to determine the action of $K^{(m)}_b$ and of $\Gamma$ on the cohomology of $\mathbb{P}^1 - \{0, \infty\}$. 
First of all $\mathbb{P}^1 - \{0, \infty\}$ is one-dimensional, hence all cohomology groups in degrees $r>2$ vanish.

Using (\ref{cohomP1}), we get the Mayer-Vietoris long exact cohomology sequence for the cohomology with compact supports arising from the decomposition 
$\mathbb{P}^1 = (\mathbb{P}^1 - \{0, \infty\}) \cup \{0 ,\infty \}$, where $\{0, \infty\}$ represents the two alcoves of $A_M$ having $P_m$ as a vertex:

\begin{equation}\nonumber \xymatrix{0 \ar[r] & H^0_c(\mathbb{P}^1 - \{0, \infty\}, \cool) \ar[r] & \cool \ar[r] & H^0_c(\{0, \infty\}, \cool) \\
\ar[r] & H^1_c(\mathbb{P}^1 - \{0, \infty\}, \cool) \ar[r] & 0 \ar[r] & 0 \\
\ar[r] & H^2_c(\mathbb{P}^1 - \{0, \infty\}, \cool) \ar[r] &\cool(-1) \ar[r] & 0.}
\end{equation}

\noindent Since $\mathbb{P}^1 - \{0, \infty\}$ and $\{0, \infty\}$ are both stable under $K^{(m)}_b$, the $\cool$-vector spaces in the above sequence 
are $K^{(m)}_b$-representations and the morphisms are $K^{(m)}_b$-equivariant. But $K^{(m)}_b = T(\mathfrak{o}_F)$ stabilizes every alcove in $A_M$, 
and thus acts trivially on $\{0, \infty\}$. Hence the group $H^0_c(\{0, \infty\}, \cool)$ is isomorphic to $1_{\cool} \oplus 1_{\cool}$ 
as $K^{(m)}_b$-representation. The map on the right in the first line is the diagonal embedding $1_{\cool} \hookrightarrow 1_{\cool} \oplus 1_{\cool}$. 
Thus 
\begin{eqnarray*}\nonumber H^0_c(\mathbb{P}^1 - \{0, \infty\}, \cool) &=& 0, \quad \text{ and } \\
H^1_c(\mathbb{P}^1- \{0, \infty\}, \cool) &=& 1_{\cool},\end{eqnarray*}
as $K^{(m)}_b$-representations. The latter is trivial as a $\Gamma$-representation. Further, 
$$H^2_c(\mathbb{P}^1 - \{0, \infty\}, \cool) = H_c^2(\mathbb{P}^1, \cool) = 1_{\cool}(-1)$$
is trivial as $K^{(m)}_b$-representation by Lemma \ref{reptriv}. From (\ref{twistegal}) we have:
$$H^r_c(\mathbb{P}^1 - \{0, \infty\}, \cool(\frac{\ell(w)-\alpha-1}{2})) = \begin{cases} 1_{\cool}(\frac{\ell(w)-\alpha-1}{2}) & \text{ if } r = 1, \\
1_{\cool}(\frac{\ell(w)-\alpha-3}{2}) & \text{ if } r = 2, \\ 0 & \text{ else}.
\end{cases}$$
From this and (\ref{cIndbd}),(\ref{neucINdbd}) follows: 
\begin{eqnarray*} H^{\ell(w)-\alpha}_c(\overline{X}_w(b),\cool) &=& \cIndd\nolimits_{K^{(m)}_b}^{J_b} 1_{\cool}(\frac{\ell(w)-\alpha-1}{2}), \quad \text{ and } \\ 
 H^{\ell(w)-\alpha+1}_c(\overline{X}_w(b),\cool) &=& \cIndd\nolimits_{K^{(m)}_b}^{J_b} 1_{\cool}(\frac{\ell(w)-\alpha-3}{2}), \end{eqnarray*}
and $H^r_c(\overline{X}_w(b),\cool) = 0$ if $r \neq \ell(w)-\alpha, \ell(w)-\alpha + 1$.
\end{proof}
In particular, these representations are smooth, since they are compactly induced from the trivial representation of an open subgroup. 
We have the exact sequence  of abelian groups: 
\begin{equation}\nonumber \xymatrix{0 \ar[r] & T(\mathfrak{o}_F) \ar[r] & T(F) \ar[r] & \mathbb{Z}^2 \ar[r] & 0.} \end{equation}
Let $R\mathbb{Z}^2$ denote the regular representation of $\mathbb{Z}^2$. This means $\{e_{a,b} \colon a,b \in \mathbb{Z} \}$ is a $\cool$-basis of $R\mathbb{Z}^2$, and 
$\mathbb{Z}^2$ operates by translation on these basis vectors: $(c,d).e_{a,b} = e_{a+c,b+d}$.

We assume all occuring representations and characters of $T(F)$ to be smooth.
Further, since $T(F)$ is abelian, every irreducible representation of $T(F)$ is a character.

\begin{thm}\label{carbn} Let $w \in W_a$ such that $X_w(b) \neq \emptyset$. If $r$ is such that $H^r_c (\overline{X}_w(b), \cool) \neq 0$ (i.e. $r = 0$ if $\ell(w) = \alpha$, and 
$r = \ell(w) - \alpha, \ell(w) - \alpha + 1$ if $\ell(w) > \alpha$), then we have:
$$H^r_c (\overline{X}_w(b), \cool) \cong \inff\nolimits_{\mathbb{Z}^2}^{T(F)} R\mathbb{Z}^2,$$
 as  $T(F)$-representations. Let now $\chi$ be a character of $T(F)$.
\begin{itemize}
\item[(i)] If $\ell(w) = \alpha$, then 
$$\Homm\nolimits_{T(F)} (H^0_c (\overline{X}_w(b), \cool), \chi) = \begin{cases} \cool & \text{ if } \chi \text{ unramified, } \\ 0 & \text{ else.} \end{cases}$$
\item[(ii)] If $\ell(w) > \alpha$, then
\begin{eqnarray*}\nonumber \Homm\nolimits_{T(F)} (H^{\ell(w) - \alpha}_c (\overline{X}_w(b), \cool), \chi) &=& \begin{cases} \cool(\frac{\alpha + 1 - \ell(w)}{2}) & \text{ if } \chi \text{ unramified, } \\ 0 & \text{ else.} \end{cases} \\
\nonumber\Homm\nolimits_{T(F)} (H^{\ell(w) - \alpha+1}_c (\overline{X}_w(b), \cool), \chi) &=& \begin{cases} \cool(\frac{\alpha + 3 - \ell(w)}{2}) & \text{ if } \chi \text{ unramified, } \\ 0 & \text{ else.} \end{cases}\end{eqnarray*}
\end{itemize}
\end{thm}

\begin{proof} All statements about the $\Gamma$-action follow from Proposition \ref{cohobdiag} similarly  as in Theorem \ref{thmb1}.
Further, all non-zero cohomology groups are isomorphic to $\cIndd\nolimits^{T(F)}_{T(\mathfrak{o}_F)} 1_{\cool}$ as $T(F)$-representations. 
For any character $\chi$ of $T(F)$, Frobenius reciprocity gives:
\begin{equation}\nonumber \Homm\nolimits_{T(F)} (\cIndd\nolimits^{T(F)}_{T(\mathfrak{o}_F)} 1_{\cool}, \chi) = \Homm\nolimits_{T(\mathfrak{o}_F)} (1_{\cool}, \chi) = 
\begin{cases} \cool & \text{ if } \chi \text{ unramified, } \\ 0 & \text{ else.} \end{cases}
\end{equation}
It remains to show that $\cIndd\nolimits^{T(F)}_{T(\mathfrak{o}_F)} 1_{\cool} \cong \inff\nolimits_{\mathbb{Z}^2}^{T(F)} R\mathbb{Z}^2$. Since $T(F)$ is abelian, 
Lemma \ref{normalind} below shows that $T(\mathfrak{o}_F)$ operates trivially on $\cIndd\nolimits^{T(F)}_{T(\mathfrak{o}_F)} 1_{\cool}$. Hence 
$\cIndd\nolimits^{T(F)}_{T(\mathfrak{o}_F)} 1_{\cool}$ is equal to an inflation of a representation from $T(F)/T(\mathfrak{o}_F) = \mathbb{Z}^2$.
Now the matrices $t^{(a,b)} = \begin{pmatrix} t^a & 0 \\ 0 & t^b\end{pmatrix}$ with $a,b \in \mathbb{Z}$ represent the left cosets in $T(F)/T(\mathfrak{o}_F)$ 
and by \cite{BH} 2.5 Lemma a basis of $\cIndd\nolimits^{T(F)}_{T(\mathfrak{o}_F)} 1_{\cool}$ is given by the functions
$$e_{a,b}: T(F) \rightarrow \cool, \quad e_{a,b}(x) = \begin{cases} 1 & \text{ if } x \in t^{(-a,-b)}T(\mathfrak{o}_F), \\ 0 &  \text{else. } \end{cases}$$

\noindent For any $x \in T(\mathfrak{o}_F)$ and $c,d \in \mathbb{Z}$, the element $t^{(c,d)}x$ operates on $e_{a,b}$ by translation:
$(t^{(c,d)}x.e_{a,b})(y) = e_{a,b}(t^{(c,d)}xy) = e_{a+c,b+d}(y)$ for any $y \in T(F)$. From this the result follows.
\end{proof}

\begin{Lm}\label{normalind} Let $K \subseteq H$ be a normal subgroup of a locally profinite group. Then $\Indd_K^H 1_{\cool}$ is trivial 
as $K$-representation.
\end{Lm}

\begin{proof} Let $k \in K$ and $f: H \rightarrow \cool$ in $\Indd_K^H 1_{\cool}$. Then we have for all $x \in H$: $(kf)(x) = f(xk) = f(k^{\prime}x) = 
k^{\prime}f(x) = f(x)$ for some $k^{\prime} \in K$. 
\end{proof}


\subsection{The third case: cohomology of $\overline{X}_w(b_1)$}  \mbox{}

\begin{Prop}\label{cohobnf}
Let $w \in W_a$ be such that $X_w(b_1) \neq \emptyset$  (compare Proposition \ref{propbn}). 
Then we have the following isomorphisms of $J_{b_1} \times \Gamma$-modules:
$$H^r_c( \overline{X}_w(b_1), \cool) = \begin{cases} \cIndd_{H_{b_1}}^{J_{b_1}} 1_{\cool}(\frac{\ell(w)}{2}) & \text{if } r = \ell(w),\\ 0 & \text{else}.\end{cases}$$
\end{Prop}

\begin{proof}
From Proposition \ref{PropbnVar} we have:

$$H^r_c( \overline{X}_w(b_1), \cool) = H^r_c( \coprod_{J_{b_1}/H_{b_1}} \mathbb{A}^{\frac{\ell(w)}{2}}, \cool) = 
\cIndd\nolimits_{H_{b_1}}^{J_{b_1}} H^r_c(\mathbb{A}^{\frac{\ell(w)}{2}}, \cool).$$

\noindent Now (\ref{compdual}) implies: $$H^r_c(\mathbb{A}^{\frac{\ell(w)}{2}}, \cool) = H^{r-\ell(w)}_c(\{pt\}, \cool(\frac{\ell(w)}{2})) = 
\begin{cases} 1_{\cool}(\frac{\ell(w)}{2}) & \text{ if } r = \ell(w), \\ 0 & \text{ else. } \end{cases}$$

\end{proof}

In particular all these representations are smooth, since they are compactly induced from an open subgroup (or zero).

\subsection{The representations $H^{\ell(w)}_c( \overline{X}_w(b_1), \cool)$}

At first, we recall  briefly some facts about the multiplicative group $J_{b_1}$ of the quaternion algebra over $F$. 
For all facts presented here we refer to \cite{BH} 53, 54. 
Recall that $E$ denotes the unramified extension of $F$ of degree two contained in $L$. Let 
$$D = \left\{  \begin{pmatrix} a & \sigma(c) \\ tc & \sigma(a) \end{pmatrix} \colon a, c \in E \right\}$$
be the corresponding quaternion algebra. Thus we have $J_{b_1} = D^{\times} = D \backslash \{0\}$.
The reduced norm on $D$ is given by the determinant:
$$\Nrd = \dett: D \rightarrow F, \quad \begin{pmatrix} a & \sigma(c) \\ tc & \sigma(a) \end{pmatrix} \mapsto a\sigma(a) - tc\sigma(c).$$
Its restriction to $D^{\times}$ gives a surjective homomorphism $\dett: D^{\times} \rightarrow F^{\times}$.
The map $$v_D : x \mapsto v_D (x) := v_L(\dett(x))$$
defines a discrete valuation on $D$ and we have the corresponding valuation ring and the group of units in it:
$$\mathfrak{O} = \{x \in D \colon v_D(x) \geq 0 \} \quad \text{ and } \quad U_D = \mathfrak{O}^{\times} = \{x \in D \colon v_D(x) = 0 \}.$$
The group $U_D$ is normal, compact and open subgroup of $D^{\times}$. By definition, we have: 
$$H_{b_1} = U_D.$$

If we speak of a representation of $D^{\times}$ or $F^{\times}$, we mean a smooth representation. 
Let $\chi$ be a character of $F^{\times}$. Then $\chi_D := \chi \circ \dett$ is a character of $D^{\times}$. If $\chi$ ranges over all characters of 
$F^{\times}$, then $\chi_D$ ranges over all characters of $D^{\times}$ (\cite{BH} 53.5). The characters $\chi_D$ are exactly the 
one-dimensional representations of $D^{\times}$. Further, if $\pi$ is an irreducible representation of $D^{\times}$, then $\pi$ is finite-dimensional 
(\cite{BH} 54.1).

We have the projection $D^{\times} \twoheadrightarrow D^{\times}/U_D \cong \mathbb{Z}$. Let $R \mathbb{Z}$ denote the regular representation of $\mathbb{Z}$.

\begin{Lm}\label{level0} Let $w \in W_a$ such that $X_w(b_1) \neq \emptyset$. Then
$$H^{\ell(w)}_c (\overline{X}_w(b_1), \cool) \cong \inff\nolimits_{\mathbb{Z}}^{D^{\times}} R \mathbb{Z}$$
 as  $D^{\times}$-representations. 
Further if $\pi$ is an irreducible representation of $D^{\times}$ with \\
$$\Homm\nolimits_{D^{\times}} (H^{\ell(w)}_c (\overline{X}_w(b_1), \cool), \pi) \neq 0,$$ 
then $U_D \subseteq \Kerr(\pi)$. In particular $\pi$ is one-dimensional and unramified.
\end{Lm}

\begin{proof}
Since $U_D$ is normal in $D^{\times}$, Lemma \ref{normalind} shows that $\cIndd\nolimits_{U_D}^{D^{\times}} 1_{\cool}$ is trivial as $U_D$-representation.
A similar computation as in the proof of Theorem \ref{carbn} shows that $H^{\ell(w)}_c (\overline{X}_w(b_1), \cool) \cong \inff\nolimits_{\mathbb{Z}}^{D^{\times}} 
R \mathbb{Z}$.

If $\alpha: \cIndd_{U_D}^{D^{\times}} 1_{\cool} \rightarrow \pi$ is some non-zero homomorphism, then $\alpha$ is surjective, since $\pi$ is irreducible.  
Hence $\pi$ must be trivial on $U_D$. Hence $\pi$ is the inflation of some irreducible representation of $\mathbb{Z}$. Since $\mathbb{Z}$ is abelian, this 
representation must be one-dimensional. Thus $\pi$ is one-dimensional.
\end{proof}
Summarizing the results, we get the following

\begin{thm}\label{posl} Let $w \in W_a$ such that $X_w(b_1) \neq \emptyset$. Then
$$H^{\ell(w)}_c (\overline{X}_w(b_1), \cool) \cong \inff\nolimits_{\mathbb{Z}}^{D^{\times}} R \mathbb{Z}$$
 as  $D^{\times}$-representations.

\begin{itemize}
\item[(i)] Let $\chi$ be a character of $F^{\times}$. Then 

\begin{equation}\nonumber \Homm\nolimits_{D^{\times}} (H^{\ell(w)}_c( \overline{X}_w(b_1), \cool), \chi_D) = 
\begin{cases} \cool(-\frac{\ell(w)}{2}) & \text{ if } \chi \text{ unramified, } \\ 0 & \text{ else.} \end{cases} \end{equation}
\item[(ii)] Let $\pi$ be an  irreducible representation of $D^{\times}$ of dimension $\geq 2$. Then \\
$$\Homm\nolimits_{D^{\times}} (H^{\frac{\ell(w)}{2}}_c( \overline{X}_w(b_1), \cool), \pi) = 0.$$
\end{itemize}
\end{thm}

\begin{proof} The first statement was already proven in Lemma \ref{level0}. 
All statements about the $\Gamma$-action follow from Proposition \ref{cohobnf} as in Theorem \ref{thmb1}. To prove (i), notice that Frobenius reciprocity 
implies:

\begin{eqnarray*}  \Homm\nolimits_{D^{\times}} (H^{\ell(w)}_c( \overline{X}_w(b_1), \cool), \chi_D) &=& 
\Homm\nolimits_{D^{\times}} (\cIndd\nolimits^{D^{\times}}_{U_D} 1_{\cool}, \chi_D) = \Homm\nolimits_{U_D} ( 1_{\cool}, \chi_D) = \\ 
  &=& \begin{cases} \cool & \text{ if } \chi \circ \dett|_{U_D} = 1_{\cool}, \\ 0 & \text{ else.} \end{cases} \end{eqnarray*}

\noindent But since $\dett$ is surjective and $U_D = \dett^{-1}(\mathfrak{o}_F^{\times})$, we have $\dett(U_D) = \mathfrak{o}_F^{\times}$. This implies:
$\chi \circ \det|_{U_D} = 1_{\cool}$ if and only if $\chi$ is unramified.

(ii) follows directly from Lemma \ref{level0}.
\end{proof}


\renewcommand{\refname}{References}

\end{document}